\documentclass[a4paper,12pt,reqno,final]{amsart}
\usepackage{amssymb,amsmath,amsthm}
\usepackage[margin=23mm]{geometry}
\usepackage[]{graphicx}
\usepackage{here}
\usepackage[all]{xy}

\usepackage{rsfso}

\usepackage[dvipdfmx]{color}

\newcommand{\INN}[2]{\langle #1, #2 \rangle}

\theoremstyle{plain}
\newtheorem{thm}{Theorem}[section]
\newtheorem{pro}[thm]{Proposition}
\newtheorem{lem}[thm]{Lemma}
\newtheorem{cor}[thm]{Corollary}

\theoremstyle{definition}
\newtheorem{dfn}[thm]{Definition}
\newtheorem{ex}[thm]{Example}
\newtheorem{rem}[thm]{Remark}
\newtheorem{nota}{Notation}

\numberwithin{equation}{section}

\title[Double Satake diagrams and
canonical forms in compact symmetric triads]{Double Satake diagrams and
canonical forms in compact symmetric triads}

\author[K.~Baba and O.~Ikawa]{Kurando Baba and Osamu Ikawa}

\thanks{The second author 
was partially supported
by JSPS KAKENHI Grant Number 22K03285.
}


\address[K. Baba]{Faculty of Science and Technology, 
Tokyo University of Science, 
Yamazaki, Noda, Chiba, 278-8510, Japan. }
\email{kurando.baba@rs.tus.ac.jp}

\address[O. Ikawa]{Faculty of Arts and Sciences, Kyoto Institute of Technology, 
Matsugasaki, Sakyoku, Kyoto 606-8585, Japan. }
\email{ikawa@kit.ac.jp}

\begin{document}

\maketitle

\begin{abstract}
In this paper,
we first introduce the notion of double Satake diagrams
for compact symmetric triads.
In terms of this notion,
we give an alternative proof for the classification theorem
for compact symmetric triads, which was originally given by 
Toshihiko Matsuki.
Secondly, we introduce the notion of canonical forms
for compact symmetric triads,
and prove the existence of canonical forms
for compact simple symmetric triads.
We also give some properties
for canonical forms.
\end{abstract}

\setcounter{tocdepth}{1}
\tableofcontents

\section{Introduction}

A \textit{compact symmetric triad} is a triple $(G,\theta_{1},\theta_{2})$
which consists of a compact connected semisimple Lie group $G$ and two involutions $\theta_{1}$ and $\theta_{2}$
on it.
The study of compact symmetric triads
is motivated by the geometry of \textit{Hermann actions}.
If we denote by $K_{i}$ the identity component of
the fixed point subgroup of $G$ for $\theta_{i}$, $i=1,2$,
then $G/K_{i}$ is a compact Riemannian symmetric space.
The natural isometric action of $K_{2}$ on $G/K_{1}$ is called the
Hermann action of $K_{2}$ on $G/K_{1}$.
In the case when $\theta_{1}=\theta_{2}$,
we have $K_{1}=K_{2}$,
and the Hermann action
is nothing but the isotropy action on $G/K_{1}$.
It is known that Hermann actions have a geometrically good property,
the so-called hyperpolarity (\cite{HPTT}).
In general, an isometric action of a compact connected Lie group on a Riemannian manifold
is called \textit{hyperpolar},
if there exists a  connected closed flat submanifold that meets all orbits
orthogonally.
Such a submanifold is called a \textit{section} or a \textit{canonical form} of the action.
It is known that any section becomes a totally geodesic submanifold.
The classification of hyperpolar actions on compact Riemannian symmetric spaces
was given by Kollross (\cite{Kollross}).
By his classification
most of hyperpolar actions on compact Riemannian symmetric spaces
are given by Hermann actions.
It is expected that
a further development of the theory for compact symmetric triads
promotes a precise understanding of Hermann actions and their orbits.

In this paper,
we first study the classification theory for compact symmetric triads.
Matsuki (\cite{Matsuki97})
introduced a non-trivial equivalence relation $\sim$
on compact symmetric triads (Definition \ref{dfm:cst_sim}).
Roughly speaking,
if two compact symmetric triads are isomorphic with respect to $\sim$,
then their Hermann actions are essentially the same.
Our concern is to classify
the local isomorphism classes of compact symmetric triads.
For this,
we will generalize the method
to classify compact symmetric pairs
due to Araki (\cite{Araki}).
In fact, he obtained the local isomorphism classes
of compact symmetric pairs in terms  of Satake diagrams.
Then, we introduce the notion of double Satake diagrams
as a generalization of Satake diagrams (Definition \ref{dfn:2satake}).
The equivalence relation $\sim$
induces a natural equivalence relation on double Satake diagrams.
In fact,
the local isomorphism of a compact symmetric triad
determines that of a double Satake diagram,
and this correspondence becomes bijective
(Theorem \ref{thm:cst_dsatake_sim}, Lemma \ref{lem:cstdsatabij}).
By using the results
we obtain the classification
of the local isomorphism classes of compact symmetric triads,
namely,
the classification of double Satake diagrams
(Theorem \ref{thm:dsatake_classify})
derives that of
the local isomorphism classes
of compact symmetric triads
(Corollary \ref{cor:cst_classify}).
Our classification is listed in Table
\ref{table:rank_ord}.
In addition, 
for each isomorphism class of compact symmetric triads,
we can give the method to determine its rank and order
by means of the corresponding double Satake diagram,
which are also given in the same table.
Although the original classification of compact symmetric triads
was given by Matsuki (\cite{Matsuki97}),
such data are advantages of our classification.
Our motivation for giving the alternative proof comes from
the study of not only Hermann action but also
the classification of noncompact symmetric pairs
in terms of the theory for compact symmetric triads (\cite{BIS}).
The results of this paper
plays an important role in the forthcoming paper \cite{BIS2}.

Next we study canonical forms for compact symmetric triads.
Intuitively,
for the isomorphism class $[(G,\theta_{1},\theta_{2})]$,
a canonical form is defined as
a representative of $[(G,\theta_{1},\theta_{2})]$
which has the most easy structure in $[(G,\theta_{1},\theta_{2})]$.
Our precise definition of the canonical form of $[(G,\theta_{1},\theta_{2})]$
is given in Definition \ref{dfn:cst_can}.
We prove the existence of canonical forms
in the case when $G$ is simple (Theorem \ref{thm:cst_exist_can}).
As mentioned above,
if two compact symmetric triads are isomorphic to each other,
then so are their Hermann actions.
Nevertheless, we find that
there is a difference in understandability between their Hermann actions.
Hence it is necessary to choose a canonical form in the isomorphic class.
This is 
significant to study canonical forms of compact symmetric triads.
For example,
in the case when $[(G,\theta_{1},\theta_{2})]$ is commutable,
its canonical form is given by a commutative compact symmetric triad
$(G,\theta_{1}',\theta_{2}')\sim(G,\theta_{1},\theta_{2})$.
The second author (\cite{Ikawa})
developed a systematic method to study
orbits of Hermann actions for such $(G,\theta_{1}',\theta_{2}')$.
By applying his method, many mathematicians contributed to
study commutative Hermann actions
(for example, \cite{Ikawa3}, \cite{ITT}, \cite{O}, \cite{Ohno}, \cite{OSU}).
On the other hand,
the study of Hermann actions for non-commutable case
was found in \cite{GT} and \cite{Ohno21},
and their studies were based on the classification.
For further direction,
we should construct a unified theory for the geometry of Hermann actions whether $(G,\theta_{1},\theta_{2})$ is commutative or not,
and expect that the canonical forms plays a important role
in such a study.

The organization of this paper is as follows:
In Section \ref{sec:cst_fundamental},
we recall the notion of compact symmetric triads.
We define the rank and the order for a compact symmetric triad
and for its isomorphism class.
The hyperpolarity
of a Hermann action is also explained.
In Section \ref{sec:pre},
we recall that 
the local isomorphism classes
of compact symmetric pairs
correspond to
$\sigma$-systems and Satake diagrams.
In Section \ref{sec:2satake},
we first introduce the notion of double $\sigma$-systems.
We also define an equivalence relation on double $\sigma$-systems
based on the equivalence relation $\sim$
(Subsection \ref{sec:2sigroot}).
Next, we introduce the notion of double Satake diagrams
and their isomorphism classes for double $\sigma$-systems
(Subsection \ref{sec:sub2satake}).
In Section \ref{sec:cst_dsatake},
we introduce the notion of double Satake diagrams
for compact symmetric triads.
We prove Theorems \ref{thm:cst_dsatake_sim}
and \ref{thm:dsatake_classify} mentioned above.
We determine the rank and the order for the isomorphism classes of
compact simple symmetric triads
based on the classification.
Furthermore, 
we give special isomorphisms for compact simple symmetric triads
and determine which compact simple symmetric triads are self-dual.
In Section \ref{sec:cst_cf},
we introduce the notion of the canonicality for compact symmetric triads
(Subsection \ref{sec:cst_can}),
and prove its existence (Theorem \ref{thm:cst_exist_can} in Subsection \ref{sec:cst_can_exist}).
We also give some properties
for the rank and the order of a canonical form
(Subsection \ref{sec:rankorder}).

\section{Compact symmetric triads}\label{sec:cst_fundamental}

\subsection{Compact symmetric triads and Hermann actions}

Let $G$ be a compact connected semisimple Lie group,
and $\theta_{1}, \theta_{2}$ be two involutions of $G$.
We call the triplet $(G,\theta_{1},\theta_{2})$
a \textit{compact symmetric triad}.
Denote by $K_{i}$ ($i=1$, $2$)
the identity component of the fixed point subgroup of $\theta_{i}$
in $G$.
Then $G/K_{i}$ is a compact Riemannian symmetric space
with respect to the Riemannian metric induced from
a bi-invariant Riemannian metric on $G$.
The natural isometric action of $K_{2}$
on $G/K_{1}$ is called the \textit{Hermann action}.

In what follows,
we show that the Hermann action is a hyperpolar action.
In particular, we give its section.
We also recall an equivalence relation on compact symmetric triads
which was introduced by Matsuki (\cite{Matsuki}).
Then we observe that
two compact symmetric triads are isomorphic in his sense,
then their Hermann actions are essentially the same.

Let $\mathfrak{g}$ be the Lie algebra of $G$
and $\exp:\mathfrak{g}\to G$ denote the exponential map.
For each $i=1,$ $2$,
the differential $d\theta_{i}$ of $\theta_{i}$ at the identity element in $G$
gives an involution of $\mathfrak{g}$, which we write the same symbol $\theta_{i}$
if there is no confusion.
Let
$\mathfrak{g}=\mathfrak{k}_{1}\oplus\mathfrak{m}_{1}=\mathfrak{k}_{2}\oplus
\mathfrak{m}_{2}$ be
the canonical decompositions of $\mathfrak{g}$
for $\theta_{1}$ and $\theta_{2}$, respectively.
We set $\mathfrak{g}^{\theta_{1}\theta_{2}}=\{X\in\mathfrak{g}\mid
\theta_{1}\theta_{2}(X)=X\}=\{X\in\mathfrak{g}\mid\theta_{1}(X)=\theta_{2}(X)\}=\mathfrak{g}^{\theta_{2}\theta_{1}}$.
Then $\mathfrak{g}^{\theta_{1}\theta_{2}}$
becomes a $(\theta_{1},\theta_{2})$-invariant Lie subalgebra of $\mathfrak{g}$.
Clearly, $\theta_{1}=\theta_{2}$ holds on $\mathfrak{g}^{\theta_{1}\theta_{2}}$.
The canonical decomposition of $\mathfrak{g}^{\theta_{1}\theta_{2}}$
for $\theta_{1}|_{\mathfrak{g}^{\theta_{1}\theta_{2}}}$
is given by
\[
\mathfrak{g}^{\theta_{1}\theta_{2}}=(\mathfrak{k}_{1}\cap\mathfrak{k}_{2})\oplus(\mathfrak{m}_{1}\cap\mathfrak{m}_{2}).
\]
Let $\mathfrak{a}$ be a maximal abelian subspace of 
$\mathfrak{m}_{1}\cap\mathfrak{m}_{2}$.
It is known that $A:=\exp(\mathfrak{a})$
is closed in $G$.
Hence, $A$ becomes a compact connected abelian Lie subgroup of $G$,
that is, a toral subgroup.
The following theorem was proved by Hermann.
\begin{thm}[\cite{Hermann}]\label{thm:Hermann}
Retain the notation as above.
Then,
\[
G=K_{1}AK_{2}
=K_{2}AK_{1}.
\]
\end{thm}

Let $\pi_{1}:G\to G/K_{1}$ denote the natural projection.
Then $\pi_{1}(A)$ is a flat totally geodesic submanifold of $G/K_{1}$.
It follows from Theorem \ref{thm:Hermann}
that each $K_{2}$-orbit intersects $\pi_{1}(A)$.
In fact, it is shown that
$\pi_{1}(A)$ gives a section of the Hermann action $K_{2}$ on $G/K_{1}$.
Hence this action is hyperpolar.

Let $\mathrm{Aut}(G)$
denote the group of automorphisms on $G$
and $\mathrm{Int}(G)$
the group of inner automorphisms on $G$.
Then $\mathrm{Int}(G)$ is a normal subgroup of $\mathrm{Aut}(G)$.
Matsuki (\cite{Matsuki}) introduced the following equivalence relation
on compact symmetric triads.

\begin{dfn}\label{dfm:cst_sim}
Two compact symmetric triads $(G,\theta_{1},\theta_{2})$
and $(G,\theta_{1}',\theta_{2}')$
are \textit{isomorphic}, which we write
$(G,\theta_{1},\theta_{2})\sim(G,\theta_{1}',\theta_{2}')$,
if there exist
$\varphi\in\mathrm{Aut}(G)$
and $\tau\in\mathrm{Int}(G)$
satisfying the following relations:
\begin{equation}\label{eqn:dfn_cstsim}
\theta_{1}'=\varphi\theta_{1}\varphi^{-1},\quad
\theta_{2}'=\tau\varphi\theta_{2}\varphi^{-1}\tau^{-1}.
\end{equation}
\end{dfn}
Geometrically,
$(G,\theta_{1},\theta_{2})\sim(G,\theta_{1}',\theta_{2}')$
means that 
their Hermann actions are isomorphic.
Indeed, we obtain an isomorphism between them as follows.
Assume that there exist $\varphi\in\mathrm{Aut}(G)$
and $\tau\in\mathrm{Int}(G)$ as in \eqref{eqn:dfn_cstsim}.
Denote by $K_{i}'$ the identity component of the fixed point subgroup
of $\theta_{i}'$ in $G$.
We obtain an isometric isomorphism
$\Phi: G/K_{1}\to G/K_{1}'$
by
\[
\Phi: G/K_{1}\to G/K_{1}';~gK_{1}\mapsto \varphi(g)K_{1}'.
\]
Then the $K_{2}$-action on $G/K_{1}$
is isomorphic to the $\tau^{-1}(K_{2}')$-action on $G/K_{1}'$
via $\Phi$.

The Lie subgroups
$K_{1}\cap K_{2}$
and
$G^{\theta_{1}\theta_{2}}:=\{g\in G\mid \theta_{1}\theta_{2}
(g)=g\}$ of $G$ play
a fundamental role in the study of $(G,\theta_{1},\theta_{2})$.
However,
their Lie group structures
and even their Lie algebra structures
depend on the choice of a representative of its isomorphism class $[(G,\theta_{1},\theta_{2})]$. 
We expect that these structures are determined by
taking a representative with `easy structure' 
in
$[(G,\theta_{1},\theta_{2})]$.
We will introduce such a representative
as a canonical form in Section \ref{sec:cst_cf},
which is one of the main subjects of the present paper.

\subsection{Rank and order for compact symmetric triads}

We define the rank of a compact symmetric triad
$(G,\theta_{1},\theta_{2})$
as the dimension of a maximal abelian subspace
$\mathfrak{a}$ of $\mathfrak{m}_{1}\cap\mathfrak{m}_{2}$,
which we write $\mathrm{rank}(G,\theta_{1},\theta_{2})$.
Its well-definedness is shown by
the $\mathrm{Ad}(K_{1}\cap K_{2})$-conjugacy
for maximal abelian subspaces of
$\mathfrak{m}_{1}\cap\mathfrak{m}_{2}$,
where $\mathrm{Ad}$ denotes the adjoint representation of $G$.
Since the tangent space of $\pi_{1}(A)$ gives 
the normal space of 
a principal orbit
of $K_{2}$ on $G/K_{1}$,
the rank is equal
to the cohomogeneity of the action.
Hence,
for two compact symmetric triads
$(G,\theta_{1},\theta_{2})\sim(G,\theta_{1}',\theta_{2}')$,
the cohomogeneity of the $K_{2}$-action on $G/K_{1}$
is equal to that of the $K_{2}'$-action
on $G/K_{2}'$.
Namely, we have the following lemma.

\begin{lem}\label{lem:cst_rank}
Assume that
two
compact symmetric triads
$(G,\theta_{1},\theta_{2})$
and
$(G,\theta_{1}',\theta_{2}')$
satisfies
$(G,\theta_{1},\theta_{2})\sim
(G,\theta_{1}',\theta_{2}')$.
Then we have
$\mathrm{rank}(G,\theta_{1},\theta_{2})=\mathrm{rank}(G,\theta_{1}',\theta_{2}')$.
Hence we define the rank of
the isomorphism class
$[(G,\theta_{1},\theta_{2})]$
of $(G,\theta_{1},\theta_{2})$
as that of $(G,\theta_{1},\theta_{2})$,
which we write $\mathrm{rank}[(G,\theta_{1},\theta_{2})]$.
\end{lem}

Let $(G,\theta_{1},\theta_{2})$
be a compact symmetric triad
and $\mathfrak{a}$ a maximal abelian subspace of $\mathfrak{m}_{1}\cap\mathfrak{m}_{2}$.
Then there exists a maximal abelian subspace of $\mathfrak{m}_{i}$
containing $\mathfrak{a}$ for each $i=1$, $2$.
However,
$[\mathfrak{a}_{1},\mathfrak{a}_{2}]=\{0\}$
does not hold in general.
In the case when $[\mathfrak{a}_{1},\mathfrak{a}_{2}]\neq\{0\}$,
there exist no maximal abelian subalgebras $\mathfrak{t}$ of $\mathfrak{g}$
such that $\mathfrak{t}$ contains both $\mathfrak{a}_{1}$ and $\mathfrak{a}_{2}$.
On the other hand,
retaking $(G,\theta_{1},\theta_{2})$ in its isomorphism class if necessary,
the following lemma holds.

\begin{lem}\label{lem:lemm2.5}
There exists a compact symmetric triad
$(G,\theta_{1},\theta_{2}')\sim(G,\theta_{1},\theta_{2})$
and a maximal abelian subalgebra $\mathfrak{t}$ of $\mathfrak{g}$
satisfying the following conditions$:$
\begin{enumerate}
\item[$(1)$] $\mathfrak{a}_{1}:=\mathfrak{t}\cap\mathfrak{m}_{1}$
and $\mathfrak{a}_{2}':=\mathfrak{t}\cap\mathfrak{m}_{2}'$
are maximal abelian subspaces of $\mathfrak{m}_{1}$
and $\mathfrak{m}_{2}'$, respectively.
In particular, $\mathfrak{t}$ is $(\theta_{1},\theta_{2}')$-invariant.
\item[$(2)$] $\mathfrak{a}:=\mathfrak{t}\cap(\mathfrak{m}_{1}\cap\mathfrak{m}_{2}')$
is a maximal abelian subspace of $\mathfrak{m}_{1}\cap\mathfrak{m}_{2}'$.
\end{enumerate}
\end{lem}

\begin{proof}
Let $\mathfrak{a}$
be a maximal abelian subspace of $\mathfrak{m}_{1}\cap\mathfrak{m}_{2}$.
Let $\mathfrak{a}_{i}$
be a maximal abelian subspace of $\mathfrak{m}_{i}$
containing $\mathfrak{a}$.
We define a closed subgroup $N(\mathfrak{a})$ of $G$ by
$N(\mathfrak{a}):=\{g\in G\mid \mathrm{Ad}(g)\mathfrak{a}=\mathfrak{a}\}$.
Since $G$ is compact, so is $N(\mathfrak{a})$.
Then the identity component $N(\mathfrak{a})_{0}$ of $N(\mathfrak{a})$
becomes a compact connected Lie group.
Furthermore, its Lie algebra $\mathfrak{n}(\mathfrak{a})$
has the following expression:
\[
\mathfrak{n}(\mathfrak{a})=\{X\in\mathfrak{g}\mid [X,\mathfrak{a}]\subset\mathfrak{a}\}.
\]
Since $[\mathfrak{a}_{i},\mathfrak{a}]\subset[\mathfrak{a}_{i},\mathfrak{a}_{i}]=\{0\}$ holds,
we have $\mathfrak{a}_{i}\subset \mathfrak{n}(\mathfrak{a})$.
Hence $\mathfrak{a}_{i}$ is an abelian subalgebra of $\mathfrak{n}(\mathfrak{a})$.
By general theory of compact connected Lie groups,
there exists $g\in N(\mathfrak{a})_{0}$ satisfying $[\mathfrak{a}_{1}, \mathrm{Ad}(g)\mathfrak{a}_{2}]=\{0\}$.
We set $\theta_{2}':=\tau_{g}\theta_{2}\tau_{g}^{-1}$.
Then we have $d\theta_{2}'=\mathrm{Ad}(g)d\theta_{2}\mathrm{Ad}(g)^{-1}$
and $\mathfrak{m}_{2}'=\mathrm{Ad}(g)\mathfrak{m}_{2}$.
From the inclusion $\mathfrak{a}\subset\mathfrak{a}_{2}$,
we get
$\mathfrak{a}=\mathrm{Ad}(g)\mathfrak{a}\subset
\mathrm{Ad}(g)\mathfrak{a}_{2}=:\mathfrak{a}_{2}'\subset \mathfrak{m}_{2}'$.
This yields $\mathfrak{a}\subset \mathfrak{a}_{1}\cap\mathfrak{a}_{2}'$.
In addition, by the maximality of $\mathfrak{a}$,
we obtain $\mathfrak{a}=\mathfrak{a}_{1}\cap\mathfrak{a}_{2}'$.
Since $[\mathfrak{a}_{1},\mathfrak{a}_{2}']=\{0\}$ holds,
there exists a maximal abelian subalgebra $\mathfrak{t}$ of $\mathfrak{g}$
containing $\mathfrak{a}_{1}$ and $\mathfrak{a}_{2}'$.
This $\mathfrak{t}$ satisfies the two conditions as in the statement.
\end{proof}

We denote by $\mathrm{rank}(G)$
the rank of the compact connected semisimple Lie group $G$,
and by $\mathrm{rank}(G,\theta_{i})$
the rank of the compact symmetric pair $(G,\theta_{i})$.
From Lemma \ref{lem:lemm2.5} we have the following corollary.

\begin{cor}
If $\mathrm{rank}(G)=\mathrm{rank}(G,\theta_{1})$,
then $\mathrm{rank}(G,\theta_{2})=\mathrm{rank}(G,\theta_{1},\theta_{2})$ holds.
\end{cor}

\begin{proof}
The statement of this corollary is independent of the choice of
a representative in the isomorphism class $[(G,\theta_{1},\theta_{2})]$.
By Lemma \ref{lem:lemm2.5} we may assume that
there exists a maximal abelian subalgebra $\mathfrak{t}$ of $\mathfrak{g}$
such that 
$\mathfrak{t}\cap\mathfrak{m}_{i}$
($i=1,2$)
is a maximal abelian subspace of $\mathfrak{m}_{i}$,
and that $\mathfrak{t}\cap(\mathfrak{m}_{1}\cap\mathfrak{m}_{2})$
is a maximal abelian subspace of $\mathfrak{m}_{1}\cap\mathfrak{m}_{2}$.
Then
we have $\mathfrak{t}\subset \mathfrak{m}_{1}$ by $\mathrm{rank}(G)=\mathrm{rank}(G,\theta_{1})$.
This implies
$\mathrm{rank}(G,\theta_{1},\theta_{2})
=\dim(\mathfrak{t}\cap\mathfrak{m}_{1}\cap\mathfrak{m}_{2})
=\dim(\mathfrak{t}\cap\mathfrak{m}_{2})=\mathrm{rank}(G,\theta_{2})$.
Thus, we have completed the proof.
\end{proof}

We will define the order for the isomorphism class $[(G,\theta_{1},\theta_{2})]$.
For the representative $(G,\theta_{1},\theta_{2})$,
the order of the composition $\theta_{1}\theta_{2}$,
which we write $\mathrm{ord}(\theta_{1}\theta_{2})$,
is defined by the smallest positive integer $k$ satisfying
$(\theta_{1}\theta_{2})^{k}=1$.
If there is no such $k$,
then $\theta_{1}\theta_{2}$ has infinite order,
which we write $\mathrm{ord}(\theta_{1}\theta_{2})=\infty$.
The value of $\mathrm{ord}(\theta_{1}\theta_{2})$
depends on the choice of a representative of $[(G,\theta_{1},\theta_{2})]$. 
We define the \textit{order} of the isomorphism class $[(G,\theta_{1},\theta_{2})]$
by
\[
\mathrm{ord}[(G,\theta_{1},\theta_{2})]
:=\min\{\mathrm{ord}(\theta_{1}'\theta_{2}')\mid
(G,\theta_{1}',\theta_{2}')\sim(G,\theta_{1},\theta_{2})\}\in
\mathbb{N}\cup\{\infty\}.
\]
It will be shown later that
$[(G,\theta_{1},\theta_{2})]$
has a finite order in the case when $G$ is simple.

Here, we observe
compact symmetric triads
with low order.
For two involutions $\theta_{1}$ and $\theta_{2}$ on $G$,
we write $\theta_{1}\sim\theta_{2}$ if there exists $\tau\in\mathrm{Int}(G)$
satisfying $\theta_{2}=\tau\theta_{1}\tau^{-1}$.
A compact symmetric triad $(G,\theta_{1},\theta_{2})$
satisfying $\theta_{1}\sim\theta_{2}$ is isomorphic to $(G,\theta_{1},\theta_{1})$.
Hence, for a compact symmetric triad $(G,\theta_{1},\theta_{2})$,
the order of $[(G,\theta_{1},\theta_{2})]$
is equal to one if and only if
$\theta_{1}\sim\theta_{2}$ holds.
The Hermann action induced from such $(G,\theta_{1},\theta_{2})$
is nothing but the isotropy action $K_{1}$ on $G/K_{1}$.
In other words,
$(G,\theta_{1},\theta_{2})$ with $\theta_{1}\not\sim\theta_{2}$
gives a nontrivial Hermann action.
The isotropy actions
of compact symmetric spaces
have been studied by many geometers.
Therefore we will 
mainly focus our attention on compact symmetric triads
$(G,\theta_{1},\theta_{2})$ with $\theta_{1}\not\sim\theta_{2}$.
A compact symmetric triad $(G,\theta_{1},\theta_{2})$
is said to be \textit{commutative},
if $\theta_{1}\theta_{2}=\theta_{2}\theta_{1}$ holds.
Clearly,
$\mathrm{ord}[(G,\theta_{1},\theta_{2})] \leq 2$ holds
if and only if $[(G,\theta_{1},\theta_{2})]$
is commutable,
i.e.,
there exists a commutative compact symmetric triad
$(G,\theta_{1}',\theta_{2}')\sim(G,\theta_{1},\theta_{2})$.

The following proposition gives a sufficient condition
that the order of $[(G,\theta_{1},\theta_{2})]$
is equal to one.

\begin{pro}\label{pro:ordn_n+1_sim}
Let $(G,\theta_{1},\theta_{2})$ and $(G,\theta_{1}',\theta_{2}')$
be two compact symmetric triads satisfying $(G,\theta_{1},\theta_{2})\sim(G,\theta_{1}',\theta_{2}')$.
Assume that there exists $n\in\mathbb{N}$ such that
$(\theta_{1}\theta_{2})^{n}=1$ and $(\theta_{1}'\theta_{2}')^{n+1}=1$.
Then we have $\theta_{1}\sim\theta_{2}$.
In particular, $\theta_{1}'\sim\theta_{2}'$ holds.
\end{pro}

\begin{proof}
Without loss of generalities
we may assume that $\theta_{1}'=\theta_{1}$
and $\theta_{2}'=\tau\theta_{2}\tau^{-1}$ for some $\tau\in\mathrm{Int}(G)$.
Then, we have $(\theta_{1}\theta_{2}')^{n+1}
=(\theta_{1}\tau\theta_{2}\tau^{-1})^{n}\theta_{1}(\tau\theta_{2}\tau^{-1})$.
Hence it is sufficient to show that
there exists $\tau_{1}\in\mathrm{Int}(G)$ such that $(\theta_{1}\tau\theta_{2}\tau^{-1})^{n}\theta_{1}=\tau_{1}\theta_{2}\tau_{1}^{-1}$.

Let us consider the case when $n$ is even: $n=2m$ for some $m\in\mathbb{N}$.
Then we have
$(\theta_{1}\tau\theta_{2}\tau^{-1})^{n}\theta_{1}=(\theta_{1}\tau\theta_{2}\tau^{-1})^{m}\theta_{1}(\theta_{1}\tau\theta_{2}\tau^{-1})^{m}$.
Let $g$ be in $G$ satisfying $\tau=\tau_{g}$.
Since $\theta_{i}\tau_{g}=\tau_{\theta_{i}(g)}\theta_{i}$ holds,
there exists $\tau_{1}\in\mathrm{Int}(G)$
satisfying $(\theta_{1}\tau\theta_{2}\tau^{-1})^{m}=\tau_{1}(\theta_{1}\theta_{2})^{m}$.
From $(\theta_{1}\theta_{2})^{2m}=1$, we obtain
\[
(\theta_{1}\tau\theta_{2}\tau^{-1})^{n}\theta_{1}
=\tau_{1}(\theta_{1}\theta_{2})^{m}\theta_{1}(\theta_{2}\theta_{1})^{m}\tau_{1}^{-1}
=\tau_{1}(\theta_{1}\theta_{2})^{2m}\theta_{1}\tau_{1}^{-1}=\tau_{1}\theta_{1}\tau_{1}^{-1}.
\]
In the case when $n$ is odd,
a similar argument shows
that there exists $\tau_{1}\in\mathrm{Int}(G)$ such that $(\theta_{1}\tau\theta_{2}\tau^{-1})^{n}\theta_{1}=\tau_{1}\theta_{2}\tau_{1}^{-1}$.
Thus, we have complete the proof.
\end{proof}

Here, let us consider the case when the rank of $[(G,\theta_{1},\theta_{2})]$
is equal to zero.
Then $K_{2}$ acts transitively on $G/K_{1}$ by Theorem \ref{thm:Hermann}.
Furthermore, 
the value of
the order of $\theta_{1}'\theta_{2}'$
is independent of the choice of a representative 
$(G,\theta_{1}',\theta_{2}')$
in $[(G,\theta_{1},\theta_{2})]$,
namely, the following proposition holds.

\begin{pro}\label{pro:rank0_ordconst}
Assume that the rank of $[(G,\theta_{1},\theta_{2})]$ is equal to zero.
If $(G,\theta_{1},\theta_{2})\sim
(G,\theta_{1}',\theta_{2}')$ then,
$\mathrm{ord}(\theta_{1}\theta_{2})=\mathrm{ord}(\theta_{1}'\theta_{2}')$
holds.
\end{pro}

\begin{proof}
It follows from $(G,\theta_{1},\theta_{2})\sim
(G,\theta_{1}',\theta_{2}')$ that
there exist $\varphi\in\mathrm{Aut}(G)$
and $g\in G$ satisfying the following relation:
\begin{equation}\label{eqn:theta'theta}
\theta_{1}'=\varphi\theta_{1}\varphi^{-1},\quad
\theta_{2}'=\tau_{g}\varphi\theta_{2}\varphi\tau_{g}^{-1},
\end{equation}
where $\tau_{g}$
is an inner automorphism of $G$ defined by $\tau_{g}(h)=ghg^{-1}$ ($h\in G$).
By applying Theorem \ref{thm:Hermann} to $(G,\theta_{1}',\theta_{2}')$
we have $k_{1}\in K_{1}'$ and $k_{2}\in K_{2}'$ satisfying $g=k_{2}'k_{1}'$.
Here, we have used the assumption $\mathrm{rank}[(G,\theta_{1},\theta_{2})]=0$.
Then from \eqref{eqn:theta'theta}
we obtain
\begin{align*}
\theta_{1}'
&=\tau_{k_{1}}\theta_{1}'\tau_{k_{1}}^{-1}
=(\tau_{k_{1}}\varphi)\theta_{1}(\tau_{k_{1}}\varphi)^{-1},\\
\theta_{2}'&=\tau_{k_{2}}^{-1}\theta_{2}'\tau_{k_{2}}
=\tau_{k_{2}}^{-1}(\tau_{k_{2}}\tau_{k_{1}}\varphi\theta_{2}\varphi\tau_{k_{1}}^{-1}\tau_{k_{2}}^{-1})\tau_{k_{2}}
=(\tau_{k_{1}}\varphi)\theta_{2}(\tau_{k_{1}}\varphi)^{-1}.
\end{align*}
This obeys $\mathrm{ord}(\theta_{1}\theta_{2})=\mathrm{ord}(\theta_{1}'\theta_{2}')$.
Thus we have completed the proof.
\end{proof}

\subsection{Compact symmetric triads at the Lie algebra level}

In the present paper,
we will also treat compact symmetric triads at the Lie algebra level.
A compact symmetric triad 
at the Lie algebra level
is a triplet $(\mathfrak{g},\theta_{1},\theta_{2})$
which consists of a compact semisimple Lie algebra $\mathfrak{g}$
and two involutions $\theta_{1}$ and $\theta_{2}$ of $\mathfrak{g}$.
Let $\mathrm{Aut}(\mathfrak{g})$
denote the group of automorphisms on $\mathfrak{g}$
and $\mathrm{Int}(\mathfrak{g})$
the group of inner automorphisms on $\mathfrak{g}$.
Then $\mathrm{Int}(\mathfrak{g})$ is a normal subgroup of $\mathrm{Aut}(\mathfrak{g})$.
Let us define the Lie algebra version of Definition \ref{dfm:cst_sim} as follows.

\begin{dfn}\label{dfn:cst_local_equiv}
Two compact symmetric triads $(\mathfrak{g},\theta_{1},\theta_{2})$
and $(\mathfrak{g},\theta_{1}',\theta_{2}')$
are \textit{isomorphic}, which we write
$(\mathfrak{g},\theta_{1},\theta_{2})\sim(\mathfrak{g},\theta_{1}',\theta_{2}')$,
if there exist
$\varphi\in\mathrm{Aut}(\mathfrak{g})$
and $\tau\in\mathrm{Int}(\mathfrak{g})$
satisfying the following relations:
\[
\theta_{1}'=\varphi\theta_{1}\varphi^{-1},\quad
\theta_{2}'=\tau\varphi\theta_{2}\varphi^{-1}\tau^{-1}.
\]
\end{dfn}

Let us consider a correspondence 
between the Lie group level and the Lie algebra level for compact symmetric triads.
For a compact symmetric triad
$(G,\theta_{1},\theta_{2})$
at the Lie group level,
$(\mathfrak{g},d\theta_{1},d\theta_{2})$
gives a compact symmetric triad at the Lie algebra level.
Then $(\mathfrak{g},d\theta_{1},d\theta_{2})$
is called the
compact symmetric triad at the Lie algebra level
associated with $(G,\theta_{1},\theta_{2})$.
We find that
for two compact symmetric triads
$(G,\theta_{1},\theta_{2})$ and $(G,\theta_{1}',\theta_{2}')$,
$(G,\theta_{1},\theta_{2})\sim(G,\theta_{1}',\theta_{2}')$
implies $(\mathfrak{g},d\theta_{1},d\theta_{2})\sim(\mathfrak{g},d\theta_{1}',d\theta_{2}')$.
We say that two compact symmetric triads
$(G,\theta_{1},\theta_{2})$ and $(G,\theta_{1}',\theta_{2}')$
are \textit{locally isomorphic},
if $(\mathfrak{g},d\theta_{1},d\theta_{2})\sim(\mathfrak{g},d\theta_{1}',d\theta_{2}')$ holds.

Conversely,
for a compact symmetric triad $(\mathfrak{g},\theta_{1},\theta_{2})$ at the Lie algebra level,
there exists a compact symmetric triad $(G, \Theta_{1}, \Theta_{2})$
satisfying $(\mathfrak{g}, d\Theta_{1}, d\Theta_{2})=(\mathfrak{g},\theta_{1},\theta_{2})$,
where $\mathfrak{g}$ is the Lie algebra of $G$.
Indeed, let $G$ denote the universal covering group
of a connected Lie group with Lie algebra $\mathfrak{g}$
or the adjoin group  of $\mathfrak{g}$.
Then we can get $\Theta_{i}$ as the extension of $\theta_{i}$ to an involution of $G$.

Let $(\mathfrak{g},\theta_{1},\theta_{2})$
be a compact symmetric triad at the Lie algebra level.
The rank of $(\mathfrak{g},\theta_{1},\theta_{2})$
is defined as the dimension of
a maximal abelian subspace of $\mathfrak{m}_{1}\cap\mathfrak{m}_{2}$.
We define the rank of $[(\mathfrak{g},\theta_{1},\theta_{2})]$
by that of $(\mathfrak{g},\theta_{1},\theta_{2})$.
In a similar manner,
the orders
of $(\mathfrak{g},\theta_{1},\theta_{2})$
and its isomorphic class $[(\mathfrak{g},\theta_{1},\theta_{2})]$
 are defined in the same way as in the case of the Lie group level.
We denote by $\mathrm{rank}[(\mathfrak{g},\theta_{1},\theta_{2})]$
the rank of $[(\mathfrak{g},\theta_{1},\theta_{2})]$,
and by $\mathrm{ord}[(\mathfrak{g},\theta_{1},\theta_{2})]$
the order of $[(\mathfrak{g},\theta_{1},\theta_{2})]$.

By definition we have the following lemma.
\begin{lem}\label{lem:rank_welldef}
For
any compact symmetric triad 
$(G,\theta_{1},\theta_{2})$,
we have
$\mathrm{rank}[(G,\theta_{1},\theta_{2})]=\mathrm{rank}[(\mathfrak{g},d\theta_{1},d\theta_{2})]$.
\end{lem}
In order to state a similar result for the order,
we prepare the following lemma.

\begin{lem}\label{lem:autoGg}
An automorphism $\theta$ of $G$
is the identity transformation on it
if and only if
so is its differential $d\theta$ on $\mathfrak{g}$.
\end{lem}

We omit the details of the proof.
The following lemma follows immediately from
Lemma \ref{lem:autoGg}.

\begin{lem}
Let $(G,\theta_{1},\theta_{2})$
be a compact symmetric triad.
Then
$\mathrm{ord}(\theta_{1}\theta_{2})=\mathrm{ord}(d\theta_{1}d\theta_{2})$
holds. In particular,
we have $\mathrm{ord}[(G,\theta_{1},\theta_{2})]=\mathrm{ord}[(\mathfrak{g},d\theta_{1},d\theta_{2})]$.
\end{lem}

\begin{lem}\label{lem:t1simt2>dt1dt2=1}
Let $(G,\theta_{1},\theta_{2})$
be a compact symmetric triad.
Assume that
there exists a maximal abelian subalgebra $\mathfrak{t}$ of $\mathfrak{g}$
such that $\mathfrak{t}\cap\mathfrak{m}_{i}$ and
$\mathfrak{t}\cap(\mathfrak{m}_{1}\cap\mathfrak{m}_{2})$
are maximal abelian subspaces of $\mathfrak{m}_{i}$ and $\mathfrak{m}_{1}\cap\mathfrak{m}_{2}$,
respectively.
Then 
$\theta_{1}\sim\theta_{2}$
implies $\mathrm{ord}(d\theta_{1}d\theta_{2}|_{\mathfrak{t}})=1$.
\end{lem}

\begin{proof}
It follows from
$\theta_{1}\sim\theta_{2}$
we have $d\theta_{2}=\mathrm{Ad}(g)d\theta_{1}\mathrm{Ad}(g)^{-1}$ for some $g\in G$.
By Theorem \ref{thm:Hermann}
there exist $k_{i}\in K_{i}$ and $H\in\mathfrak{t}\cap(\mathfrak{m}_{1}\cap\mathfrak{m}_{2})$
such that $g=k_{2}\exp(H)k_{1}$ holds.
Here, we have used the maximality of $\mathfrak{t}\cap(\mathfrak{m}_{1}\cap\mathfrak{m}_{2})$
in $\mathfrak{m}_{1}\cap\mathfrak{m}_{2}$.
Then we have
\begin{align*}
\mathrm{Ad}(g)d\theta_{1}\mathrm{Ad}(g)^{-1}
&=\mathrm{Ad}(k_{2})e^{\mathrm{ad}(H)}\mathrm{Ad}(k_{1})d\theta_{1}\mathrm{Ad}(k_{1})^{-1}e^{-\mathrm{ad}(H)}\mathrm{Ad}(k_{2})^{-1}\\
&=\mathrm{Ad}(k_{2})e^{\mathrm{ad}(H)}d\theta_{1}e^{-\mathrm{ad}(H)}\mathrm{Ad}(k_{2})^{-1},
\end{align*}
from which $d\theta_{2}=e^{\mathrm{ad}(H)}d\theta_{1}e^{-\mathrm{ad}(H)}$ holds.
Since the automorphism $e^{\mathrm{ad}(H)}$ gives the identity transformation on $\mathfrak{t}$,
we obtain
$d\theta_{2}|_{\mathfrak{t}}
=d\theta_{1}|_{\mathfrak{t}}$.
This yields $\mathrm{ord}(d\theta_{1}d\theta_{2}|_{\mathfrak{t}})=1$.
\end{proof}

\section{$\sigma$-systems, Satake diagrams and compact symmetric pairs}\label{sec:pre}

In this section,
we recall the notions
of $\sigma$-systems,
Satake diagrams and compact symmetric pairs.
We refer to the references
\cite{Helgason} and \cite{Warner}, for example.
The contents of this section will be generalized
in Sections
\ref{sec:2satake}
and \ref{sec:cst_dsatake}.

\subsection{Root systems}

We begin with recalling the definition of a root system.
Let $\mathfrak{t}$ be a finite dimensional real vector space.
Fix an inner product $\INN{\,}{\,}$ on $\mathfrak{t}$.
We write $\|\alpha\|=\INN{\alpha}{\alpha}^{1/2}$
as the norm of $\alpha\in\mathfrak{t}$.
For $\alpha\in\mathfrak{t}-\{0\}$
we define a linear isometry $w_{\alpha}\in O(\mathfrak{t})$ by
\begin{equation*}\label{eqn:wWeyl}
w_{\alpha}(H)=H-2\dfrac{\INN{\alpha}{H}}{\|\alpha\|^{2}}\alpha\quad
(H\in\mathfrak{t}).
\end{equation*}
Then $w_{\alpha}$
satisfies $w_{\alpha}^{2}=1$ and $w_{\alpha}(\alpha)=-\alpha$.

\begin{dfn}
A finite subset $\Delta\subset\mathfrak{t}-\{0\}$ is called a \textit{root system}
of $\mathfrak{t}$, if it satisfies the following two conditions:
\begin{enumerate}
\item $\mathfrak{t}=\mathrm{span}_{\mathbb{R}}(\Delta)$.
\item If $\alpha$ and $\beta$ are in $\Delta$,
then $w_{\alpha}(\beta)=\beta-2\dfrac{\INN{\alpha}{\beta}}{\|\alpha\|^{2}}\alpha$ is in $\Delta$, and
$2\dfrac{\langle\alpha ,\beta\rangle}{\|\alpha\|^2}$ is in $\mathbb{Z}$.
\end{enumerate}
In addition, a root system $\Delta$ is said to be \textit{reduced},
if it satisfies the following condition:
\begin{enumerate}
\setcounter{enumi}{2}
\item If $\alpha$ and $\beta$ are in $\Delta$ with $\beta = m\alpha$,
then $m=\pm 1$ holds.
\end{enumerate}
\end{dfn}
A root system $\Delta$
of $\mathfrak{t}$
is said to be reducible
if
there exist two non-empty subsets $\Delta_{1}$
and $\Delta_{2}$ of $\Delta$ satisfying the following conditions:
\[
\Delta=\Delta_{1}\cup\Delta_{2},\quad
\Delta_{1}\cap\Delta_{2}=\emptyset,\quad
\INN{\Delta_{1}}{\Delta_{2}}=\{0\}.
\]
Otherwise it is said to be \textit{irreducible}.
Any root system is decomposed into irreducible ones,
namely, there exist unique irreducible root systems $\Delta_{1},\dotsc,\Delta_{l}$
up to permutation of the indices
such that
$\Delta=\Delta_{1}\cup \dotsb \cup \Delta_{l}$
and that $\INN{\Delta_{i}}{\Delta_{j}}=\{0\}$ for $1\leq i\neq j\leq l$.
This decomposition of $\Delta$ is called the irreducible decomposition of $\Delta$.

Let $\Delta$ and $\Delta'$ be reduced
root systems of $\mathfrak{t}$ and $\mathfrak{t}'$, respectively.
It is shown that,
if $\Delta$
and $\Delta'$ are irreducible,
then,
for any linear isomorphism $\varphi:\mathfrak{t}\to\mathfrak{t}'$
satisfying $\varphi(\Delta)=\Delta'$, we have
\[
2\dfrac{\INN{\beta}{\alpha}}{\|\alpha\|^{2}}
=2\dfrac{\INN{\varphi(\beta)}{\varphi(\alpha)}}{\|\varphi(\alpha)\|^{2}},\quad
\alpha,\beta\in\Delta.
\]
Based on this observation we define an isomorphism of root systems as follows:
A linear isomorphism $\varphi:\mathfrak{t}\to\mathfrak{t}'$
satisfying $\varphi(\Delta)=\Delta'$
is called an \textit{isomorphism} of root systems between $\Delta$ and $\Delta'$.
Two root systems $\Delta$ and $\Delta'$ are \textit{isomorphic},
which we write $\Delta\simeq\Delta'$,
if there exists such $\varphi$.
Then we have $\varphi(\Delta)=\Delta'$.
We find that
$\simeq$ gives an equivalence relation on the set of root systems.

In the case when $\mathfrak{t}=\mathfrak{t}', \INN{\,}{\,}=\INN{\,}{\,}', \Delta=\Delta'$,
an isomorphism $\varphi:\mathfrak{t}\to\mathfrak{t}$ of $\Delta$
is called an automorphism of $\Delta$.
Denote by $\mathrm{Aut}(\Delta)$ the group of automorphisms of $\Delta$.
It is clear that $\mathrm{Aut}(\Delta)$ is a finite group.
The subgroup of $O(\mathfrak{t})$
generated by $\{w_{\alpha}\mid \alpha\in\Delta\}$
is called the \textit{Weyl group} of $\Delta$, which we write $W(\Delta)$.
Then $W(\Delta)$ is a normal subgroup  of $\mathrm{Aut}(\Delta)$.
In particular, $W(\Delta)$ is a finite group.

\subsection{$\sigma$-systems}

Let $\Delta$ be a reduced root system of $\mathfrak{t}$.
Let $\sigma:\mathfrak{t}\to\mathfrak{t}$ be an involutive linear isometry of $\Delta$,
which we call an involution.
Then the pair $(\Delta, \sigma)$ is called a \textit{$\sigma$-system} of
$\mathfrak{t}$.
If we put $\mathfrak{t}^{\pm\sigma}:=\{H\in\mathfrak{t}\mid
\sigma(H)=\pm H\}$,
then we have an orthogonal decomposition
$\mathfrak{t}=\mathfrak{t}^{\sigma}\oplus\mathfrak{t}^{-\sigma}$
with respect to the inner product $\INN{\;}{\;}$.
The rank of $(\Delta,\sigma)$
is defined as the dimension of $\mathfrak{t}^{\sigma}$,
which we write $\mathrm{rank}(\Delta,\sigma)$.
By definition,
we have $\mathrm{rank}(\Delta,\sigma)\leq \mathrm{rank}(\Delta)$.
Let $pr:\mathfrak{t}\to\mathfrak{t}^{\sigma}$
denote the orthogonal projection,
that is,
\[
pr:\mathfrak{t}\to\mathfrak{t}^{\sigma};~
H\mapsto \dfrac{1}{2}(H+\sigma(H)).
\]
Set $\Delta_{0}:=\{\alpha\in\Delta\mid pr(\alpha)=0\}=\{
\alpha\in\Delta\mid\sigma(\alpha)=-\alpha\}$.
Then $\Delta_{0}$
satisfies
$\Delta_{0}=-\Delta_{0}$
and $\alpha+\beta\in\Delta_{0}$
for all $\alpha,\beta\in\Delta_{0}$ with $\alpha+\beta\in\Delta$.
We call such a subset of $\Delta$
a closed subsystem of $\Delta$.
Then
$\Delta_{0}$ becomes a root system of
$\mathrm{span}_{\mathbb{R}}(\Delta_{0})$.

A $\sigma$-system $(\Delta,\sigma)$
is said to be $\sigma$-reducible if
there exist two non-empty $\sigma$-invariant subsets $\Delta_{1}$
and $\Delta_{2}$ of $\Delta$ satisfying the following conditions:
\[
\Delta=\Delta_{1}\cup\Delta_{2},\quad
\Delta_{1}\cap\Delta_{2}=\emptyset,\quad
\INN{\Delta_{1}}{\Delta_{2}}=\{0\}.
\]
Otherwise it is said to be \textit{$\sigma$-irreducible}.
Any $\sigma$-system is decomposed into $\sigma$-irreducible ones,
that is, there exist unique
mutually orthogonal,
$\sigma$-irreducible $\sigma$-systems
$(\Delta_{1},\sigma_{1})$,
$\dotsc$,
$(\Delta_{l},\sigma_{l})$
up to permutation of the indices
such that
$\Delta=\Delta_{1}\cup\dotsb\cup\Delta_{l}$
and $\sigma=\sigma_{j}$ holds on $\Delta_{j}$ for each $1\leq j\leq l$.
Then this decomposition
is called the $\sigma$-irreducible decomposition of the $\sigma$-system $(\Delta, \sigma)$, which we write
\[
(\Delta, \sigma)=(\Delta_{1},\sigma_{1})\cup
\dotsb\cup(\Delta_{l},\sigma_{l}).
\]
It is clear that $(\Delta,\sigma)$ is $\sigma$-irreducible if $\Delta$ is irreducible as a root system.
Two $\sigma$-systems $(\Delta,\sigma)$ and $(\Delta',\sigma')$
are said to be isomorphic,
which we write
$(\Delta,\sigma)\simeq (\Delta',\sigma')$,
if there exists an isomorphism $\varphi: \mathfrak{t}\to \mathfrak{t}'$
of root systems satisfying $\sigma'=\varphi\sigma\varphi^{-1}$.
We call such $\varphi$ an isomorphism of $\sigma$-systems.
Then $\simeq$ gives an equivalence relation on the set of $\sigma$-systems.
We find that
if
$(\Delta,\sigma)\simeq (\Delta',\sigma')$,
then their rank are the same.

\subsection{Normal $\sigma$-systems and their Satake diagrams}\label{sec:NssSd}
A $\sigma$-system $(\Delta,\sigma)$ is said to be \textit{normal} if
$\sigma(\alpha)-\alpha\notin\Delta$ for all $\alpha\in\Delta$.
For a normal $\sigma$-system $(\Delta,\sigma)$,
Araki (\cite{Araki})
proved that
the set $\{pr(\alpha) \mid \alpha \in \Delta-\Delta_{0}\}=:\Sigma$ becomes a root system of $\mathfrak{t}^{\sigma}$ (see also \cite[Proposition 1.1.3.1]{Warner}),
which is called the \textit{restricted root system} of $(\Delta,\sigma)$.
Then we have $\mathrm{rank}(\Sigma)=\mathrm{rank}(\Delta,\sigma)$.
The equivalence relation $\simeq$ is compatible
with the normality of a $\sigma$-system.
Namely, if $(\Delta,\sigma)\simeq (\Delta',\sigma')$ and $(\Delta,\sigma)$ is normal,
then $(\Delta',\sigma')$ is also normal.
In addition, if we denote by $\Sigma'$ the restricted root system of $(\Delta',\sigma')$,
then $\Sigma \simeq \Sigma'$ holds as root systems.

Now, let us recall the notion of Satake diagrams for normal $\sigma$-systems.
Let $(\Delta,\sigma)$ be a normal  $\sigma$-system.
Let $\Pi$ be a fundamental system of $\Delta$.
The positive root system $\Delta^{+}$ for $\Pi$
is described by
$\Delta^{+}=\{\sum_{\alpha\in\Pi}m_{\alpha}\alpha\in\Delta\mid m_{\alpha}\in\mathbb{Z}_{\geq 0}\}$,
where $\mathbb{Z}_{\geq 0}:=\{m\in\mathbb{Z} \mid m \geq 0\}$.
Then
$\Pi$ is called a \textit{$\sigma$-fundamental system},
if $\sigma(\alpha)$ is in $\Delta^{+}$
for all $\alpha\in\Delta^{+}-\Delta_{0}$.
It is known that
a $\sigma$-fundamental system
always exists (cf.~\cite[p.~11]{Araki}).
The following lemma will be needed later.

\begin{lem}\label{lem:Pi_varphiPi}
Let $\Pi$ be a $\sigma$-fundamental system.
For any $\varphi\in \mathrm{Aut}(\Delta)$,
$\varphi(\Pi)$
is a $(\varphi\sigma\varphi^{-1})$-fundamental system of $\Delta$.
\end{lem}

The proof is straightforward and is omitted.
Let $\Pi$ be a $\sigma$-fundamental system of $\Delta$.
It is known that $\Pi\cap\Delta_{0}=: \Pi_{0}$ is a fundamental system of $\Delta_{0}$ (cf.~\cite[p.~23]{Warner}).
Denote by $(\Pi_{0})_{\mathbb{Z}}$ the $\mathbb{Z}$-submodule of $\mathfrak{t}$
generated by $\Pi_{0}$.
It follows from \cite[Lemma 1.1.3.2]{Warner} that
there exists uniquely a permutation $p:\Pi-\Pi_{0}\to\Pi-\Pi_{0}$ of order two such that
\[
\sigma(\alpha)\equiv p(\alpha)\quad
(\operatorname{mod}(\Pi_{0})_{\mathbb{Z}}),
\]
which is called the \textit{Satake involution}
of $(\Delta,\sigma)$ associated with $\Pi$.
Then the \textit{Satake diagram} $S=S(\Pi,\Pi_{0},p)$ of $(\Delta,\sigma)$ associated with $\Pi$
is described as follows:
In the Dynkin diagram of $\Pi$,
every root in $\Pi_{0}$ is replaced from a white circle to a black circle,
and two roots $\alpha,\alpha'\in\Pi-\Pi_{0}$ with $\alpha\neq\alpha'$
are connected by a curved arrow
if $p(\alpha)=\alpha'$.

The normal $\sigma$-system
$(\Delta,\sigma)$ can be reconstructed from $S(\Pi,\Pi_{0},p)$.
The Dynkin diagram of $\Pi$
determines the structures
of $\Delta$,
$\mathfrak{t}=\mathrm{span}_{\mathbb{R}}(\Pi)$ and $\INN{\,}{\,}$.
We write $\Pi=\{\alpha_{1},\dotsc,\alpha_{l}\}$ with $l=\mathrm{rank}(\Delta)$.
By renumbering the indices if necessary,
we may assume that 
there exists $l_{1},l_{2}\leq l$
such that
\[\Pi-\Pi_{0}=\{\alpha_{1},\dotsc,\alpha_{l_{1}},
\alpha_{l_{1}+1},
\dotsc,
\alpha_{l_{1}+l_{2}},
\alpha_{l_{1}+l_{2}+1},\dotsc,
\alpha_{l_{1}+2l_{2}}\},
\] and
\[
p(\alpha_{j})=\alpha_{j}~(1\leq j\leq l_{1}),\quad
p(\alpha_{l_{1}+j'})=\alpha_{l_{1}+l_{2}+j'}~(1\leq j'\leq l_{2}).
\]
In particular,
$l_{2}$ is equal to the number of arrows in $S(\Pi,\Pi_{0},p)$.
From this assumption the cardinality of $\Pi_{0}$
is equal to $l-(l_{1}+2l_{2})=:l_{0}$.
Clearly, $\Pi_{0}=\{\alpha_{l-l_{0}+1},\dotsc,\alpha_{l}\}$ holds.
For $1\leq j\leq l_{1}$,
the condition $p(\alpha_{j})=\alpha_{j}$
implies that
\begin{equation}\label{eqn:satakesigma1}
\alpha_{j}-\sigma(\alpha_{j})
\in\sum_{k=l-l_{0}+1}^{l}\mathbb{Z}\alpha_{k}\subset\mathfrak{t}^{-\sigma}.
\end{equation}
Furthermore,
for $1\leq j'\leq l_{2}$,
from $p(\alpha_{l_{1}+j'})=\alpha_{l_{1}+l_{2}+j'}$
we have
\begin{equation}\label{eqn:satakesigma2}
\mathfrak{t}^{-\sigma}\ni
\alpha_{l_{1}+j'}-\sigma(\alpha_{l_{1}+j'})
=
\alpha_{l_{1}+j'}-\alpha_{l_{1}+l_{2}+j'}+\sum_{k=l-l_{0}+1}^{l}m_{k}\alpha_{k},
\end{equation}
for some integers $m_{l-l_{0}+1},\dotsc,m_{l}$.
Hence it follows from \eqref{eqn:satakesigma1} and \eqref{eqn:satakesigma2}
that the $(-1)$-eigenspace 
$\mathfrak{t}^{-\sigma}$ of $\sigma$ in $\mathfrak{t}$
has the following description:
\[
\mathfrak{t}^{-\sigma}
=\sum_{k=1}^{l}\mathbb{R}(\alpha_{k}-\sigma(\alpha_{k}))
=\sum_{j'=1}^{l_{2}}\mathbb{R}(\alpha_{l_{1}+j'}-\alpha_{l_{1}+l_{2}+j'})
\oplus
\sum_{k=l-l_{0}+1}^{l}\mathbb{R}\alpha_{k}.
\]
In addition, we obtain $\mathfrak{t}^{\sigma}$
as the orthogonal complement of $\mathfrak{t}^{-\sigma}$ in $\mathfrak{t}$.
Thus, the action of $\sigma$ on $\Delta$ is reconstructed.
In particular, we get $\mathrm{rank}(\Delta,\sigma)=
l-(l_{2}+l_{0})=l_{1}+l_{2}$.

Let us explain that
the definition of the Satake diagram of $(\Delta,\sigma)$
is independent of the choice of $\sigma$-fundamental systems.
Suppose that $\tilde{\Pi}$ is another $\sigma$-fundamental system of $\Delta$.
Set $\tilde{\Pi}_{0}:=\tilde{\Pi}\cap\Delta_{0}$.
We denote by $\tilde{p}$ the Satake involution associated with $\tilde{\Pi}$.
Then it follows from \cite[Proposition A in Appendix]{Satake} that
there exists $w\in W(\Delta)$ satisfying 
$\tilde{\Pi}=w(\Pi)$ and $w\sigma=\sigma w$.
Thus, we have $\tilde{\Pi}_{0}=w(\Pi_{0})$ and $w(p(\alpha))=\tilde{p}(w(\alpha))$ for all $\alpha\in\Pi-\Pi_{0}$.
Then we write
\begin{equation*}\label{eqn:satake_eq}
S(\Pi,\Pi_{0},p)=S(\Pi',\Pi_{0}',p').
\end{equation*}

\begin{dfn}\label{dfn:satake_simeq}
Let $(\Delta,\sigma)$ and $(\Delta',\sigma')$ be two normal $\sigma$-systems,
and $S(\Pi,\Pi_{0},p)$ and $S(\Pi',\Pi_{0}',p')$ denote the Satake diagrams of $(\Delta,\sigma)$
and $(\Delta',\sigma')$, respectively.
We write $S(\Pi,\Pi_{0},p)\simeq S(\Pi',\Pi'_{0},p')$
if there exists an isomorphism $\psi:\Pi\to\Pi'$ of Dynkin diagrams
such that $\psi(\Pi_{0})=\Pi_{0}'$ and $\psi(p(\alpha))=p'(\psi(\alpha))$ for all $\alpha \in \Pi-\Pi_{0}$.
We call such $\psi$ an isomorphism of Satake diagrams.
Then $\simeq$ gives an equivalence relation
for Satake diagrams.
\end{dfn}

By the reconstruction of
$(\Delta, \sigma)$ from $S(\Pi,\Pi_{0},p)$
we have the following lemma:

\begin{lem}\label{lem:sigma_satake}
Retain the notation as in Definition $\ref{dfn:satake_simeq}$.
Then,
$(\Delta,\sigma)\simeq (\Delta',\sigma')$ if and only if $S(\Pi,\Pi_{0},p)\simeq S(\Pi',\Pi_{0}',p')$.
In particular, any isomorphism of Satake diagrams
can be extended to an isomorphism of $\sigma$-systems.
\end{lem}

\subsection{Compact symmetric pairs and their Satake diagrams}

Let $G$ be a compact connected
semisimple Lie group,
and $\theta$ be an involution of $G$.
We call the pair $(G,\theta)$
a compact symmetric pair.
Denote by $\mathfrak{g}$ the Lie algebra of $G$.
Fix an $\mathrm{ad}(\mathfrak{g})$-invariant inner product $\INN{\,}{\,}$
on $\mathfrak{g}$.
The differential $d\theta$ of $\theta$ at the identity element in $G$
gives an involution of $\mathfrak{g}$, which we write the same symbol $\theta$
if there is no confusion.
Let $\mathfrak{g}=\mathfrak{g}^{\theta}\oplus\mathfrak{g}^{-\theta}=:\mathfrak{k}\oplus\mathfrak{m}$
be the canonical decomposition of $\mathfrak{g}$ for $\theta$.
Take a maximal abelian subalgebra $\mathfrak{t}$ of $\mathfrak{g}$
such that $\mathfrak{t}\cap\mathfrak{m}$ is a maximal abelian subspace of $\mathfrak{m}$.
This implies that $\mathfrak{t}$ is $\theta$-invariant.

Let $\Delta(\subset \mathfrak{t})$ denote the root system of $\mathfrak{g}$
with respect to $\mathfrak{t}$.
Since $\theta$ induces an automorphism of $\Delta$,
the pair $(\Delta, \sigma):=(\Delta, -\theta|_{\mathfrak{t}})$
gives a $\sigma$-system of $\mathfrak{t}$.
It follows from \cite[Lemma 1.1.3.6]{Warner}
that $(\Delta,\sigma)$ is normal.
We call it the $\sigma$-system of $(G, \theta)$ (or $(\mathfrak{g},\theta)$) for $\mathfrak{t}$.
We will show that
the $\sigma$-system
$(\Delta,\sigma)$
 is uniquely determined from $(G,\theta)$ up to isomorphism.
Let $\mathfrak{t}'$ be another
maximal abelian subalgebra of $\mathfrak{g}$
such that $\mathfrak{t}'\cap\mathfrak{m}$ is a maximal abelian subspace of $\mathfrak{m}$.
Denote by $(\Delta',\sigma'):=(\Delta',-\theta|_{\mathfrak{t}'})$
the $\sigma$-system of $(G,\theta)$ for $\mathfrak{t}'$.
Let $K$ be the
identity component of the fixed point subgroup of $\theta$ in $G$.
From the $\mathrm{Ad}(K)$-conjugacy of maximal abelian subspaces of $\mathfrak{m}$,
there exists $k\in K$ satisfying
$\mathfrak{t}'\cap\mathfrak{m}=\mathrm{Ad}(k)(\mathfrak{t}\cap\mathfrak{m})$.
In addition,
there also exists $k'\in K$
satisfying the following relations (cf.~\cite[Proposition 5]{Sugiura}):
\[
\mathrm{Ad}(k')(H)=H \quad (H\in\mathfrak{t}'\cap\mathfrak{m}),\quad
\mathrm{Ad}(k')(\mathrm{Ad}(k)(\mathfrak{t}\cap\mathfrak{k}))=\mathfrak{t}'\cap\mathfrak{k}.
\]
If we put $\tilde{k}:=k'k\in K$,
then we have
$\mathrm{Ad}(\tilde{k})(\mathfrak{t}\cap\mathfrak{m})=\mathfrak{t}'\cap\mathfrak{m}$ and
$\mathrm{Ad}(\tilde{k})(\mathfrak{t}\cap\mathfrak{k})=\mathfrak{t}'\cap\mathfrak{k}$.
This obeys $\mathrm{Ad}(\tilde{k})(\mathfrak{t})=\mathfrak{t}'$.
Thus, we obtain
\begin{equation*}\label{eqn:csp_sigma}
(\Delta',\sigma')
=(\mathrm{Ad}(\tilde{k})(\Delta),-(\mathrm{Ad}(\tilde{k})\theta\mathrm{Ad}(\tilde{k})^{-1})|_{\mathrm{Ad}(\tilde{k})(\mathfrak{t})})
\simeq (\Delta, -\theta|_{\mathfrak{t}})=(\Delta,\sigma).
\end{equation*}
We define the Satake diagram of $(G,\theta)$
as that of $(\Delta,\sigma)$,
which is uniquely determined up to isomorphisms
due to Lemma \ref{lem:sigma_satake}.

Two compact symmetric pairs
$(G,\theta)$ and $(G,\theta')$
are said to be isomorphic,
which we write $(G,\theta)\simeq (G,\theta')$,
if there exists $\varphi\in\mathrm{Aut}(G)$
satisfying $\theta'=\varphi\theta\varphi^{-1}$.
Then their Satake diagrams are isomorphic
in the sense of Definition \ref{dfn:satake_simeq}.
In order to consider
the converse (see Theorem \ref{thm:cps_sigma_satake_equiv}
for the precise statement),
we will recall the result
for compact symmetric pairs
due to Araki (\cite{Araki}).

Let $(G,\theta)$ and $(G,\theta')$
be two compact symmetric pairs.
Let $\mathfrak{t}$
be a maximal abelian subalgebra of $\mathfrak{g}$
such that $\mathfrak{t}\cap\mathfrak{m}$
and $\mathfrak{t}\cap\mathfrak{m}'$
are maximal abelian subspace of $\mathfrak{m}$
and $\mathfrak{m}'$, respectively.
The following theorem
states that,
if their differentials $d\theta$
and $d\theta'$
coincide with each other on $\mathfrak{t}$,
then
$\theta$
and $\theta'$ are the same on the whole $G$
up to the $\mathrm{Int}(G)$-conjugacy.

\begin{thm}[Araki]\label{thm:araki}
Retain the notation above.
Assume that 
$d\theta|_{\mathfrak{t}}=d\theta'|_{\mathfrak{t}}$ holds.
Then, 
there exists $H\in\mathfrak{t}\cap\mathfrak{m}$
satisfying $d\theta'=e^{\mathrm{ad}(H)}d\theta e^{-\mathrm{ad}(H)}$.
In addition, $\theta'=\tau_{h}\theta\tau_{h}^{-1}$
holds for $h=\exp(H)$, where 
$\tau_{h}$ denotes the inner automorphism of $G$ defined by
$g\mapsto hgh^{-1}$.
\end{thm}

Here, we note that,
under the assumption of this theorem,
$d\theta|_{\mathfrak{t}}=d\theta'|_{\mathfrak{t}}$ implies
$\mathfrak{t}\cap\mathfrak{m}=\mathfrak{t}\cap\mathfrak{m}'$.
Theorem \ref{thm:araki} will be used in the proof of Theorem
\ref{thm:cst_dsatake_sim}.
For the completeness of our proof of 
Theorem \ref{thm:cst_dsatake_sim},
we will prove Theorem \ref{thm:araki}
(see \cite[Theorem 2.14]{Araki} for the original statement and its proof).

For this purpose we need some preparation.
We first restate Theorem \ref{thm:araki} in terms of the complexification of $\mathfrak{g}$.
Let $\mathfrak{g}^{\mathbb{C}}$ denote
the complexification of $\mathfrak{g}$.
We write $d\theta^{\mathbb{C}}$
and $d\theta'^{\mathbb{C}}$
as the complexifications of $d\theta$
and $d\theta'$, respectively.
Then it is sufficient to show
that,
if $d\theta|_{\mathfrak{t}}=d\theta'|_{\mathfrak{t}}$ holds,
then
there exists $H\in\mathfrak{t}\cap\mathfrak{m}$
satisfying $d\theta'^{\mathbb{C}}=e^{\mathrm{ad}(H)}d\theta^{\mathbb{C}} e^{-\mathrm{ad}(H)}$.
In order to give such $H$,
we next recall
the result for compact symmetric pairs
due to Klein (\cite{Klein}).
In fact,
following to his result,
we can obtain a description for
the action of $d\theta^{\mathbb{C}}$
on $\mathfrak{g}^{\mathbb{C}}$
by means of the corresponding Satake diagram.

Let $(G,\theta)$ be a compact symmetric pair.
Take a maximal abelian subalgebra $\mathfrak{t}$ of $\mathfrak{g}$
such that 
$\mathfrak{t}\cap\mathfrak{m}$ is a maximal abelian subspace of $\mathfrak{m}$.
Denote by $\Delta$ the root system of $\mathfrak{g}$ with respect to $\mathfrak{t}$.
We write the root space decomposition of the complexification $\mathfrak{g}^{\mathbb{C}}$
as follows:
\[
\mathfrak{g}^{\mathbb{C}}=\mathfrak{t}^{\mathbb{C}}\oplus\sum_{\alpha\in\Delta}\mathfrak{g}(\mathfrak{t},\alpha),
\]
where $\mathfrak{g}(\mathfrak{t},\alpha):=\{X\in\mathfrak{g}^{\mathbb{C}}\mid [H,X]=\sqrt{-1}\INN{\alpha}{H}X,\,H\in\mathfrak{t}\}$.
For each $\alpha\in\Delta$,
$\mathfrak{g}(\mathfrak{t},\alpha)$
is a complex one dimensional subspace of $\mathfrak{g}^{\mathbb{C}}$.
A family $\{X_{\alpha}\}_{\alpha\in\Delta}$ of vectors in $\mathfrak{g}^{\mathbb{C}}$
is called a \textit{Chevalley basis} of $\mathfrak{g}^{\mathbb{C}}$,
if it satisfies the following conditions:
\begin{enumerate}
\item For each $\alpha\in\Delta$,
$X_{\alpha}$ is a nonzero vector in $\mathfrak{g}(\mathfrak{t},\alpha)$.
\item $[X_{\alpha},X_{-\alpha}]=-\sqrt{-1}\alpha$ for $\alpha\in\Delta$.
\item There exists a family
$\{c_{\alpha,\beta} \mid \alpha,\beta\in\Delta,
\alpha+\beta\in\Delta\}$ of real numbers 
satisfying
$[X_{\alpha},X_{\beta}]=c_{\alpha,\beta}X_{\alpha+\beta}$
and $c_{\alpha,\beta}=-c_{-\alpha,-\beta}$.
\item $[X_{\alpha},X_{\beta}]=0$ for $\alpha,\beta\in\Delta$ with $\alpha+\beta\not\in\Delta\cup\{0\}$.
\end{enumerate}
For formal reasons we put $c_{\alpha,\beta}=0$
for $\alpha,\beta\in\Delta$ with $\alpha+\beta\not\in\Delta\cup\{0\}$.
Then
$\{c_{\alpha,\beta}\}$
is called the \textit{Chevalley constants} associated with $\{X_{\alpha}\}_{\alpha\in\Delta}$.
We note that 
a Chevalley basis is not a basis of 
the whole $\mathfrak{g}^{\mathbb{C}}$ but 
the subspace $\sum_{\alpha\in\Delta}\mathfrak{g}(\mathfrak{t},\alpha)$.
It is known that
a Chevalley basis exists
(see \cite[Theorem 6.6, Chapter VI]{Knapp} for the proof).

We extend $\INN{\,}{\,}$
to a complex bilinear form on $\mathfrak{g}^{\mathbb{C}}$,
which is denote by the same symbol $\INN{\,}{\,}$.
Then it is $\mathrm{ad}(\mathfrak{g}^{\mathbb{C}})$-invariant
and nondegenerate. 
For each $\alpha\in\Delta$,
by taking the scalar product of both sides
of $[X_{\alpha},X_{-\alpha}]=-\sqrt{-1}\alpha$ with 
$\alpha$ we get $\INN{X_{\alpha}}{X_{-\alpha}}=-1$.
We write $\overline{X}$ the complex conjugate of $X\in\mathfrak{g}^{\mathbb{C}}$
with respect to $\mathfrak{g}$ in $\mathfrak{g}^{\mathbb{C}}$.
By a similar argument in the proof of
\cite[Proposition 3.5]{Klein}
we obtain the following lemma.

\begin{lem}\label{lem:Cbasis_conj}
There exists a Chevalley basis $\{X_{\alpha}\}_{\alpha\in\Delta}$ of $\mathfrak{g}^{\mathbb{C}}$
which satisfies
$\overline{X_{\alpha}}=-X_{-\alpha}$ for $\alpha\in\Delta$.
Then we have $\INN{X_{\alpha}}{\overline{X_{\alpha}}}=1$.
\end{lem}

Let $\{X_{\alpha}\}_{\alpha\in\Delta}$
be a Chevalley basis of $\mathfrak{g}^{\mathbb{C}}$.
The complexification of $\theta$ will be denoted by the same symbol
$\theta$.
For each $\alpha\in\Delta$,
it follows
from $\theta(\mathfrak{g}(\mathfrak{t},\alpha))=\mathfrak{g}(\mathfrak{t},\theta(\alpha))$ that
there exists a nonzero complex number $s_{\alpha}$
satisfying $\theta(X_{\alpha})=s_{\alpha}X_{\theta(\alpha)}$.
The family $\{s_{\alpha}\}_{\alpha\in\Delta}$
is called the \textit{Klein constants}
of $(\mathfrak{g},\theta)$ associated with 
$\{X_{\alpha}\}_{\alpha\in\Delta}$.
Then we get $s_{\theta(\alpha)}=s_{\alpha}^{-1}$
because of $\theta^{2}=1$.
Furthermore,
if $\overline{X}_{\alpha}=-X_{-\alpha}$ holds
for $\alpha\in\Delta$,
then
$\{s_{\alpha}\}_{\alpha\in\Delta}$ has the following properties:

\begin{lem}[{\cite[Proposition 4.1]{Klein}}]\label{lem:Klein_pro4.1}
Assume that $\{X_{\alpha}\}_{\alpha\in\Delta}$
satisfies $\overline{X_{\alpha}}=-X_{-\alpha}$ for $\alpha\in\Delta$.
Let $\alpha$ and $\beta$ be in $\Delta$.
\begin{enumerate}
\item We have $s_{-\alpha}=\overline{s_{\alpha}}=s_{\alpha}^{-1}$.
In particular, $|s_{\alpha}|=1$ holds.
\item If $\theta(\beta)=\beta$, then
$\mathfrak{g}(\mathfrak{t},\beta)\subset\mathfrak{k}^{\mathbb{C}}$.
In particular, we get $s_{\beta}=1$.
\end{enumerate}
\end{lem}

Fix an element $H\in\mathfrak{t}$.
We have another involution $\theta':=e^{\mathrm{ad}(H)}\theta e^{-\mathrm{ad}(H)}$
of $\mathfrak{g}$,
which satisfies $\theta'|_{\mathfrak{t}}=\theta|_{\mathfrak{t}}$.
Set $\mathfrak{m}':=\mathfrak{g}^{-\theta'}=e^{\mathrm{ad}(H)}(\mathfrak{m})$.
Then we obtain $\mathfrak{t}\cap\mathfrak{m}'=e^{\mathrm{ad}(H)}(\mathfrak{t}\cap\mathfrak{m})$,
from which $\mathfrak{t}\cap\mathfrak{m}'$
is a maximal abelian subspace of $\mathfrak{m}'$.
Denote by $\{s'_{\alpha}\}_{\alpha\in\Delta}$
the Klein constants of $(\mathfrak{g},\theta')$ associated with $\{X_{\alpha}\}_{\alpha\in\Delta}$.
Then 
we have
$s_{\alpha}'=e^{\sqrt{-1}\INN{H}{\theta(\alpha)-\alpha}}s_{\alpha}$
for $\alpha\in\Delta$.

The following lemma is useful in our proof of Theorem \ref{thm:araki}.

\begin{lem}\label{lem:uehtgs}
Let $\{X_{\alpha}\}_{\alpha\in\Delta}$
be a Chevalley basis of $\mathfrak{g}^{\mathbb{C}}$
satisfying $\overline{X_{\alpha}}=-X_{-\alpha}$ for $\alpha\in\Delta$.
We put $\Delta_{0}=\{\alpha\in\Delta\mid \theta(\alpha)=\alpha\}$.
Assume that 
$\gamma_{1},\dotsc,\gamma_{r}$
$(r\in\mathbb{N})$
are in $\Delta-\Delta_{0}$ which satisfy
$\{\theta(\gamma_{1})-\gamma_{1},\dotsc,\theta(\gamma_{r})-\gamma_{r}\}(\subset \mathfrak{t}\cap\mathfrak{m})$ are linearly independent.
Then, for any $u_{1},\dotsc,u_{r}\in U(1)$,
there exists $H\in\mathfrak{t}\cap\mathfrak{m}$ satisfying the following relation:
\begin{equation}\label{eqn:ues}
u_{j}=e^{\sqrt{-1}\INN{H}{\theta(\gamma_{j})-\gamma_{j}}}s_{\gamma_{j}}\quad
(1\leq j\leq r).
\end{equation}
\end{lem}

\begin{proof}
For each $1\leq j\leq r$,
it follows from
Lemma \ref{lem:Klein_pro4.1}, (1)
that
there exists $t_{j}\in\mathbb{R}$ satisfying $s_{\gamma_{j}}=e^{\sqrt{-1}t_{j}}$.
Since $u_{j}$ is also in $U(1)$,
there exists $v_{j}\in\mathbb{R}$ satisfying $u_{j}=e^{\sqrt{-1}v_{j}}$.
We define a matrix $C$ by
\[
C:=\bigm(\INN{\theta(\gamma_{j})-\gamma_{j}}{\theta(\gamma_{k})-\gamma_{k}}\bigm)_{1\leq j,k\leq r}.
\]
It follows from the assumption that the square matrix $C$ is invertible.
Let $h_{1},\dotsc,h_{r}$ be real numbers defined by
\[
\left(
\begin{array}{c}
h_{1}\\
\vdots\\
h_{r}
\end{array}
\right):=C^{-1}\left(
\begin{array}{c}
v_{1}-t_{1}\\
\vdots\\
v_{r}-t_{r}
\end{array}
\right).
\]
Then, if we put
$H:=\sum_{k=1}^{r}h_{k}(\theta(\gamma_{k})-\gamma_{k})\in\mathfrak{t}\cap\mathfrak{m}$,
then the following relation holds:
\[
v_{j}=t_{j}+\INN{H}{\theta(\gamma_{j})-\gamma_{j}}\quad
(1\leq j\leq r).
\]
Thus we obtain the assertion because $H$ satisfies (\ref{eqn:ues}).
\end{proof}

Now, we are ready to prove Theorem \ref{thm:araki}.

\begin{proof}[Proof of Theorem $\ref{thm:araki}$]
Let $\{X_{\alpha}\}_{\alpha\in\Delta}$
be a Chevalley basis of $\mathfrak{g}^{\mathbb{C}}$
with $\overline{X_{\alpha}}=-X_{-\alpha}$.
Denote by $\{s_{\alpha}\}_{\alpha\in\Delta}$ (resp.~$\{s_{\alpha}'\}_{\alpha\in\Delta}$)
the Klein constants of $(\mathfrak{g},\theta)$ (resp.~$(\mathfrak{g},\theta')$)
associated with $\{X_{\alpha}\}_{\alpha\in\Delta}$.
By the assumption $\theta|_{\mathfrak{t}}=\theta'|_{\mathfrak{t}}$
we obtain
$(\Delta,-\theta|_{\mathfrak{t}})=(\Delta,-\theta'|_{\mathfrak{t}})=:(\Delta,\sigma)$.
Set $r:=\mathrm{rank}(\Delta,\sigma)$.
Let $\Pi$ be a $\sigma$-fundamental system of $\Delta$,
and $\Pi_{0}:=\Pi\cap\Delta_{0}$.
We can take $\alpha_{1},\dotsc,\alpha_{r}\in\Pi-\Pi_{0}$
such that $\{\theta(\alpha_{1})-\alpha_{1},\dotsc,\theta(\alpha_{r})-\alpha_{r}\}(\subset \mathfrak{t}\cap\mathfrak{m})$
are linearly independent (cf.~\cite[p.~23]{Warner}).
By applying Lemma \ref{lem:uehtgs} to
$s_{\alpha_{1}}',\dotsc,s_{\alpha_{r}}'\in U(1)$,
there exists $H\in\mathfrak{t}\cap\mathfrak{m}$ satisfying the following relation:
\begin{equation*}\label{eqn:saj'sa_rel}
s_{\alpha_{j}}'=e^{\sqrt{-1}\INN{H}{\theta(\alpha_{j})-\alpha_{j}}}
s_{\alpha_{j}}\quad
(1\leq j\leq r).
\end{equation*}
Then we have $\theta'=e^{\mathrm{ad}(H)}\theta e^{-\mathrm{ad}(H)}$ on
$\sum_{j=1}^{r}\mathfrak{g}(\mathfrak{t},\alpha_{j})$.
Furthermore,
if we put 
\[
\mathfrak{h}
:=\mathfrak{t}\oplus
\sum_{\beta\in\Pi_{0}}\mathfrak{g}(\mathfrak{t},\beta)
\oplus\sum_{j=1}^{r}(\mathfrak{g}(t,\alpha_{j})\oplus
\mathfrak{g}(\mathfrak{t},-\theta(\alpha_{j}))),
\]
then we have
$\theta'=e^{\mathrm{ad}(H)}\theta e^{-\mathrm{ad}(H)}$
on the subset $\mathfrak{h}\cup\overline{\mathfrak{h}}$
of $\mathfrak{g}^{\mathbb{C}}$.
Since $\mathfrak{g}^{\mathbb{C}}$
is generated by
$\mathfrak{h}\cup\overline{\mathfrak{h}}$,
we have $\theta'=e^{\mathrm{ad}(H)}\theta e^{-\mathrm{ad}(H)}$
on $\mathfrak{g}^{\mathbb{C}}$.
Therefore we have completed the proof.
\end{proof}

The following theorem is shown by Lemma \ref{lem:sigma_satake} and
Theorem \ref{thm:araki}.

\begin{thm}\label{thm:cps_sigma_satake_equiv}
Let $(G,\theta)$ and $(G,\theta')$ be two compact symmetric pairs.
Then the followings conditions are equivalent:
\begin{enumerate}
\item $(G,\theta)$ and $(G,\theta')$ are locally isomorphic,
namely, there exists $\varphi\in\mathrm{Aut}(\mathfrak{g})$
satisfying $d\theta'=\varphi d\theta\varphi^{-1}$.
\item The $\sigma$-systems of $(G,\theta)$ and $(G,\theta')$ are isomorphic.
\item The Satake diagrams of $(G,\theta)$ and $(G,\theta')$ are isomorphic.
\end{enumerate}
In addition, in the case when $G$ is simply-connected
or when $G$ is the adjoint group,
$(G,\theta)$ and $(G,\theta')$
are isomorphic if and only if
one of the above conditions $(1)$--$(3)$ holds.
\end{thm}

An abstract $\sigma$-system $(\Delta,\sigma)$
is said to be \textit{admissible},
if there exists a compact symmetric pair whose $\sigma$-system
is isomorphic to $(\Delta,\sigma)$.
Clearly, any admissible $\sigma$-system is normal.
Araki (\cite[No.~5.11]{Araki})
determined the admissibilities of abstract normal $\sigma$-systems
based on the classification.
As a consequence of Theorem \ref{thm:cps_sigma_satake_equiv},
he gave an alternative proof of Cartan's classification
for compact symmetric pairs at the Lie algebra level.
Hence the locally isomorphism class
of a compact symmetric pair
is represented by a diagram.
Furthermore,
we can determine the restricted root system of $(G,\theta)$
with multiplicity by means of the Satake diagram,
which characterizes the local isomorphism classes of compact symmetric pairs.
This is a significance to give the alternative proof.
In Section
\ref{sec:cst_dsatake},
we will generalize this method to classify compact symmetric triads
at the Lie algebra level.

Here, in order to
present concrete examples of compact symmetric triads,
we give an explicit description of the classification
for the isomorphism classes of compact symmetric pairs $(\mathfrak{g},\theta)$
at the Lie algebra level.
The following theorem gives a criterion
for two compact symmetric pairs
to be isomorphic to each other.

\begin{thm}\label{thm:csp_iso_iff}
Assume that $\mathfrak{g}$ is simple.
Two compact symmetric pairs
$(\mathfrak{g},\theta)$ and $(\mathfrak{g},\theta')$
are isomorphic if and only if
the fixed point subalgebras
$\mathfrak{k}$ and $\mathfrak{k}'$
are isomorphic as Lie algebras.
\end{thm}

The proof is essentially due to 
Helgason (\cite{Helgason}).

\begin{proof}
The necessity is clear.
In order to prove the sufficiency
we assume that $\mathfrak{k}$ and $\mathfrak{k}'$ are isomorphic.
We extend $\theta$ and $\theta'$
to complex linear 
involutions on $\mathfrak{g}^{\mathbb{C}}$,
which we write $\theta^{\mathbb{C}}$ and $\theta'^{\mathbb{C}}$,
respectively.
Then the fixed point subalgebras of
$\theta^{\mathbb{C}}$
and $\theta'^{\mathbb{C}}$ are isomorphic to each other.
It follows from
\cite[Theorem 6.2, Chapter X]{Helgason}
that 
$\theta^{\mathbb{C}}$
and $\theta'^{\mathbb{C}}$
are $\mathrm{Aut}(\mathfrak{g}^{\mathbb{C}})$-conjugate.
In addition, by \cite[Proposition 1.4, Chapter X]{Helgason}
there exists $\varphi\in\mathrm{Aut}(\mathfrak{g})$
satisfying $\theta'=\varphi\theta\varphi^{-1}$.
Hence the assertion holds.
\end{proof}

From Theorem \ref{thm:csp_iso_iff}
there is no confusion when we write
$[(\mathfrak{g},\mathfrak{k})]$
in place of $[(\mathfrak{g},\theta)]$.
Table \ref{table:fixed_pt_algebra} exhibits
the classification of the fixed point subalgebras
of involutions on $\mathfrak{g}$.
In Section \ref{sec:cst_dsatake},
we will
classify compact simple symmetric triads at the Lie algebra level,
based on the classification for compact simple symmetric pairs.

\begin{table}[H]
\caption{The classification of fixed point subalgebras
of involutions (\cite[TABLE V, p.~518]{Helgason})}\label{table:fixed_pt_algebra}
\centering
\renewcommand{\arraystretch}{1.5}
\begin{tabular}{cc}
\hline
\hline
$\mathfrak{g}$ & Fixed point subalgebra\\
\hline
\hline
$\mathfrak{su}(n)$ & $\mathfrak{so}(n)$, $\mathfrak{sp}(n/2)$ ($n$: even) , $\mathfrak{s}(\mathfrak{u}(a)\oplus\mathfrak{u}(b))$
($a+b=n$) \\
$\mathfrak{so}(n)$ & $\mathfrak{so}(a)\oplus\mathfrak{so}(b)$ ($a+b=n\neq 2,4$), $\mathfrak{u}(n/2)$ ($n\geq 6$, even) \\
$\mathfrak{sp}(n)$ & $\mathfrak{u}(n)$, $\mathfrak{sp}(a)\oplus\mathfrak{sp}(b)$ ($a+b=n$) \\
$\mathfrak{e}_{6}$ & $\mathfrak{sp}(4)$, $\mathfrak{su}(6)\oplus\mathfrak{su}(2)$, $\mathfrak{so}(10)\oplus\mathfrak{so}(2)$, $\mathfrak{f}_{4}$\\
$\mathfrak{e}_{7}$ & $\mathfrak{su}(8)$, $\mathfrak{so}(12)\oplus\mathfrak{su}(2)$, $\mathfrak{e}_{6}\oplus\mathfrak{so}(2)$\\
$\mathfrak{e}_{8}$ & $\mathfrak{so}(16)$, $\mathfrak{e}_{7}\oplus\mathfrak{su}(2)$\\
$\mathfrak{f}_{4}$ & $\mathfrak{sp}(3)\oplus\mathfrak{su}(2)$, $\mathfrak{so}(9)$\\
$\mathfrak{g}_{2}$ & $\mathfrak{su}(2)\oplus\mathfrak{su}(2)$\\
\hline
\hline
\end{tabular}
\renewcommand{\arraystretch}{1.0}
\end{table}

\section{Double Satake diagrams for double $\sigma$-systems}\label{sec:2satake}

In this section, we will introduce the notions
of double $\sigma$-systems
and double Satake diagrams,
which are generalizations
of $\sigma$-systems and Satake diagrams, respectively.
Based on the equivalence relation for compact symmetric triads,
we define equivalence relations for double $\sigma$-systems
and for double Satake diagrams.
In Theorem \ref{thm:dsig_dsatake_equiv}
we give
a necessary and sufficient condition
for two double $\sigma$-systems to be equivalent.
As explained in more detail
in Section \ref{sec:cst_dsatake},
this theorem plays a fundamental role
in the definition of double Satake diagrams
for compact symmetric triads.
We also define the rank and the order
for the equivalence class of a double $\sigma$-systems.
We will discuss
a geometrical meaning of the rank
in Sections \ref{sec:cst_dsatake}.
On the other hand,
we will show the relation
of the ranks and the orders
between compact symmetric triads
and double $\sigma$-systems
in Section \ref{sec:cst_cf}.

\subsection{Double $\sigma$-systems}\label{sec:2sigroot}

Let $\mathfrak{t}$ be a finite dimensional real vector space.
Fix an inner product $\INN{\,}{\,}$ on $\mathfrak{t}$.
Let $\Delta$ be a reduced root system of $\mathfrak{t}$.
For two involutions $\sigma_{1}$ and $\sigma_{2}$ on $\Delta$,
the triplet $(\Delta,\sigma_{1},\sigma_{2})$
is called a \textit{double $\sigma$-system} of $\mathfrak{t}$.
In this paper, 
$\sigma_{1}$ and $\sigma_{2}$
are not necessarily commutative unless otherwise stated.
Based on the equivalence relation for compact symmetric triads
as in Definition \ref{dfm:cst_sim},
we introduce an equivalence relation $\sim$
on double $\sigma$-systems as follows.

\begin{dfn}
Two double $\sigma$-systems
$(\Delta,\sigma_{1},\sigma_{2})$
and $(\Delta',\sigma_{1}',\sigma_{2}')$
are isomorphic, which we write 
$(\Delta,\sigma_{1},\sigma_{2})\sim(\Delta',\sigma_{1}',\sigma_{2}')$,
if there exist an isomorphism $\varphi:\mathfrak{t}\to \mathfrak{t}'$
of root systems between $\Delta$ and $\Delta'$,
and $w'\in W(\Delta')$ satisfying the following relations:
\begin{equation}\label{eqn:dsig_sim}
\sigma_{1}'=\varphi\sigma_{1}\varphi^{-1},\quad
\sigma_{2}'=w'\varphi\sigma_{2}\varphi^{-1}w'^{-1}.
\end{equation}
We write $[(\Delta,\sigma_{1},\sigma_{2})]$ the isomorphism class of 
$(\Delta,\sigma_{1},\sigma_{2})$.
\end{dfn}

A double $\sigma$-system 
$(\Delta,\sigma_{1},\sigma_{2})$
is said to be \textit{normal},
if both 
$(\Delta,\sigma_{1})$
and $(\Delta,\sigma_{2})$ are normal as $\sigma$-systems.
The normality of a double $\sigma$-system is compatible with $\sim$,
namely,
for two double $\sigma$-systems
$(\Delta,\sigma_{1},\sigma_{2})$ and $(\Delta',\sigma_{1}',\sigma_{2}')$
satisfying $(\Delta,\sigma_{1},\sigma_{2})\sim(\Delta',\sigma_{1}',\sigma_{2}')$,
if $(\Delta,\sigma_{1},\sigma_{2})$ is normal,
then so is $(\Delta',\sigma_{1}',\sigma_{2}')$.

\begin{dfn}\label{dfn:dsig_canonical}
Let 
$(\Delta,\sigma_{1},\sigma_{2})$ be a normal double $\sigma$-system.
\begin{enumerate}
\item A fundamental system $\Pi$ of $\Delta$
is called a \textit{$(\sigma_{1},\sigma_{2})$-fundamental system},
if $\Pi$ is both $\sigma_{1}$- and $\sigma_{2}$-fundamental systems.
\item $(\Delta,\sigma_{1},\sigma_{2})$
is said to be \textit{canonical},
if $\Delta$ admits a $(\sigma_{1},\sigma_{2})$-fundamental system.
\end{enumerate}
\end{dfn}

\begin{pro}\label{pro:exist_dfs}
For any normal double $\sigma$-system $(\Delta,\sigma_{1},\sigma_{2})$,
there exists a normal double $\sigma$-system
$(\Delta,\sigma_{1},\sigma_{2}')\sim(\Delta,\sigma_{1},\sigma_{2})$
such that $(\Delta,\sigma_{1},\sigma_{2}')$ is canonical.
\end{pro}

\begin{proof}
For $i=1,2$,
let $\Pi_{i}$ be a $\sigma_{i}$-fundamental system of $\Delta$.
Since $W(\Delta)$ acts transitively on the set of fundamental systems of $\Delta$,
there exists $w\in W(\Delta)$ such that $\Pi_{1}=w(\Pi_{2})=:\Pi$.
If we put $\sigma_{2}':=w\sigma_{2}w^{-1}$,
then $(\Delta,\sigma_{1},\sigma_{2}')\sim(\Delta,\sigma_{1},\sigma_{2})$ holds.
It follows from 
Lemma \ref{lem:Pi_varphiPi}
that $\Pi$ is a $\sigma_{2}'$-fundamental system.
Hence we get the assertion.
\end{proof}

In general, a normal double $\sigma$-system $(\Delta, \sigma_{1},\sigma_{2})$
is not necessarily canonical.
Furthermore,
there exist two normal double $\sigma$-systems
$(\Delta, \sigma_{1},\sigma_{2})\not\sim(\Delta, \sigma_{1},\sigma_{2}')$
such that they are canonical and
that $(\Delta,\sigma_{2})\simeq(\Delta,\sigma_{2}')$ holds.
Before giving an example we prepare the following notation.

\begin{nota}\label{nota:Dr}
Let
$e_{1},\dotsc,e_{r}$
be the canonical basis of $\mathbb{R}^{r}$.
We write
$D_{r}^{+}=\{e_{i}\pm e_{j}\mid 1\leq i< j\leq r\}$
as the set of all the positive roots for the root system of type $D$
with rank $r$ (\cite{Bo}).
Then the following gives the set of all the simple roots for $D_{r}^{+}$:
\[
\Pi=\{\alpha_{1}=e_{1}-e_{2},\dotsc,\alpha_{r-1}=e_{r-1}-e_{r},\alpha_{r}=e_{r-1}+e_{r}\}.
\]
\end{nota}

\begin{ex}\label{ex:so8so3so5_ndsig}
Let $(\Delta,\sigma)$ be 
the $\sigma$-system corresponding to
the compact symmetric pair
$(\mathfrak{so}(8),\mathfrak{so}(3)\oplus\mathfrak{so}(5))$.
Then we have $\Delta=\{\pm e_{i}\pm e_{j}\mid 1\leq i<j \leq 4\}$.
There exists a $\sigma$-fundamental system $\Pi=\{\alpha_{1},\dotsc,\alpha_{4}\}$
of $\Delta$
such that 
its Satake diagram is described as follows:
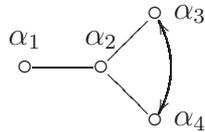
\begin{equation*}
\begin{tabular}{c}
\begin{xy}
\ar@{-}(0,0)*++!D{\alpha_{1}}*{\circ}="a1";(10,0)*++!D{\alpha_{2}}*{\circ}="a2"
\ar@{-}"a2";(17.09,7.09)*++!L{\alpha_{3}}*{\circ}="a3"
\ar@{-}"a2";(17.09,-7.09)*++!L{\alpha_{4}}*{\circ}="a4"
\ar@/^/@{<->} "a3";"a4"
\end{xy}
\end{tabular}
\end{equation*}
Then we have $\sigma:(\alpha_{1},\alpha_{2},\alpha_{3},\alpha_{4})\mapsto
(\alpha_{1},\alpha_{2},\alpha_{4},\alpha_{3})$.
Clearly, $(\Delta, \sigma,\sigma)$
gives a trivial example
of canonical normal double $\sigma$-systems.
In what follows,
we shall give an example of normal double $\sigma$-system
$(\Delta,\sigma,\sigma')$
such that
$(\Delta,\sigma,\sigma')\sim(\Delta,\sigma,\sigma)$ is not canonical.
Furthermore,
we give another example of normal double $\sigma$-system
$(\Delta,\sigma,\sigma'')$
such that $(\Delta,\sigma,\sigma'')$
is canonical,
$(\Delta,\sigma)\simeq(\Delta,\sigma'')$ and
$(\Delta,\sigma,\sigma'')\not\sim(\Delta,\sigma,\sigma)$.

We define an automorphism $w\in \mathrm{Aut}(\Delta)$
by $w:(e_{1},e_{2},e_{3},e_{4})\mapsto(e_{1},e_{2},e_{4},e_{3})$.
Then $w\in W(\Delta)$ holds.
If we put $\sigma':=w\sigma w^{-1}$,
then $(\Delta,\sigma,\sigma')\sim(\Delta,\sigma,\sigma)$
is a normal double $\sigma$-system.
In addition, from $\mathrm{ord}(\sigma\sigma')=2\neq\mathrm{ord}(\sigma\sigma)$,
$(\Delta,\sigma,\sigma')$
cannot be canonical due to Theorem \ref{thm:dsig_dsatake_equiv}
as will be seen later.

Let $\kappa$ be an automorphism of $\Delta$ with order three
defined by
$\kappa:(\alpha_{1},\alpha_{2},\alpha_{3},\alpha_{4})\mapsto(\alpha_{4},\alpha_{2},\alpha_{1},\alpha_{3})$.
We set $\sigma'':=\kappa\sigma\kappa^{-1}$.
Then $(\Delta,\sigma)\simeq(\Delta,\sigma'')$ holds.
In addition, $(\Delta,\sigma'')$ is normal.
Hence the double $\sigma$-system $(\Delta,\sigma,\sigma'')$
is normal.
It follows from $\kappa(\Pi)=\Pi$
that
$\Pi$ becomes a $\sigma''$-fundamental system
by Lemma \ref{lem:Pi_varphiPi}.
This yields that $(\Delta,\sigma,\sigma'')$
is canonical.
Since
the order of
$\sigma\sigma''$
has three,
we have
$\mathrm{ord}(\sigma\sigma)\neq\mathrm{ord}(\sigma\sigma'')$.
Thus
$(\Delta,\sigma,\sigma)\not\sim(\Delta,\sigma,\sigma'')$
holds
by means of Theorem \ref{thm:dsig_dsatake_equiv}.
\end{ex}

\subsection{Double Satake diagrams}\label{sec:sub2satake}

Let $(\Delta,\sigma_{1},\sigma_{2})$
be a canonical normal double $\sigma$-system,
and $\Pi$ be a $(\sigma_{1},\sigma_{2})$-fundamental system of $\Delta$.
Set $\Delta_{i,0}:=\{\alpha\in\Delta\mid \sigma_{i}(\alpha)=-\alpha\}$
for $i=1,2$.
We denote by 
$S_{i}=S(\Pi, \Pi_{i,0}, p_{i})$
the Satake diagram of $(\Delta,\sigma_{i})$
associated with $\Pi$,
where
$\Pi_{i,0}:=\Pi\cap\Delta_{i,0}$
and $p_{i}$ is the Satake involution.
We note that 
these Satake diagrams
$S_{1}$ and $S_{2}$
are described from the common Dynkin diagram of $\Pi$.

\begin{dfn}\label{dfn:2satake}
Retain the notation above.
The pair $(S_{1},S_{2})$ is called the
\textit{double Satake diagram}
of $(\Delta,\sigma_{1},\sigma_{2})$
associated with $\Pi$.
\end{dfn}

Let us prove that
the double Satake diagram $(S_{1},S_{2})$
of $(\Delta, \sigma_{1},\sigma_{2})$
is independent of the choice of $\Pi$.
Let $\Pi'$ be another $(\sigma_{1},\sigma_{2})$-fundamental system of $\Delta$,
and $(S_{1}',S_{2}')$ denote the double Satake diagram of $(\Delta,\sigma_{1},\sigma_{2})$
associated with $\Pi'$.
It follows from 
\cite[Proposition A in Appendix]{Satake}
that there exist
$w_{1}\in W(\Delta)_{\sigma_{1}}$
and
$w_{2}\in W(\Delta)_{\sigma_{2}}$
satisfying $w_{1}(\Pi)=\Pi'=w_{2}(\Pi)$,
where $W(\Delta)_{\sigma_{i}}:=\{w\in W(\Delta)\mid \sigma_{i}w=w\sigma_{i}\}$.
Since the action of $W(\Delta)$ is simply transitive,
we obtain $w:=w_{1}=w_{2}\in W(\Delta)_{\sigma_{1}}\cap W(\Delta)_{\sigma_{2}}$.
Thus, we get
\begin{equation*}\label{eqn:dsatake_eq}
S_{1}=S_{1}',\quad
S_{2}=S_{2}'.
\end{equation*}
Then we write
$(S_{1},S_{2})=(S_{1}',S_{2}')$.

\begin{dfn}\label{dfn:2satake_sim}
Two double Satake diagrams
$(S_{1},S_{2})$
and
$(S_{1}',S_{2}')$
are isomorphic,
if there exists a common isomorphism $\psi$
of Satake diagrams between
$S_{i}$ and $S_{i}'$ for $i=1,2$.
Then we write $(S_{1},S_{2})\sim(S_{1},S_{2})$ for short.
Such $\psi$ is called an isomorphism of double Satake diagrams.
We denote by $[(S_{1},S_{2})]$
the isomorphism class of $(S_{1},S_{2})$.
\end{dfn}

\begin{thm}\label{thm:dsig_dsatake_equiv}
Let $(\Delta,\sigma_{1},\sigma_{2})$
and $(\Delta',\sigma_{1}',\sigma_{2}')$
be two canonical double $\sigma$-systems of $\mathfrak{t}$ and $\mathfrak{t}'$, respectively.
Let $(S_{1},S_{2})$ and
$(S_{1}',S_{2}')$ denote their double Satake diagrams.
Then,
the following three condition are equivalent$:$
\begin{enumerate}
\item $(\Delta,\sigma_{1},\sigma_{2})\sim(\Delta',\sigma_{1}',\sigma_{2}')$.
\item There exists an isomorphism $\varphi:\mathfrak{t}\to\mathfrak{t}'$
of root systems between $\Delta$ and $\Delta'$ satisfying
$\sigma_{i}'=\varphi\sigma_{i}\varphi^{-1}$ for $i=1,2$.
\item $(S_{1},S_{2})\sim(S_{1}',S_{2}')$.
\end{enumerate}
In particular, 
we have 
$\dim(\mathfrak{t}^{\sigma_{1}}\cap\mathfrak{t}^{\sigma_{2}})
=\dim(\mathfrak{t}'^{\sigma_{1}'}\cap\mathfrak{t}'^{\sigma_{2}'})$
and
$\mathrm{ord}(\sigma_{1}\sigma_{2})=\mathrm{ord}(\sigma_{1}'\sigma_{2}')$.
\end{thm}

\begin{proof}
It is sufficient to show $(1)\Rightarrow (2)$
and $(2)\Leftrightarrow (3)$
because
$(2)\Rightarrow (1)$ is clear.

$(1)\Rightarrow (2)$:
Assume that
$(\Delta,\sigma_{1},\sigma_{2})\sim(\Delta',\sigma_{1}',\sigma_{2}')$.
Then there exist
an isomorphism $\varphi:\Delta\to\Delta'$
and $w'\in W(\Delta')$ satisfying 
\eqref{eqn:dsig_sim}.
Let $\Pi$
and $\Pi'$ be a $(\sigma_{1},\sigma_{2})$-fundamental system of $\Delta$
and a $(\sigma_{1}',\sigma_{2}')$-fundamental system of $\Delta'$, respectively.
It follows from Lemma \ref{lem:Pi_varphiPi} that $\varphi(\Pi)$ is a
$(\sigma_{1}',\varphi\sigma_{2}\varphi^{-1})$-fundamental system of $\Delta'$.
Then there exist $w'_{1}\in W(\Delta')_{\sigma_{1}'}$
and $w'_{2}\in W(\Delta')_{\sigma_{2}'}$
satisfying the following relations:
\[
\Pi'=w_{1}'(\varphi(\Pi)),\quad
\Pi'=w'_{2}(w'\varphi(\Pi)),
\]
from which we have
$w_{1}'(\varphi(\Pi))=w_{2}'w'(\varphi(\Pi))$.
This yields $w'=w_{2}'^{-1}w_{1}'$.
If we put $\varphi':=w_{1}'\varphi$,
then it is an isomorphism of root systems
which satisfies $\sigma_{1}'=w_{1}'\sigma_{1}'w_{1}'^{-1}=\varphi'\sigma_{1}\varphi'^{-1}$
and
\[
\sigma_{2}'
=w'_{2}\sigma_{2}'w_{2}'^{-1}
=w'_{2}(w'\varphi\sigma_{2}\varphi^{-1}w'^{-1})w_{2}'^{-1}
=w'_{2}(w_{2}'^{-1}w_{1}'\varphi\sigma_{2}\varphi^{-1}w_{1}'^{-1}w'_{2})w_{2}'^{-1}
=\varphi'\sigma_{2}\varphi'^{-1}.
\]
Hence we have the implication
$(1)\Rightarrow (2)$.

$(2)\Rightarrow (3)$:
Let $\varphi:\mathfrak{t}\to\mathfrak{t}'$
be an isomorphism of root systems between $\Delta$ and $\Delta'$ satisfying 
$\sigma_{i}'=\varphi\sigma_{i}\varphi^{-1}$ for $i=1,2$.
If $\Pi$ is a $(\sigma_{1},\sigma_{2})$-fundamental system of $\Delta$,
then $\varphi(\Pi)$ is a $(\sigma_{1}',\sigma_{2}')$-fundamental system of $\Delta'$.
This implies $(S_{1},S_{2})\sim(S_{1}',S_{2}')$.

$(3)\Rightarrow (2)$:
Let $\psi:\Pi\to\Pi'$
be an isomorphism of double Satake diagrams
between $(S_{1},S_{2})$ and $(S_{1}',S_{2}')$.
We extend $\psi$ to an isomorphism $\tilde{\psi}$ of root systems between $\Delta$ and $\Delta'$
(cf.~Lemma \ref{lem:sigma_satake}).
The
$\tilde{\psi}$ satisfies $\sigma_{i}'=\tilde{\psi}
\sigma_{i}\tilde{\psi}^{-1}$ for $i=1,2$.
Thus, we have the implication $(3)\Rightarrow (2)$.

From the above argument we have completed the proof.
\end{proof}

For
two double $\sigma$-systems
$(\Delta,\sigma_{1},\sigma_{2})$
and
$(\Delta',\sigma_{1}',\sigma_{2}')$,
we write 
$(\Delta,\sigma_{1},\sigma_{2})\equiv(\Delta',\sigma_{1}',\sigma_{2}')$
if 
they satisfies the condition
stated in Theorem \ref{thm:dsig_dsatake_equiv}, (2).
Then $\equiv$ gives
an equivalence relation on the set of double $\sigma$-systems.

We define the rank and the order
for the isomorphism class of a normal double $\sigma$-system
$(\Delta,\sigma_{1},\sigma_{2})$
as follows:
For a canonical representative $(\Delta,\sigma_{1}',\sigma_{2}')\in
[(\Delta,\sigma_{1},\sigma_{2})]$,
\[
\mathrm{rank}[(\Delta,\sigma_{1},\sigma_{2})]:=\dim(\mathfrak{t}^{\sigma_{1}'}\cap\mathfrak{t}^{\sigma_{2}'}),\quad
\mathrm{ord}[(\Delta,\sigma_{1},\sigma_{2})]:=\mathrm{ord}(\sigma_{1}'\sigma_{2}').
\]
It follows from
Theorem \ref{thm:dsig_dsatake_equiv}
that
the values
of $\dim(\mathfrak{t}^{\sigma_{1}'}\cap\mathfrak{t}^{\sigma_{2}'})$
and $\mathrm{ord}(\sigma_{1}'\sigma_{2}')$
are independent of the choice of $(\Delta,\sigma_{1}',\sigma_{2}')$.
Thus the rank and the order of $[(\Delta,\sigma,\sigma_{2})]$
are well-defined.
Since $\sigma_{1}'\sigma_{2}'$
induces a permutation of $\Delta$,
the order of $[(\Delta,\sigma_{1},\sigma_{2})]$
is finite.
As will be shown later,
in the case when $G$ is simple,
the rank and the order of $[(G,\theta_{1},\theta_{2})]$
coincide with those of $[(\Delta,\sigma_{1},\sigma_{2})]$
(see Theorem \ref{thm:cst_RO_can}).

\section{Double Satake diagrams for compact symmetric triads}\label{sec:cst_dsatake}

In Subsection \ref{sec:dsfcst},
we give a normal double $\sigma$-system for a compact symmetric triad.
In Subsection
\ref{sec:cst_qcan},
we define a quasi-canonical compact symmetric
triad as a compact symmetric triad
which admits a canonical normal double $\sigma$-system.
Furthermore, we prove that,
for any compact symmetric triad $(G,\theta_{1},\theta_{2})$,
there exists $(G,\theta_{1}',\theta_{2}')\sim
(G,\theta_{1},\theta_{2})$ such that 
$(G,\theta_{1}',\theta_{2}')$ is quasi-canonical.
In Subsection \ref{sec:cst_2satake_const},
we introduce the notion of double Satake diagrams
for quasi-canonical compact symmetric triads.
We will show
that the isomorphism class of a compact symmetric triad
uniquely determines
the double Satake diagram up to isomorphism
(Propositions
\ref{pro:exist_qcan} and \ref{pro:cst_dstake_determ}).
For its converse,
we generalize Theorem \ref{thm:cps_sigma_satake_equiv}
to compact symmetric triads,
which is given in Theorem \ref{thm:cst_dsatake_sim}.
In Subsection \ref{sec:cst_class_2satake},
we classify compact symmetric triads $(G,\theta_{1},\theta_{2})$
such that $G$ is simple in terms of double Satake diagrams.
Our classification will be given in Corollary \ref{cor:cst_classify}.
In addition, we give some results by means of the classification.

\subsection{Double $\sigma$-systems for compact symmetric triads}\label{sec:dsfcst}

\subsubsection{Construction of double $\sigma$-systems from compact symmetric triads}

Let $(G,\theta_{1},\theta_{2})$
be a compact symmetric triad
and $\mathfrak{t}$ be a maximal abelian subalgebra of $\mathfrak{g}$
 such that $\mathfrak{t}\cap\mathfrak{m}_{i}$
is a maximal abelian subspace of $\mathfrak{m}_{i}$
(cf.~Lemma \ref{lem:lemm2.5}).
Denote by $\Delta$ the root system of $\mathfrak{g}$
with respect to $\mathfrak{t}$.
Then, for each $i=1$, $2$,
$(\Delta, \sigma_{i}):=(\Delta, -d\theta_{i}|_{\mathfrak{t}})$
gives a normal $\sigma$-system of $\Delta$.
Hence $(\Delta,\sigma_{1},\sigma_{2})$
becomes a normal double $\sigma$-system.
We call
$(\Delta,\sigma_{1},\sigma_{2})$
the double $\sigma$-system of $(G,\theta_{1},\theta_{2})$
 with respect to $\mathfrak{t}$.
We will show that the double $\sigma$-system
$(\Delta,\sigma_{1},\sigma_{2})$
is uniquely determined up to isomorphism,
that is, we have the following lemma.

\begin{lem}\label{lem:dsig_isomorphic}
Let $\mathfrak{t}'$ be another maximal abelian subalgebra of $\mathfrak{g}$
such that $\mathfrak{t}'\cap\mathfrak{m}_{i}$
is a maximal abelian subspace of $\mathfrak{m}_{i}$
and $(\Delta',\sigma_{1}',\sigma_{2}')$ denote
the corresponding normal double $\sigma$-system.
Then we have $(\Delta,\sigma_{1},\sigma_{2})\sim(\Delta,\sigma_{1}',\sigma_{2}')$.
\end{lem}

\begin{proof}
By the choices of $\mathfrak{t}$ and $\mathfrak{t}'$,
there exist $\nu_{1}\in\mathrm{Int}(\mathfrak{k}_{1})$
and $\nu_{2}\in\mathrm{Int}(\mathfrak{k}_{2})$
satisfying $\nu_{1}(\mathfrak{t})=\mathfrak{t}'=\nu_{2}(\mathfrak{t})$
(cf.~\cite[Proposition 5]{Sugiura}).
In particular,
$\nu_{1}^{-1}\nu_{2}(\mathfrak{t})=\mathfrak{t}$ holds.
Thus we obtain
\begin{align*}
(\Delta',\sigma_{1}',\sigma_{2}')
&=(\nu_{1}(\Delta),-\nu_{1}d\theta_{1}\nu_{1}^{-1}|_{\nu_{1}(\mathfrak{t})},
-\nu_{1}(\nu_{1}^{-1}d\theta_{2}\nu_{1})\nu_{1}^{-1}|_{\nu_{1}(\mathfrak{t})})\\
&\sim(\Delta,-d\theta_{1}|_{\mathfrak{t}},(\nu_{1}^{-1}\nu_{2})|_{\mathfrak{t}}(-d\theta_{2}|_{\mathfrak{t}})
(\nu_{2}^{-1}\nu_{1})|_{\mathfrak{t}})\\
&\sim(\Delta,\sigma_{1},\sigma_{2}).
\end{align*}
Hence we have the assertion.
\end{proof}

\begin{lem}\label{lem:CST_dsigma2}
Let $(G,\theta_{1}',\theta_{2}')\sim (G,\theta_{1},\theta_{2})$
be another compact symmetric triad and $\mathfrak{t}'$
be a maximal abelian subalgebra of $\mathfrak{g}$
such that $\mathfrak{t}'\cap\mathfrak{m}_{i}'$ $(i=1,2)$
is a maximal abelian subspaces of $\mathfrak{m}_{i}'$.
Let $(\Delta',\sigma_{1}',\sigma_{2}')$
be the corresponding normal double $\sigma$-system of $(G,\theta_{1}',\theta_{2}')$.
Then we have $(\Delta,\sigma_{1},\sigma_{2})\sim (\Delta',\sigma_{1}',\sigma_{2}')$.
\end{lem}

\begin{proof}
It is sufficient to consider the case when $\theta_{1}'=\theta_{1}$
and $\theta_{2}'=\tau\theta_{2}\tau^{-1}$ for some $\tau\in\mathrm{Int}(G)$.
Then $\tau(\mathfrak{t})$
is $\theta_{2}'$-invariant and
$\tau(\mathfrak{t})\cap\mathfrak{m}_{2}'$
is a maximal abelian subspace of $\mathfrak{m}_{2}'$.
Furthermore,
$\tau(\Delta)$ is the root system of $\mathfrak{g}$
with respect to $\tau(\mathfrak{t})$.
If we put $\sigma_{2}''=-d\theta_{1}|_{\tau(\mathfrak{t})}$
and
$\sigma_{2}''=-d\theta_{2}'|_{\tau(\mathfrak{t})}$,
then $(\tau(\Delta),\sigma_{1}'', \sigma_{2}'')$ gives
the normal double $\sigma$-system
of $(G,\theta_{1}',\theta_{2}')$
corresponding to $\tau(\mathfrak{t})$.
By Lemma \ref{lem:dsig_isomorphic}
we have $(\Delta',\sigma_{1}',\sigma_{2}')\sim(\tau(\Delta),\sigma_{1}'',\sigma_{2}'')$.
On the other hand,
we have $(\tau(\Delta),\sigma_{1}'',\sigma_{2}'')\sim(\Delta,\sigma_{1},\sigma_{2})$.
Indeed,
there exists $\nu_{1}\in\mathrm{Int}(\mathfrak{k}_{1})$
such that $\nu_{1}(\mathfrak{t})=\tau(\mathfrak{t})$,
from which $\tau^{-1}\nu_{1}(\mathfrak{t})=\mathfrak{t}$ holds.
Hence we have
\begin{equation*}
(\tau(\Delta),\sigma_{1}'',\sigma_{2}'')
\equiv
(\Delta,-\tau^{-1}\theta_{1}\tau|_{\mathfrak{t}},-d\theta_{2}|_{\mathfrak{t}})
\sim(\Delta,(\tau^{-1}\nu_{1})|_{\mathfrak{t}}(-d\theta_{1}|_{\mathfrak{t}})(\nu_{1}^{-1}\tau)|_{\mathfrak{t}},-d\theta_{2}|_{\mathfrak{t}}).
\end{equation*}
We have complete the proof.
\end{proof}

\subsubsection{Another interpretation of the rank for compact symmetric triads}
As shown in Section \ref{sec:cst_fundamental},
the rank of a compact symmetric triad $(G,\theta_{1},\theta_{2})$
coincides with the cohomogeneity of the Hermann action induced from $(G,\theta_{1},\theta_{2})$.
We give another interpretation of the rank
in terms of the double $\sigma$-system of $(G,\theta_{1},\theta_{2})$.
More precisely, we prove the following proposition.

\begin{pro}\label{pro:rank=maxdim}
Let $\mathfrak{t}$ be a maximal abelian subalgebra of $\mathfrak{g}$
such that $\mathfrak{t}\cap\mathfrak{m}_{i}$
$(i=1,2)$ is a maximal abelian subspace
of $\mathfrak{m}_{i}$,
and $(\Delta,\sigma_{1},\sigma_{2}):=(\Delta,-d\theta_{1}|_{\mathfrak{t}},-d\theta_{2}|_{\mathfrak{t}})$.
Then we have$:$
\[
\mathrm{rank}(G,\theta_{1},\theta_{2})
=\max\{\dim(\mathfrak{t}^{\sigma_{1}}\cap s\mathfrak{t}^{\sigma_{2}})\mid s\in W(\Delta)\}.
\]
\end{pro}

\begin{proof}
First, we prove 
\begin{equation}\label{eqn:rankgeqmax}
\mathrm{rank}(G,\theta_{1},\theta_{2})
\geq \max\{\dim(\mathfrak{t}^{\sigma_{1}}\cap s\mathfrak{t}^{\sigma_{2}})\mid s\in W(\Delta)\}.
\end{equation}
Let $s$ be in $W(\Delta)$
and $g$ be an element of $G$ with $\mathrm{Ad}(g)|_{\mathfrak{t}}=s$.
If we put $\theta_{2}'=\tau_{g}\theta_{2}\tau_{g}^{-1}$,
then $(G,\theta_{1},\theta_{2}')$
is a compact symmetric triad which is isomorphic to $(G,\theta_{1},\theta_{2})$.
Furthermore, we find that $\mathfrak{t}^{\sigma_{1}}\cap s\mathfrak{t}^{\sigma_{2}}=\mathfrak{t}\cap(\mathfrak{m}_{1}\cap\mathfrak{m}_{2}')$
is an abelian subspace of $\mathfrak{m}_{1}\cap\mathfrak{m}_{2}'$.
Hence we have
\[
\mathrm{rank}(G,\theta_{1},\theta_{2})
=\mathrm{rank}(G,\theta_{1},\theta_{2}')
\geq \dim(\mathfrak{t}^{\sigma_{1}}\cap s\mathfrak{t}^{\sigma_{2}}).
\]
By the arbitrariness of $s$,
this yields \eqref{eqn:rankgeqmax}.

Next, we show the reverse inequality of \eqref{eqn:rankgeqmax}.
It follows from Lemma \ref{lem:lemm2.5}
that there exist a compact symmetric triad
$(G,\theta_{1}',\theta_{2}')\sim(G,\theta_{1},\theta_{2})$
and a maximal abelian subalgebra $\mathfrak{t}'$ such that
$\mathfrak{t}'\cap\mathfrak{m}_{i}'$
($i=1,2$)
is a maximal abelian subspace of $\mathfrak{m}_{i}'$,
and that $\mathfrak{t}'\cap(\mathfrak{m}_{1}'\cap\mathfrak{m}_{2}')$
is a maximal abelian subspace of $\mathfrak{m}_{1}'\cap\mathfrak{m}_{2}'$.
We write $(\Delta',\sigma_{1}',\sigma_{2}'):=(\Delta',-\theta_{1}'|_{\mathfrak{t}'},-\theta'_{2}|_{\mathfrak{t}'})$
as the double $\sigma$-system of $(G,\theta_{1}',\theta_{2}')$.
From Lemma \ref{lem:dsig_isomorphic},
$(G,\theta_{1},\theta_{2})\sim(G,\theta_{1}',\theta_{2}')$
yields $(\Delta,\sigma_{1},\sigma_{2})\sim(\Delta',\sigma_{1}',\sigma_{2}')$.
Then
there exist an isomorphism $\varphi:\Delta\to \Delta'$ of root systems and $s\in W(\Delta)$
satisfying
$\sigma_{1}'=\varphi\sigma_{1}\varphi^{-1}$
and $\sigma_{2}'=\varphi s\sigma_{2}s^{-1} \varphi^{-1}$,
from which we get
\[
\dim(\mathfrak{t}^{\sigma_{1}'}\cap\mathfrak{t}^{\sigma_{2}'})
=\dim(\varphi(\mathfrak{t}^{\sigma_{1}})\cap \varphi(s\mathfrak{t}^{\sigma_{2}}))
=\dim(\mathfrak{t}^{\sigma_{1}}\cap s\mathfrak{t}^{\sigma_{2}}).
\]
Hence we obtain
\[
\mathrm{rank}(G,\theta_{1},\theta_{2})
=\dim(\mathfrak{t}^{\sigma_{1}'}\cap\mathfrak{t}^{\sigma_{2}'})
\leq \max\{\dim(\mathfrak{t}^{\sigma_{1}}\cap s\mathfrak{t}^{\sigma_{2}})\mid s\in W(\Delta)\}.
\]
From the above we have complete the proof.
\end{proof}

\subsection{Quasi-canonical forms in compact symmetric triads}\label{sec:cst_qcan}

\subsubsection{Definition and existence for quasi-canonical compact symmetric triads}

Let us introduce the notion
of a quasi-canonical compact symmetric
triad as follows.

\begin{dfn}\label{dfn:cst_qcf}
A compact symmetric triad 
$(G,\theta_{1},\theta_{2})$
is said to be \textit{quasi-canonical},
if there exists a maximal abelian subalgebra $\mathfrak{t}$
of $\mathfrak{g}$ which satisfies the following conditions:
\begin{enumerate}
\item $\mathfrak{t}\cap\mathfrak{m}_{i}$
is a maximal abelian subspace of $\mathfrak{m}_{i}$ for $i=1,2$.
\item The normal double $\sigma$-system
$(\Delta,\sigma_{1},\sigma_{2}):=(\Delta,-d\theta_{1}|_{\mathfrak{t}},-d\theta_{2}|_{\mathfrak{t}})$ is canonical,
that is, there exists a $(\sigma_{1},\sigma_{2})$-fundamental system of $\Delta$.
\end{enumerate}
Then, $\mathfrak{t}$ is said to be quasi-canonical with respect to 
$(G,\theta_{1},\theta_{2})$.
A \textit{quasi-canonical form} of $[(G,\theta_{1},\theta_{2})]$
is a representative $(G,\theta_{1}',\theta_{2}')$
of the isomorphism class $[(G,\theta_{1},\theta_{2})]$ such that
$(G,\theta_{1}',\theta_{2}')$
is quasi-canonical as a compact symmetric triad.
\end{dfn}

\begin{pro}\label{pro:exist_qcan}
For a compact symmetric triad $(G,\theta_{1},\theta_{2})$,
there exists a quasi-canonical
compact symmetric triad $(G,\theta_{1},\theta_{2}')\sim(G,\theta_{1},\theta_{2})$.
\end{pro}

\begin{proof}
Let $(G,\theta_{1},\theta_{2})$
be a compact symmetric triad
and $\mathfrak{t}$ be a maximal abelian subalgebra of $\mathfrak{g}$
such that $\mathfrak{t}\cap\mathfrak{m}_{i}$
is a maximal abelian subspace of $\mathfrak{m}_{i}$.
Denote by $(\Delta,\sigma_{1},\sigma_{2})$
the corresponding normal double $\sigma$-system of $(G,\theta_{1},\theta_{2})$.
Let $\Pi_{i}$ be a $\sigma_{i}$-fundamental system of $\Delta$.
Since $N(\mathfrak{t})$
acts transitively on the set of fundamental systems of $\Delta$,
there exists $g\in N(\mathfrak{t})$
satisfying $\Pi_{1}=\mathrm{Ad}(g)(\Pi_{2})$.
If we put $\theta_{2}':=\tau_{g}\theta_{2}\tau_{g}^{-1}$,
then it is verified that $(G,\theta_{1},\theta_{2}')\sim(G,\theta_{1},\theta_{2})$
is quasi-canonical.
\end{proof}

In Section \ref{sec:cst_cf},
we will define the notion of a canonicality for compact symmetric triads,
which is a stronger condition than 
the quasi-canonicality (see Definition \ref{dfn:cst_can}).
Furthermore, in the case when $G$ is simple,
we will prove that the existence of a representative of $[(G,\theta_{1},\theta_{2})]$
which is canonical as a compact symmetric triad (see Theorem \ref{thm:cst_exist_can}).

\subsubsection{Commutative compact symmetric triads are quasi-canonical}

The following proposition
means that a quasi-canonical compact symmetric triad
is a generalization of a commutative one.

\begin{pro}\label{pro:ccst_qcan}
Any commutative compact symmetric triad is quasi-canonical.
\end{pro}

The proof of this proposition consists of the following three lemmas,
which are essentially due to Oshima-Sekiguchi (\cite{OS}).
Roughly speaking,
the first lemma states
that Lemma \ref{lem:lemm2.5}
holds without changing representatives
of $[(G,\theta_{1},\theta_{2})]$
in the case when
$(G,\theta_{1},\theta_{2})$
is commutative.

\begin{lem}\label{lem:comm_quasican}
Assume that $(G,\theta_{1},\theta_{2})$ is commutative.
Then there exists a maximal abelian subalgebra $\mathfrak{t}$ of $\mathfrak{g}$
such that $\mathfrak{t}\cap\mathfrak{m}_{i}$ and
$\mathfrak{t}\cap(\mathfrak{m}_{1}\cap\mathfrak{m}_{2})$
are maximal abelian subspaces of $\mathfrak{m}_{i}$ and $\mathfrak{m}_{1}\cap\mathfrak{m}_{2}$,
respectively.
In particular, $(G,\theta_{1},\theta_{2})$ satisfies the condition {\rm (1)}
as in Definition {\rm \ref{dfn:cst_qcf}}.
\end{lem}

\begin{proof}
From $\theta_{1}\theta_{2}=\theta_{2}\theta_{1}$
we have
$\mathfrak{g}
=(\mathfrak{k}_{1}\cap\mathfrak{k}_{2})\oplus(\mathfrak{m}_{1}\cap\mathfrak{m}_{2})\oplus
(\mathfrak{k}_{1}\cap\mathfrak{m}_{2})\oplus(\mathfrak{m}_{1}\cap\mathfrak{k}_{2})$.
Let $\mathfrak{a}$ be a maximal abelian subspace of $\mathfrak{m}_{1}\cap\mathfrak{m}_{2}$.
Let $\mathfrak{a}_{i}$ be a maximal abelian subspace of $\mathfrak{m}_{i}$ containing $\mathfrak{a}$.
In a similar argument in the proofs of \cite[Lemmas (2.2) and (2.4)]{OS},
it is shown that
$\mathfrak{a}_{1}$
and $\mathfrak{a}_{2}$ are $(\theta_{1},\theta_{2})$-invariant and that $[\mathfrak{a}_{1},\mathfrak{a}_{2}]=\{0\}$.
In particular, $\mathfrak{a}_{1}+\mathfrak{a}_{2}$ is an abelian subalgebra of $\mathfrak{g}$.
Let $\mathfrak{t}$ be a maximal abelian subalgebra of $\mathfrak{g}$ containing
$\mathfrak{a}_{1}+\mathfrak{a}_{2}$.
Since $\mathfrak{t}$ contains $\mathfrak{a}_{1}$ and $\mathfrak{a}_{2}$,
it is shown that $\mathfrak{t}$ is $(\theta_{1},\theta_{2})$-invariant.
We also obtain $\mathfrak{t}\cap\mathfrak{m}_{i}=\mathfrak{a}_{i}$ and
$\mathfrak{t}\cap(\mathfrak{m}_{1}\cap\mathfrak{m}_{2})=\mathfrak{a}$.
Hence we get the assertion.
\end{proof}

\begin{lem}\label{lem:a=0a_i=0}
Assume that $(G,\theta_{1},\theta_{2})$ is commutative.
Let $\mathfrak{t}$ be a maximal abelian subalgebra of $\mathfrak{g}$
which satisfies the condition stated in Lemma $\ref{lem:comm_quasican}$.
Set $\mathfrak{a}_{2}:=\mathfrak{t}\cap\mathfrak{m}_{2}$ and $\mathfrak{a}:=\mathfrak{t}\cap(\mathfrak{m}_{1}\cap\mathfrak{m}_{2})$.
Then, we have the followings$:$
\begin{enumerate}
\item We denote by $\Sigma_{2}$ the restricted root system of $(G,\theta_{2})$
with respect to $\mathfrak{a}_{2}$.
For $\lambda\in\Sigma_{2}$ with $\INN{\lambda}{\mathfrak{a}}=\{0\}$,
we have $\mathfrak{g}(\mathfrak{a}_{2},\lambda)\subset\mathfrak{k}_{1}^{\mathbb{C}}$,
where
\[
\mathfrak{g}(\mathfrak{a}_{2},\lambda)
:=\{X\in\mathfrak{g}^{\mathbb{C}}\mid
[H,X]=\sqrt{-1}\INN{\lambda}{H}X,H\in\mathfrak{a}_{2}\}.
\]

\item We denote by $\Delta$ the root system of $\mathfrak{g}$ with respect to $\mathfrak{t}$.
For $\alpha\in\Delta$ with $\INN{\alpha}{\mathfrak{a}}=\{0\}$,
if $\INN{\alpha}{\mathfrak{a}_{2}}\neq\{0\}$ holds, then
we obtain $\INN{\alpha}{\mathfrak{a}_{1}}=\{0\}$.
\end{enumerate}
\end{lem}

We omit its proof since one can prove this lemma by a similar argument in the proofs of
\cite[Lemmas (2.7) and (2.8)]{OS}.

\begin{lem}\label{lem:ccst_t_std}
Retain the notations
$(G,\theta_{1},\theta_{2})$
and $\mathfrak{t}$
as in Lemma $\ref{lem:comm_quasican}$.
Let $\Delta$ be the root system of $\mathfrak{g}$
with respect to $\mathfrak{t}$,
and $\sigma_{i}:=-d\theta_{i}|_{\mathfrak{t}}$ for $i=1,2$.
Then there exists a $(\sigma_{1},\sigma_{2})$-fundamental system
of $\Delta$.
Hence $(G,\theta_{1},\theta_{2})$ satisfies the condition {\rm (2)}
as in Definition {\rm \ref{dfn:cst_qcf}}.
\end{lem}

\begin{proof}
Let $\mathfrak{a}_{i}:=\mathfrak{t}\cap\mathfrak{m}_{i}$
($i=1,2$),
and $\mathfrak{a}:=\mathfrak{t}\cap(\mathfrak{m}_{1}\cap\mathfrak{m}_{2})$.
Then $\mathfrak{t}$ is decomposed into
$\mathfrak{t}
=\mathfrak{a}
\oplus
(\mathfrak{a}_{1}\cap\mathfrak{k}_{2})
\oplus
(\mathfrak{a}_{2}\cap\mathfrak{k}_{1})
\oplus
(\mathfrak{t}\cap(\mathfrak{k}_{1}\cap\mathfrak{k}_{2}))$.
Take a ordered basis of $\{X_{j}, Y_{k}, Z_{l}, W_{r}\}$
such that $\{X_{j}\}$,
$\{Y_{k}\}$, $\{Z_{l}\}$
and $\{W_{r}\}$
are bases of $\mathfrak{a}$,
$\mathfrak{a}_{1}\cap\mathfrak{k}_{2}$,
$\mathfrak{a}_{2}\cap\mathfrak{k}_{1}$ and
$\mathfrak{t}\cap(\mathfrak{k}_{1}\cap\mathfrak{k}_{2})$,
respectively.
We denote by $\Delta^{+}$
the set of positive roots in $\Delta$
with respect to
the lexicographic ordering $>$ of $\mathfrak{t}$
with respect to this basis.
We obtain a fundamental system $\Pi$ of $\Delta$
such that $\Delta^{+}=\{\sum_{\alpha\in\Pi}m_{\alpha}\alpha\in\Delta\mid
m_{\alpha}\in\mathbb{Z}_{\geq0}\}$.
A similar argument as in
\cite[p.~453]{OS}
shows that
$\Pi$
becomes a $(\sigma_{1},\sigma_{2})$-fundamental system of $\Delta$
by means of Lemma \ref{lem:a=0a_i=0}.
\end{proof}

From the above argument
we conclude that 
Proposition \ref{pro:ccst_qcan}
holds.

\subsection{Double Satake diagrams for compact symmetric triads}\label{sec:cst_2satake_const}

Let us explain
our construction of 
the double Satake diagram from a quasi-canonical compact symmetric triad.
Let $(G,\theta_{1},\theta_{2})$
be a quasi-canonical compact symmetric triad
and $\mathfrak{t}$
be a quasi-canonical maximal abelian subalgebra of
$\mathfrak{g}$ with respect to $(G,\theta_{1},\theta_{2})$.
Denote by $\Delta$ the root system of $\mathfrak{g}$
with respect to $\mathfrak{t}$.
Then we obtain a normal double $\sigma$-system $(\Delta,\sigma_{1},\sigma_{2}):=(\Delta,-d\theta_{1}|_{\mathfrak{t}},-d\theta_{2}|_{\mathfrak{t}})$.
It follows from 
the quasi-canonicality of $(G,\theta_{1},\theta_{2})$
that $(\Delta,\sigma_{1},\sigma_{2})$ becomes canonical
in the sense of Definition \ref{dfn:dsig_canonical}.
We define the double Satake diagram of $(G,\theta_{1},\theta_{2})$
as that of $(\Delta,\sigma_{1},\sigma_{2})$.
Then
the definition of the double Satake diagram of $(G,\theta_{1},\theta_{2})$
is independent of the choice of $\mathfrak{t}$.
Indeed, we can show the following lemma.

\begin{lem}
Let $\mathfrak{t}'$ be another
quasi-canonical maximal abelian subalgebra of
$\mathfrak{g}$ with respect to $(G,\theta_{1},\theta_{2})$.
We denote by $(\Delta',\sigma_{1}',\sigma_{2}'):=(\Delta',-d\theta_{1}|_{\mathfrak{t}'},-d\theta_{2}|_{\mathfrak{t}'})$
the corresponding canonical normal double $\sigma$-system of $(G,\theta_{1},\theta_{2})$.
Then we have
$(\Delta,\sigma_{1},\sigma_{2})\equiv(\Delta',\sigma_{1}',\sigma_{2}')$.
Therefore
the double Satake diagram
of $(\Delta,\sigma_{1},\sigma_{2})$
is isomorphic to that of $(\Delta',\sigma_{1}',\sigma_{2}')$.
\end{lem}

The proof is omitted since
it is immediate from
Theorem \ref{thm:dsig_dsatake_equiv}
and Lemma \ref{lem:dsig_isomorphic}.

We next show
that the double Satake diagram
of a quasi-canonical compact symmetric triad
is independent of the choice
of the representative
of its isomorphism class,
namely,
we have the following proposition,
which is immediate from
Theorem \ref{thm:dsig_dsatake_equiv}
and Lemma \ref{lem:CST_dsigma2}.

\begin{pro}\label{pro:cst_dstake_determ}
Let $(G,\theta_{1}',\theta_{2}')\sim(G,\theta_{1},\theta_{2})$
be another quasi-canonical compact symmetric
triad and $(\Delta',\sigma_{1}',\sigma_{2}')$
be the corresponding canonical normal double $\sigma$-system
of $(G,\theta_{1}',\theta_{2}')$.
Then we have
$(\Delta,\sigma_{1},\sigma_{2})\equiv(\Delta',\sigma_{1}',\sigma_{2}')$.
\end{pro}

It follows
from Propositions
\ref{pro:exist_qcan}
and
\ref{pro:cst_dstake_determ}
that,
for a compact symmetric triad
$(G,\theta_{1},\theta_{2})$,
its isomorphism class $[(G,\theta_{1},\theta_{2})]$
uniquely determines
the double Satake diagram up to isomorphism.
In fact,
the converse also holds as shown in the following theorem.

\begin{thm}\label{thm:cst_dsatake_sim}
Let $(G,\theta_{1},\theta_{2})$ and 
$(G,\theta_{1}',\theta_{2}')$ be two compact symmetric triads.
We write
$(\Delta,\sigma_{1},\sigma_{2})$ and
$(\Delta',\sigma_{1}',\sigma_{2}')$
as
the corresponding canonical normal double $\sigma$-systems
of the isomorphism classes $[(G,\theta_{1},\theta_{2})]$
and $[(G,\theta_{1}',\theta_{2}')]$, respectively.
We also write
$(S_{1},S_{2})$ and
$(S_{1}',S_{2}')$
as the double Satake diagrams
of $(\Delta,\sigma_{1},\sigma_{2})$ and
$(\Delta',\sigma_{1}',\sigma_{2}')$, respectively.
Then the following conditions are equivalent$:$
\begin{enumerate}
\item $(G,\theta_{1},\theta_{2})$
and $(G,\theta_{1}',\theta_{2}')$
are locally isomorphic, namely,
there exist $\varphi\in\mathrm{Aut}(\mathfrak{g})$
and $\tau\in\mathrm{Int}(\mathfrak{g})$
satisfying $d\theta_{1}'=\varphi d\theta_{1}\varphi^{-1}$
and $d\theta_{2}'=\tau\varphi d\theta_{2}\varphi^{-1}\tau^{-1}$.
\item $(\Delta,\sigma_{1},\sigma_{2})\sim(\Delta',\sigma_{1}',\sigma_{2}')$.
\item $(S_{1},S_{2})\sim(S_{1}',S_{2}')$.
\end{enumerate}
In addition, in the case when $G$ is simply-connected
or when $G$ is the adjoint group,
$(G,\theta_{1},\theta_{2})$ and $(G,\theta_{1}',\theta_{2}')$
are isomorphic if and only if
one of the above conditions $(1)$--$(3)$ holds.
\end{thm}

\begin{proof}
The implication
$(1)\Rightarrow(2)$
follows from Propositions
\ref{pro:exist_qcan}
and
\ref{pro:cst_dstake_determ}.
We obtain
$(2)\Leftrightarrow(3)$
from Theorem \ref{thm:dsig_dsatake_equiv}.

We prove the implication $(2)\Rightarrow(1)$.
Without loss of generalities
we may assume that $(G,\theta_{1},\theta_{2})$
and $(G,\theta_{1}',\theta_{2}')$
are quasi-canonical.
We write $(\Delta,\sigma_{1},\sigma_{2})=(\Delta,-d\theta_{1}|_{\mathfrak{t}},-d\theta_{2}|_{\mathfrak{t}})$
and $(\Delta',\sigma_{1}',\sigma_{2}')=(\Delta',-d\theta'_{1}|_{\mathfrak{t}'},-d\theta_{2}'|_{\mathfrak{t}'})$.
It follows from
Theorem \ref{thm:dsig_dsatake_equiv}
that there exists 
an isomorphism $\varphi:\mathfrak{t}\to\mathfrak{t}'$
of root systems between $\Delta$ and $\Delta'$ satisfying
$\sigma_{i}'=\varphi\sigma_{i}\varphi^{-1}$ for $i=1,2$.

Let $\tilde{\varphi}$
be an automorphism of $\mathfrak{g}$ with $\tilde{\varphi}|_{\mathfrak{t}}=\varphi$.
Since $d\theta_{1}'|_{\mathfrak{t}'}=\tilde{\varphi}d\theta_{1}\tilde{\varphi}^{-1}|_{\mathfrak{t}'}$ holds,
it follows from Theorem \ref{thm:araki} that there exists $H_{1}'\in\mathfrak{t}'\cap\mathfrak{m}_{1}'$ such that
\begin{equation}\label{eqn:theta1'theta1}
d\theta_{1}'=e^{\mathrm{ad}(H_{1}')}\tilde{\varphi}d\theta_{1}\tilde{\varphi}^{-1}e^{-\mathrm{ad}(H_{1}')}.
\end{equation}
In addition,
from 
$d\theta_{2}'|_{\mathfrak{t}'}=e^{\mathrm{ad}(H_{1}')}\tilde{\varphi}d\theta_{2}\tilde{\varphi}^{-1}e^{-\mathrm{ad}(H_{1}')}|_{\mathfrak{t}'}$
there also exists $H_{2}'\in\mathfrak{t}'\cap\mathfrak{m}_{2}'$ such that
\begin{equation}\label{eqn:theta2'theta2}
d\theta_{2}'=e^{\mathrm{ad}(H_{2}')}e^{\mathrm{ad}(H_{1}')}\tilde{\varphi}d\theta_{2}\tilde{\varphi}^{-1}e^{-\mathrm{ad}(H_{1}')}e^{-\mathrm{ad}(H_{2}')}.
\end{equation}
By combining  (\ref{eqn:theta1'theta1}) and 
(\ref{eqn:theta2'theta2}),
we find that
$(G,\theta_{1},\theta_{2})$
and $(G,\theta_{1}',\theta_{2}')$
are locally isomorphic.
Therefore we have completed the proof.
\end{proof}

\subsection{The classification of compact symmetric triads by double Satake diagrams}\label{sec:cst_class_2satake}

In this subsection,
we consider the classification problem
for
compact symmetric triads at the Lie algebra level.

\subsubsection{Reduction of the problem}

A compact symmetric triad $(\mathfrak{g},\theta_{1},\theta_{2})$
is said to be \textit{irreducible},
if it does not admit non-trivial $(\theta_{1},\theta_{2})$-invariant ideals of $\mathfrak{g}$ (cf.~\cite[p.~48]{Matsuki}).
Any compact symmetric triad
$(\mathfrak{g},\theta_{1},\theta_{2})$
is decomposed into irreducible ones,
namely,
there exist
unique irreducible compact symmetric triads
$(\mathfrak{g}^{(1)},\theta_{1}^{(1)},\theta_{2}^{(1)})$,
$\dotsc$,
$(\mathfrak{g}^{(k)},\theta_{1}^{(k)},\theta_{2}^{(k)})$
such that
$\mathfrak{g}=\mathfrak{g}^{(1)}
\oplus\dotsb\oplus\mathfrak{g}^{(k)}$
and that $\theta_{i}=\theta_{i}^{(j)}$
holds on $\mathfrak{g}^{(j)}$
for $i=1,2$
and
$j=1,\dotsc,k$.
Then we write
\[
(\mathfrak{g},\theta_{1},\theta_{2})
=(\mathfrak{g}^{(1)},\theta_{1}^{(1)},\theta_{2}^{(1)})\oplus
\dotsb\oplus(\mathfrak{g}^{(k)},\theta_{1}^{(k)},\theta_{2}^{(k)}).
\]
This decomposition
is called the irreducible decomposition of $(\mathfrak{g},\theta_{1},\theta_{2})$.
The equivalence relation $\sim$
is compatible with the irreducibility of a compact symmetric triad,
that is,
if $(\mathfrak{g},\theta_{1},\theta_{2})\sim(\mathfrak{g},\theta_{1}',\theta_{2}')$
and $(\mathfrak{g},\theta_{1},\theta_{2})$
is irreducible,
then $(\mathfrak{g},\theta_{1}',\theta_{2}')$
is also irreducible.
This means that
the classification problem for
compact symmetric triads reduces to
that for irreducible ones.
Clearly,
$(\mathfrak{g},\theta_{1},\theta_{2})$
is irreducible if $\mathfrak{g}$ is simple.
Irreducible compact symmetric triads
$(\mathfrak{g},\theta_{1},\theta_{2})$
can be classified depending on
whether $\mathfrak{g}$ is simple or not.
In the present paper,
we only deal with the classification problem
for compact symmetric triads $(\mathfrak{g},\theta_{1},\theta_{2})$
such that $\mathfrak{g}$ is simple.

Let $\mathfrak{g}$ be any fixed compact simple Lie algebra.
We write $\mathrm{Inv}(\mathfrak{g})$
as the set of all the involutions on $\mathfrak{g}$.
We will explain our strategy to find
all elements
of
the set $\mathcal{T}(\mathfrak{g}):=\{[(\mathfrak{g},\theta_{1},\theta_{2})]
\mid \theta_{1},\theta_{2}\in\mathrm{Inv}(\mathfrak{g})\}$.
Denote by $\mathrm{Inv}(\mathfrak{g})/\mathrm{Aut}(\mathfrak{g})$
the set of conjugacy classes in $\mathrm{Aut}(\mathfrak{g})$
of the elements in $\mathrm{Inv}(\mathfrak{g})$.
Let $[\theta_{i}]$ be in $\mathrm{Inv}(\mathfrak{g})/\mathrm{Aut}(\mathfrak{g})$
for $i=1,2$,
and $\mathfrak{k}_{i}$
denote the fixed point subalgebra of $\theta_{i}$ in $\mathfrak{g}$.
We set
\[\mathcal{T}(\mathfrak{g},\mathfrak{k}_{1},\mathfrak{k}_{2})
:=
\{
[(\mathfrak{g},\varphi_{1}\theta_{1}\varphi_{1}^{-1},\varphi_{2}\theta_{2}\varphi_{2}^{-1})]
\mid
\varphi_{1},\varphi_{2}\in\mathrm{Aut}(\mathfrak{g})
\}.\]
Then
$\mathcal{T}(\mathfrak{g})$
has the following decomposition:
\[
\mathcal{T}(\mathfrak{g})
=\bigcup_{[\theta_{1}],[\theta_{2}]\in\mathrm{Inv}(\mathfrak{g})/\mathrm{Aut}(\mathfrak{g})}
\mathcal{T}(\mathfrak{g}, \mathfrak{k}_{1}, \mathfrak{k}_{2})\quad
(\text{disjoint union}).
\]
Thus,
it is sufficient to
determine $\mathcal{T}(\mathfrak{g},\mathfrak{k}_{1},\mathfrak{k}_{2})$
for each $[\theta_{1}],[\theta_{2}]\in\mathrm{Inv}(\mathfrak{g})/\mathrm{Aut}(\mathfrak{g})$.

For this purpose
we make use of
the classification of $\mathrm{Inv}(\mathfrak{g})/\mathrm{Aut}(\mathfrak{g})$
and
a one-to-one correspondence between
$\mathcal{T}(\mathfrak{g}, \mathfrak{k}_{1}, \mathfrak{k}_{2})$
and
the set $\mathcal{DS}(S_{1},S_{2})$
which is defined as follows:
We may assume that $(\mathfrak{g},\theta_{1},\theta_{2})$
is quasi-canonical (cf.~Proposition \ref{pro:exist_qcan}).
Let $(\Delta,\sigma_{1},\sigma_{2})$
be the double $\sigma$-system of $(\mathfrak{g},\theta_{1},\theta_{2})$
and $\Pi$ be a $(\sigma_{1},\sigma_{2})$-fundamental system of $\Delta$.
We write $(S_{1},S_{2})=(S(\Pi,\Pi_{1,0},p_{1}),S(\Pi,\Pi_{2,0},p_{2}))$
as the double Satake diagram associated with $\Pi$.
We define
\begin{align}
\mathcal{DS}(S_{1},S_{2})
&=\{[(\psi_{1}\cdot S_{1},\psi_{2}\cdot S_{2})]\mid
\psi_{1},\psi_{2}\in\mathrm{Aut}(\Pi)\}\notag\\
&=\{[(S_{1},\psi\cdot S_{2})]\mid
\psi\in\mathrm{Aut}(\Pi)\},\label{eqn:DSS1S2}
\end{align}
where
$\psi_{i}\cdot S_{i}$
is the Satake diagram
defined by
$\psi_{i}\cdot S_{i}= S(\Pi,\psi_{i}(\Pi_{i,0}), \psi_{i}\cdot p_{i})$
with $\psi_{i}\cdot p_{i}=\psi_{i} p_{i} \psi_{i}^{-1}|_{\Pi-\psi_{i}(\Pi_{i,0})}$.
Here, we describe the one-to-one correspondence between
$\mathcal{T}(\mathfrak{g}, \mathfrak{k}_{1}, \mathfrak{k}_{2})$
and $\mathcal{DS}(S_{1},S_{2})$.
Let $[(\mathfrak{g},\theta_{1}',\theta_{2}')]$ be
in $\mathcal{T}(\mathfrak{g},\mathfrak{k}_{1},\mathfrak{k}_{2})$
such that $(\mathfrak{g},\theta_{1}',\theta_{2}')$ is quasi-canonical.
Denote by $(S_{1}',S_{2}')$ the double Satake diagram of $(\mathfrak{g},\theta_{1}',\theta_{2}')$.
By Theorem \ref{thm:cps_sigma_satake_equiv}
it follows from
$(\mathfrak{g},\theta_{i})\simeq(\mathfrak{g},\theta_{i}')$
that there exists an isomorphism $\psi_{i}:S_{i}\to S_{i}'$ of Satake diagrams.
Then we have
$(S_{1}',S_{2}')
=(\psi_{1}\cdot S_{1},\psi_{2}\cdot S_{2})
\sim (S_{1},\psi_{1}^{-1}\psi_{2}\cdot S_{2})$,
so that $[(S_{1}',S_{2}')]=[(S_{1},\psi_{1}^{-1}\psi_{2}\cdot S_{2})]$
is in $\mathcal{DS}(S_{1},S_{2})$ from \eqref{eqn:DSS1S2}.
Furthermore, 
it can be shown that
the following correspondence is well-defined
in terms of Theorem \ref{thm:cst_dsatake_sim}, $(1) \Rightarrow (3)$:
\begin{equation}\label{eqn:corrTtoDS}
\mathcal{T}(\mathfrak{g},\mathfrak{k}_{1},\mathfrak{k}_{2})\to
\mathcal{DS}(S_{1},S_{2});~
[(\mathfrak{g},\theta_{1}',\theta_{2}')]\mapsto
[(S_{1}',S_{2}')].
\end{equation}

\begin{lem}\label{lem:cstdsatabij}
The correspondence \eqref{eqn:corrTtoDS} is bijective.
\end{lem}

\begin{proof}
We first prove that \eqref{eqn:corrTtoDS} is injective.
Let $[(\mathfrak{g},\theta_{1}',\theta_{2}')], [(\mathfrak{g},\theta_{1}'',\theta_{2}'')]$ be in $\mathcal{T}(\mathfrak{g},\mathfrak{k}_{1},\mathfrak{k}_{2})$
such that
$(\mathfrak{g},\theta_{1}',\theta_{2}')$
and $(\mathfrak{g},\theta_{1}'',\theta_{2}'')$
are quasi-canonical.
We write $(S_{1}',S_{2}')$
and $(S_{1}'',S_{2}'')$ as the double Satake diagrams
of $(\mathfrak{g},\theta_{1}',\theta_{2}')$
and $(\mathfrak{g},\theta_{1}'',\theta_{2}'')$, respectively.
If $[(S_{1}',S_{2}')]=[(S_{1}'',S_{2}'')]$ holds,
then we obtain
$[(\mathfrak{g},\theta_{1}',\theta_{2}')]
=[(\mathfrak{g},\theta_{1}'',\theta_{2}'')]$
from Theorem \ref{thm:cst_dsatake_sim}, $(3) \Rightarrow (1)$.

Next, we prove that \eqref{eqn:corrTtoDS} is surjective.
Let $[(S_{1}',S_{2}')]$ be in $\mathcal{DS}(S_{1},S_{2})$.
Then, for each $i=1,2$,
there exists $\psi_{i}\in\mathrm{Aut}(\Pi)$
satisfying $S_{i}'=\psi_{i}\cdot S_{i}$.
Let $\varphi_{i}$
be an automorphism of $\mathfrak{g}$
such that $\varphi_{i}|_{\Pi}=\psi_{i}$ holds.
Then $(\mathfrak{g},\varphi_{1}\theta_{1}\varphi_{1}^{-1},\varphi_{2}\theta_{2}\varphi_{2}^{-1})$
gives a compact symmetric triad.
Let $(\Delta,\sigma_{1}',\sigma_{2}')$ be its double $\sigma$-system.
Since $\Pi$ becomes a $(\sigma_{1}',\sigma_{2}')$-fundamental system,
$(\mathfrak{g},\varphi_{1}\theta_{1}\varphi_{1}^{-1},\varphi_{2}\theta_{2}\varphi_{2}^{-1})$ is quasi-canonical
and its double Satake diagram coincides with $(S_{1}',S_{2}')$.
Thus we have complete the proof.
\end{proof}

\subsubsection{The classification}
Under the above argument
we first determine $\mathcal{DS}(S_{1},S_{2})$
for $[\theta_{1}],[\theta_{2}]\in\mathrm{Inv}(\mathfrak{g})/\mathrm{Aut}(\mathfrak{g})$.
This can be easily obtained by means of
the structure of $\mathrm{Aut}(\Pi)$
and the table of Satake diagrams of compact symmetric pairs
(cf.~\cite[TABLE VI]{Helgason}).
Our determination will be given in Theorem \ref{thm:dsatake_classify}.

Following to \cite[Chapter X, Theorem 3.29]{Helgason}
the structure of $\mathrm{Aut}(\Pi)$
is given as follows:
\[
\mathrm{Aut}(\Pi)=\begin{cases}
\{1\} & (\mathfrak{g}=\mathfrak{su}(2), \mathfrak{so}(2m+1), \mathfrak{sp}(n), \mathfrak{e}_{7}, \mathfrak{e}_{8}, \mathfrak{f}_{4}, \mathfrak{g}_{2}),\\
\mathbb{Z}_{2} & (\mathfrak{g}=\mathfrak{su}(n)\,(n\geq 3), \mathfrak{so}(2m)\,(m\geq 5), \mathfrak{e}_{6}),\\
\mathfrak{S}_{3} & (\mathfrak{g}=\mathfrak{so}(8)),
\end{cases}
\]
where $\mathbb{Z}_{2}$
and $\mathfrak{S}_{3}$
are the cyclic group of order two
and the symmetric group of order three, respectively.
Clearly, in the case when $\mathrm{Aut}(\Pi)=\{1\}$,
$\mathcal{DS}(S_{1},S_{2})$
consists of only one element, that is,
$\mathcal{DS}(S_{1},S_{2})=\{[(S_{1},S_{2})]\}$.
For the others,
we will obtain $\mathcal{DS}(S_{1},S_{2})$ by a case-by-case verification
based on the classification of $\mathrm{Inv}(\mathfrak{g})/\mathrm{Aut}(\mathfrak{g})$
as shown in Table \ref{table:fixed_pt_algebra}.

Let us consider the case when $\mathrm{Aut}(\Pi)=\mathbb{Z}_{2}$.
We first determine $\mathcal{DS}(S_{1},S_{2})$
for $(\mathfrak{g},\mathfrak{k}_{1})=(\mathfrak{g},\mathfrak{k}_{2})=(\mathfrak{so}(4m),\mathfrak{u}(2m))$
with $m\geq 3$.

\begin{ex}\label{ex:so4mu2mu2m}
Let $(\mathfrak{g},\mathfrak{k}_{1})=(\mathfrak{g},\mathfrak{k}_{2})=(\mathfrak{so}(4m),\mathfrak{u}(2m))$
with $m\geq 3$.
Denote by $(\Delta,\sigma_{1},\sigma_{2})$
the double $\sigma$-system of a quasi canonical form $(\mathfrak{g},\theta_{1},\theta_{2})$.
Let $\Pi$ be a $(\sigma_{1},\sigma_{2})$-fundamental system of $\Delta$.
If we write $\Pi=\{\alpha_{1},\dotsc,\alpha_{2m}\}$ as in Notation \ref{nota:Dr},
then $\mathrm{Aut}(\Pi)$
is generated by
\[
\tau:\Pi\to\Pi;
~(\alpha_{1},\dotsc,\alpha_{2m-2},\alpha_{2m-1},\alpha_{2m})
\mapsto
(\alpha_{1},\dotsc,\alpha_{2m-2},\alpha_{2m},\alpha_{2m-1}).
\]
Denote by $S_{i}=S(\Pi,\Pi_{i,0},p_{i})$ the Satake diagram
of $(\Delta,\sigma_{i})$ associated with $\Pi$.
Then, for each $i=1,2$, the graph of $S_{i}$
coincides with that of $S(\Pi,\Pi_{0}, p)$
or $S(\Pi,\tau(\Pi_{0}),\tau\cdot p))$ as in Table \ref{table:Satake_DIII}.
It can be shown that
$(S(\Pi,\Pi_{0},p),S(\Pi,\Pi_{0},p))$
and $(S(\Pi,\Pi_{0},p),S(\Pi,\tau(\Pi_{0}),\tau\cdot p))$
give a complete representative of
$\mathcal{DS}(S_{1},S_{2})$.
\end{ex}

\begin{table}[H]
\centering
\caption{Satake diagram of $(\mathfrak{so}(4m),\mathfrak{u}(2m))$ with $m\geq 3$}\label{table:Satake_DIII}
\begin{tabular}{cc}
\hline
\hline
$S(\Pi,\Pi_{0}, p)$ & $S(\Pi,\tau(\Pi_{0}),\tau\cdot p))$ \\
\hline
\hline
\begin{xy}
\ar@{-}(0,0)*++!D{\alpha_{1}}*{\bullet};(10,0)*++!D{\alpha_{2}}*{\circ}="a2"
\ar@{-}"a2";(15,0)*{}
\ar@{.}(15,0)*{};(20,0)*{}
\ar@{-}(20,0)*{};(25,0)*++!D{\alpha_{2m-3}}*{\bullet}="a2m-3"
\ar@{-}"a2m-3";(35,0)*++!L{\alpha_{2m-2}}*{\circ}="a2m-2"
\ar@{-}"a2m-2";(40,8)*++!L{\alpha_{2m-1}}*{\bullet}
\ar@{-}"a2m-2";(40,-8)*++!L{\alpha_{2m}}*{\circ}
\end{xy}
&
\begin{xy}
\ar@{-}(0,0)*++!D{\alpha_{1}}*{\bullet};(10,0)*++!D{\alpha_{2}}*{\circ}="a2"
\ar@{-}"a2";(15,0)*{}
\ar@{.}(15,0)*{};(20,0)*{}
\ar@{-}(20,0)*{};(25,0)*++!D{\alpha_{2m-3}}*{\bullet}="a2m-3"
\ar@{-}"a2m-3";(35,0)*++!L{\alpha_{2m-2}}*{\circ}="a2m-2"
\ar@{-}"a2m-2";(40,8)*++!L{\alpha_{2m-1}}*{\circ}
\ar@{-}"a2m-2";(40,-8)*++!L{\alpha_{2m}}*{\bullet}
\end{xy}\\
\hline
\hline
\end{tabular}
\end{table}

Except for
this example among
compact symmetric triads $(\mathfrak{g},\theta_{1},\theta_{2})$
with $\mathrm{Aut}(\Pi)=\mathbb{Z}_{2}$,
it is verified that $\mathcal{DS}(S_{1},S_{2})=\{[(S_{1},S_{2})]\}$ holds
by means of the following lemma.

\begin{lem}\label{lem:2satake_determ_1pt}
Assume that $S_{1}$ or $S_{2}$ is invariant under the action of
$\mathrm{Aut}(\Pi)$, that is,
there exists $i\in\{1,2\}$ such that
$S_{i}=\psi\cdot S_{i}$ holds for all $\psi\in\mathrm{Aut}(\Pi)$.
Then we have $\mathcal{DS}(S_{1},S_{2})=\{[(S_{1},S_{2})]\}$.
\end{lem}

We omit its proof since
this lemma is easily shown by the definition of $\mathcal{DS}(S_{1},S_{2})$.

\begin{ex}\label{ex:susosuu_ds}
Let us consider the case when
$(\mathfrak{g},\mathfrak{k}_{1})=(\mathfrak{su}(n),\mathfrak{so}(n))$
and $(\mathfrak{g},\mathfrak{k}_{2})=(\mathfrak{su}(n),\mathfrak{s}(\mathfrak{u}(a)\oplus\mathfrak{u}(b)))$
with $n\geq 3$.
Since 
the Satake diagram $S_{1}$
contains no black circles and no curved arrows,
$S_{1}$ is invariant under the action of $\mathrm{Aut}(\Pi)$.
From Lemma \ref{lem:2satake_determ_1pt}
we get $\mathcal{DS}(S_{1},S_{2})=\{[(S_{1},S_{2})]\}$.
\end{ex}

From the above argument
we conclude that $\mathcal{DS}(S_{1},S_{2})$
have been determined
in the case when $\mathrm{Aut}(\Pi)=\mathbb{Z}_{2}$.

Finally,
we consider the case when
$\mathrm{Aut}(\Pi)=\mathfrak{S}_{3}$.

\begin{ex}\label{ex:so8soasob}
Let $\mathfrak{g}=\mathfrak{so}(8)$.
From Table \ref{table:fixed_pt_algebra}
we have
$\{[(\mathfrak{so}(8),\theta)]\mid \theta\in\mathrm{Inv}(\mathfrak{so}(8))\}
=\{[(\mathfrak{so}(8),\mathfrak{so}(a)\oplus\mathfrak{so}(8-a))]\mid a=1,2,3,4\}$.
Here, we have used a special isomorphism $\mathfrak{u}(4)\simeq \mathfrak{so}(2)\oplus\mathfrak{so}(6)$.
Our argument proceeds by a case-by-case argument as follows.

We first consider the case when
$(\mathfrak{g},\mathfrak{k}_{1})$ or 
$(\mathfrak{g},\mathfrak{k}_{2})$ is isomorphic to $(\mathfrak{so}(8),\mathfrak{so}(4)\oplus\mathfrak{so}(4))$.
A similar manner as in Example \ref{ex:susosuu_ds} obeys $\mathcal{DS}(S_{1},S_{2})=\{[(S_{1},S_{2})]\}$.

Next, let us consider the case when
$(\mathfrak{g},\mathfrak{k}_{1})=(\mathfrak{so}(8),\mathfrak{so}(a)\oplus\mathfrak{so}(8-a))$
and $(\mathfrak{g},\mathfrak{k}_{2})=(\mathfrak{so}(8),\mathfrak{so}(c)\oplus\mathfrak{so}(8-c))$
for some $a,c\in\{1,2,3\}$.
Denote by $(\Delta,\sigma_{1},\sigma_{2})$
the double $\sigma$-system of a quasi-canonical form $(\mathfrak{g},\theta_{1},\theta_{2})$.
Let $\Pi$ be a $(\sigma_{1},\sigma_{2})$-fundamental system of $\Delta$.
If we write $\Pi=\{\alpha_{1},\alpha_{2},\alpha_{3},\alpha_{4}\}$
as in Notation \ref{nota:Dr},
then 
$\mathrm{Aut}(\Pi)=\{1,\kappa,\kappa^{2},\tau,\kappa\tau\kappa^{-1},\kappa^{2}\tau\kappa^{-2}\}$
holds, where
$\kappa, \tau\in\mathrm{Aut}(\Pi)$ are defined by
$\kappa:(\alpha_{1},\alpha_{2},\alpha_{3},\alpha_{4})
\mapsto
(\alpha_{4},\alpha_{2},\alpha_{1},\alpha_{3})$
and by
$\tau:(\alpha_{1},\alpha_{2},\alpha_{3},\alpha_{4})
\mapsto
(\alpha_{1},\alpha_{2},\alpha_{4},\alpha_{3})$.
Denote by $S_{i}$ the Satake diagram
of $(\Delta,\sigma_{i})$ associated with $\Pi$.
Then, there exist $\psi,\psi'\in\{1,\kappa,\kappa^{2}\}$ satisfying
$S_{1}=S(\Pi,\psi(\Pi_{0}^{(a)}),\psi\cdot p^{(a)})$
and $S_{2}=S(\Pi,\psi'(\Pi_{0}^{(c)}),\psi'\cdot p^{(c)})$,
where
the Satake diagram $S(\Pi,\psi(\Pi_{0}^{(*)}),\psi\cdot p^{(*)})$ 
are in Table \ref{table:Satake_DI} for $\psi\in\{1,\kappa,\kappa^{2}\}$.
We write $S^{*,\psi}:=S(\Pi,\psi(\Pi_{0}^{(*)}),\psi\cdot p^{(*)})$
for short.
Then it can be verified that
$\mathcal{DS}(S_{1},S_{2})=\{[(S^{a,1},S^{c,1})],[(S^{a,1},S^{c,\kappa})]\}$ holds
by a case-by-case verification.
For example, in the case when $a=1, c=2$,
we have
$(S^{1,1},S^{2,1})\not\sim(S^{1,1},S^{2,\kappa})\overset{\tau}{\sim}(S^{1,1},S^{2,\kappa^{2}})$.
This implies that 
$(S^{1,1},S^{2,1})$ and $(S^{1,1},S^{2,\kappa})$ give a complete representative of $\mathcal{DS}(S_{1},S_{2})$.
\end{ex}

\begin{table}[H]
\centering
\caption{Satake diagram of $(\mathfrak{so}(8),\mathfrak{so}(a)\oplus\mathfrak{so}(8-a))$ with $a=1,2,3$}\label{table:Satake_DI}
\begin{tabular}{cccc}
\hline
\hline
$a$ & $S(\Pi,\Pi_{0}^{(a)},p^{(a)})$ & $S(\Pi,\kappa(\Pi_{0}^{(a)}),\kappa\cdot p^{(a)})$
& $S(\Pi,\kappa^{2}(\Pi_{0}^{(a)}),\kappa^{2}\cdot p^{(a)})$ \\
\hline
\hline
$1$ &
\begin{xy}
\ar@{-}(0,0)*++!D{\alpha_{1}}*{\circ}="a1";(10,0)*++!D{\alpha_{2}}*{\bullet}="a2"
\ar@{-}"a2";(17.09,7.09)*++!D{\alpha_{3}}*{\bullet}="a3"
\ar@{-}"a2";(17.09,-7.09)*++!D{\alpha_{4}}*{\bullet}="a4"
\end{xy}
&
\begin{xy}
\ar@{-}(0,0)*++!D{\alpha_{1}}*{\bullet}="a1";(10,0)*++!D{\alpha_{2}}*{\bullet}="a2"
\ar@{-}"a2";(17.09,7.09)*++!D{\alpha_{3}}*{\bullet}="a3"
\ar@{-}"a2";(17.09,-7.09)*++!D{\alpha_{4}}*{\circ}="a4"
\end{xy} &
\begin{xy}
\ar@{-}(0,0)*++!D{\alpha_{1}}*{\bullet}="a1";(10,0)*++!D{\alpha_{2}}*{\bullet}="a2"
\ar@{-}"a2";(17.09,7.09)*++!D{\alpha_{3}}*{\circ}="a3"
\ar@{-}"a2";(17.09,-7.09)*++!D{\alpha_{4}}*{\bullet}="a4"
\end{xy}
\\
$2$ &
\begin{xy}
\ar@{-}(0,0)*++!D{\alpha_{1}}*{\circ}="a1";(10,0)*++!D{\alpha_{2}}*{\circ}="a2"
\ar@{-}"a2";(17.09,7.09)*++!D{\alpha_{3}}*{\bullet}="a3"
\ar@{-}"a2";(17.09,-7.09)*++!D{\alpha_{4}}*{\bullet}="a4"
\end{xy}
&
\begin{xy}
\ar@{-}(0,0)*++!D{\alpha_{1}}*{\bullet}="a1";(10,0)*++!D{\alpha_{2}}*{\circ}="a2"
\ar@{-}"a2";(17.09,7.09)*++!D{\alpha_{3}}*{\bullet}="a3"
\ar@{-}"a2";(17.09,-7.09)*++!D{\alpha_{4}}*{\circ}="a4"
\end{xy}
&\begin{xy}
\ar@{-}(0,0)*++!D{\alpha_{1}}*{\bullet}="a1";(10,0)*++!D{\alpha_{2}}*{\circ}="a2"
\ar@{-}"a2";(17.09,7.09)*++!D{\alpha_{3}}*{\circ}="a3"
\ar@{-}"a2";(17.09,-7.09)*++!D{\alpha_{4}}*{\bullet}="a4"
\end{xy}
\\
$3$ &
\begin{xy}
\ar@{-}(0,0)*++!D{\alpha_{1}}*{\circ}="a1";(10,0)*++!D{\alpha_{2}}*{\circ}="a2"
\ar@{-}"a2";(17.09,7.09)*++!D{\alpha_{3}}*{\circ}="a3"
\ar@{-}"a2";(17.09,-7.09)*++!D{\alpha_{4}}*{\circ}="a4"
\ar@/^/@{<->} "a3";"a4"
\end{xy}
&
\begin{xy}
\ar@{-}(0,0)*++!D{\alpha_{1}}*{\circ}="a1";(10,0)*++!D{\alpha_{2}}*{\circ}="a2"
\ar@{-}"a2";(17.09,7.09)*++!D{\alpha_{3}}*{\circ}="a3"
\ar@{-}"a2";(17.09,-7.09)*++!D{\alpha_{4}}*{\circ}="a4"
\ar@/^/@{<->} "a1";"a3"
\end{xy}&
\begin{xy}
\ar@{-}(0,0)*++!D{\alpha_{1}}*{\circ}="a1";(10,0)*++!D{\alpha_{2}}*{\circ}="a2"
\ar@{-}"a2";(17.09,7.09)*++!D{\alpha_{3}}*{\circ}="a3"
\ar@{-}"a2";(17.09,-7.09)*++!D{\alpha_{4}}*{\circ}="a4"
\ar@/_/@{<->} "a1";"a4"
\end{xy}\\
\hline
\hline
\end{tabular}
\end{table}

From the above argument we conclude:

\begin{thm}\label{thm:dsatake_classify}
Fix a compact simple Lie algebra $\mathfrak{g}$.
Let $[\theta_{1}],[\theta_{2}]\in\mathrm{Inv}(\mathfrak{g})/\mathrm{Aut}(\mathfrak{g})$ such that $(\mathfrak{g},\theta_{1},\theta_{2})$
is quasi-canonical.
Denote by $(S_{1},S_{2})$ the double Satake diagram
corresponding to $(\mathfrak{g},\theta_{1},\theta_{2})$.
Then we obtain $\mathcal{DS}(S_{1},S_{2})$ as follows:

\noindent
$(1)$
Let $\mathfrak{g}\neq\mathfrak{so}(4m)$ with $m\geq 2${\rm :} $\mathcal{DS}(S_{1},S_{2})=\{[(S_{1},S_{2})]\}$ holds.

\noindent
$(2)$
Let $\mathfrak{g}=\mathfrak{so}(4m)$ with $m\geq 3${\rm :}
\begin{enumerate}
\item[{\rm (2-a)}] In the case when $(\mathfrak{g},\mathfrak{k}_{i})=(\mathfrak{so}(4m),\mathfrak{u}(2m))$ for $i=1,2$,
the two double Satake diagrams
$(S(\Pi,\Pi_{0},p),S(\Pi,\Pi_{0},p))$ and $(S(\Pi,\Pi_{0},p),S(\Pi,\tau(\Pi_{0}),\tau\cdot p))$
as in Example {\rm \ref{ex:so4mu2mu2m}} give a complete representative of $\mathcal{DS}(S_{1},S_{2})$.
\item[{\rm (2-b)}] Otherwise, $\mathcal{DS}(S_{1},S_{2})=\{[(S_{1},S_{2})]\}$ holds.
\end{enumerate}

\noindent
$(3)$
Let $\mathfrak{g}=\mathfrak{so}(8)${\rm :}
\begin{enumerate}
\item[{\rm (3-a)}] In the case when $\mathfrak{k}_{1}$ or $\mathfrak{k}_{2}$
is isomorphic to $\mathfrak{so}(4)\oplus\mathfrak{so}(4)$,
$\mathcal{DS}(S_{1},S_{2})=\{[(S_{1},S_{2})]\}$ holds.
\item[{\rm (3-b)}] Otherwise, there exist $a,c\in\{1,2,3\}$ such that
$(\mathfrak{g},\mathfrak{k}_{1})=(\mathfrak{so}(8),\mathfrak{so}(a)\oplus\mathfrak{so}(8-a))$
and
$(\mathfrak{g},\mathfrak{k}_{2})=(\mathfrak{so}(8),\mathfrak{so}(c)\oplus\mathfrak{so}(8-c))$.
Then,
the two double Satake diagrams
$(S(\Pi,\Pi_{0}^{(a)},p^{(a)}),S(\Pi,\Pi_{0}^{(c)},p^{(c)}))$
and
$(S(\Pi,\Pi_{0}^{(a)},p^{(a)}),S(\Pi,\kappa(\Pi_{0}^{(c)}),\kappa\cdot p^{(c)}))$
as in Example {\rm \ref{ex:so8soasob}} give a
complete representative of $\mathcal{DS}(S_{1},S_{2})$.
\end{enumerate}
\end{thm}

As a corollary of
Theorem \ref{thm:dsatake_classify}
we can obtain $\mathcal{T}(\mathfrak{g},\mathfrak{k}_{1},\mathfrak{k}_{2})$
for $[\theta_{1}],[\theta_{2}]\in\mathrm{Inv}(\mathfrak{g})/\mathrm{Aut}(\mathfrak{g})$.
In order to present our determination of $\mathcal{T}(\mathfrak{g},\mathfrak{k}_{1},\mathfrak{k}_{2})$
we prepare the following notation.

\begin{nota}\label{nota:matrix_involutions}
In order to give involutions on a compact simple Lie algebra $\mathfrak{g}$,
we utilize the following notation:
If $I_{n}$ denotes the unit matrix of order $n$,
then we put
\begin{equation}\label{eqn:IJ}
\begin{array}{cc}
I_{a,b}=\begin{pmatrix}
I_{a} & O \\
O & -I_{b}
\end{pmatrix}\in GL(a+b,\mathbb{R}),&
J_{n}=\begin{pmatrix}
O & -I_{n} \\
I_{n} & O
\end{pmatrix}
\in GL(2n,\mathbb{R}),
\end{array}
\end{equation}
and $J_{n}'=I_{n-1,n+1}J_{n}\in GL(2n,\mathbb{R})$.
\end{nota}

\begin{cor}\label{cor:cst_classify}
Fix a compact simple Lie algebra $\mathfrak{g}$.
Then we have{\rm :}

\noindent
$(1)$
Let $\mathfrak{g}\neq\mathfrak{so}(4m)$ with $m\geq 2${\rm :}
$\mathcal{T}(\mathfrak{g},\mathfrak{k}_{1},\mathfrak{k}_{2})=\{[(\mathfrak{g},\theta_{1},\theta_{2})]\}$ holds.

\noindent
$(2)$
Let $\mathfrak{g}=\mathfrak{so}(4m)$ with $m\geq 3${\rm :}
\begin{enumerate}
\item[{\rm (2-a)}] In the case when $(\mathfrak{g},\mathfrak{k}_{i})=(\mathfrak{so}(4m),\mathfrak{u}(2m))$ for $i=1,2$,
$(S(\Pi,\Pi_{0},p),S(\Pi,\Pi_{0},p))$, $(S(\Pi,\Pi_{0},p),S(\Pi,\tau(\Pi_{0}),\tau\cdot p))$
as in Example {\rm \ref{ex:so4mu2mu2m}}
correspond to
the two compact symmetric triads
\begin{equation}\label{eqn:CST_2-a}
(\mathfrak{so}(4m),\mathrm{Ad}(J_{2m}),\mathrm{Ad}(J_{2m})),
\quad
(\mathfrak{so}(4m),\mathrm{Ad}(J_{2m}),\mathrm{Ad}(J_{2m}')),
\end{equation}
respectively.
Furthermore, 
the compact symmetric triads
\eqref{eqn:CST_2-a}
give a complete representative of $\mathcal{T}(\mathfrak{so}(4m),\mathfrak{u}(2m),\mathfrak{u}(2m))$.
\item[{\rm (2-b)}] Otherwise,
$\mathcal{T}(\mathfrak{so}(4m),\mathfrak{k}_{1},\mathfrak{k}_{2})=\{[(\mathfrak{so}(4m),\theta_{1},\theta_{2})]\}$ holds.
\end{enumerate}

\noindent
$(3)$
Let $\mathfrak{g}=\mathfrak{so}(8)${\rm :}
\begin{enumerate}
\item[{\rm (3-a)}] In the case when $\mathfrak{k}_{1}$
or 
$\mathfrak{k}_{2}$
is isomorphic to $\mathfrak{so}(4)\oplus\mathfrak{so}(4)$,
we have
$\mathcal{T}(\mathfrak{so}(8),\mathfrak{k}_{1},\mathfrak{k}_{2})=\{[(\mathfrak{so}(8),\theta_{1},\theta_{2})]\}$.
\item[{\rm (3-b)}] Otherwise,
we have
$(\mathfrak{g},\mathfrak{k}_{1})=(\mathfrak{so}(8),\mathfrak{so}(a)\oplus\mathfrak{so}(8-a))$
and
$(\mathfrak{g},\mathfrak{k}_{2})=(\mathfrak{so}(8),\mathfrak{so}(c)\oplus\mathfrak{so}(8-c))$
for some $a,c\in\{1,2,3\}$.
Then,
$(S(\Pi,\Pi_{0}^{(a)},p^{(a)}),S(\Pi,\Pi_{0}^{(c)},p^{(c)}))$
and
$(S(\Pi,\Pi_{0}^{(a)},p^{(a)}),S(\Pi,\kappa(\Pi_{0}^{(c)}),\kappa\cdot p^{(c)}))$
as in Example {\rm \ref{ex:so8soasob}}
correspond to
the two compact symmetric triads
\begin{equation}\label{eqn:CST_3-b}
(\mathfrak{so}(8),\mathrm{Ad}(I_{a,8-a}),\mathrm{Ad}(I_{c,8-c})),\quad
(\mathfrak{so}(8),\mathrm{Ad}(I_{a,8-a}),\tilde{\kappa}\mathrm{Ad}(I_{c,8-c})\tilde{\kappa}^{-1}),
\end{equation}
respectively,
where $\tilde{\kappa}$ denotes the extension of $\kappa$ to an automorphism of $\mathfrak{so}(8)$.
Furthermore,
\eqref{eqn:CST_3-b}
give a complete representative of
$\mathcal{T}(\mathfrak{so}(8),\mathfrak{k}_{1},\mathfrak{k}_{2})$.
\end{enumerate}
\end{cor}

From Corollary \ref{cor:cst_classify},
with a few exceptions,
the isomorphism class
$[(\mathfrak{g},\theta_{1},\theta_{2})]$
is uniquely determined
by means of the Lie algebra structures of $\mathfrak{k}_{1}$ and 
$\mathfrak{k}_{2}$.
Then, there is no confusion when we write
$[(\mathfrak{g},\mathfrak{k}_{1},\mathfrak{k}_{2})]$
in place of $[(\mathfrak{g},\theta_{1},\theta_{2})]$
except for compact simple symmetric triads as in Corollary \ref{cor:cst_classify}, (2-a) and (3-b).
On the other hand,
in the case of (2-a) in Corollary \ref{cor:cst_classify},
we shall use the symbols
$[(\mathfrak{so}(4m),\mathfrak{u}(2m),\mathfrak{u}(2m))]$
and 
$[(\mathfrak{so}(4m),\mathfrak{u}(2m),\mathfrak{u}(2m)')]$
as the isomorphism classes of 
$(\mathfrak{so}(4m),\mathrm{Ad}(J_{2m}),\mathrm{Ad}(J_{2m}))$ and
$(\mathfrak{so}(4m),\mathrm{Ad}(J_{2m}),\mathrm{Ad}(J_{2m}'))$,
respectively.
In the case of (3-b) in Corollary \ref{cor:cst_classify},
we shall also 
write
$[(\mathfrak{so}(8),\mathfrak{so}(a)\oplus\mathfrak{so}(8-a),\mathfrak{so}(c)\oplus\mathfrak{so}(8-c))]$ and
$[(\mathfrak{so}(8),\mathfrak{so}(a)\oplus\mathfrak{so}(8-a),\tilde{\kappa}(\mathfrak{so}(c)\oplus\mathfrak{so}(8-c)))]$
as the isomorphism classes of 
$(\mathfrak{so}(8),\mathrm{Ad}(I_{a,8-a}),\mathrm{Ad}(I_{c,8-c}))$ and
$(\mathfrak{so}(8),\mathrm{Ad}(I_{a,8-a}),\tilde{\kappa}\mathrm{Ad}(I_{c,8-c})\tilde{\kappa}^{-1})$,
respectively.

\subsubsection{Determination of rank and order for double $\sigma$-systems}

Based on the classification,
we will determine the rank and the order of
the double $\sigma$-system $(\Delta,\sigma_{1},\sigma_{2})$ for compact symmetric triads
$(\mathfrak{g},\theta_{1},\theta_{2})$
such that $\mathfrak{g}$ is simple.
Since $(\Delta,\sigma_{1},\sigma_{2})$ is canonical,
we have $\mathrm{rank}[(\Delta,\sigma_{1},\sigma_{2})]=\mathrm{rank}(\Delta,\sigma_{1},\sigma_{2})$
and $\mathrm{ord}[(\Delta,\sigma_{1},\sigma_{2})]=\mathrm{ord}(\Delta,\sigma_{1},\sigma_{2})$.

First, we consider the case when $\theta_{1}\sim\theta_{2}$.
Then
$(\Delta,\sigma_{1},\sigma_{2})\sim(\Delta,\sigma_{1},\sigma_{1})$
holds.
Since $(\Delta,\sigma_{1},\sigma_{1})$ is canonical,
by Theorem \ref{thm:dsig_dsatake_equiv},
we obtain
$\mathrm{rank}(\Delta, \sigma_{1},\sigma_{2})=\mathrm{rank}(\Delta,\sigma_{1},\sigma_{1})
=\mathrm{rank}(\mathfrak{g},\theta_{1})$
and
$\mathrm{ord}(\Delta, \sigma_{1},\sigma_{2})=\mathrm{ord}(\Delta,\sigma_{1},\sigma_{1})=1$.
In addition, we have
the value of $\mathrm{rank}(\mathfrak{g},\theta_{1})$ from TABLE V
in \cite{Helgason}.
Thus, we have determined
$\mathrm{rank}[(\Delta,\sigma_{1},\sigma_{2})]$
and
$\mathrm{ord}[(\Delta,\sigma_{1},\sigma_{2})]$
in the case when $\theta_{1}\sim\theta_{2}$.

Secondly,
we consider the case when $\theta_{1}\not\sim\theta_{2}$.
In a similar manner as in Subsection \ref{sec:NssSd},
$(\Delta,\sigma_{1},\sigma_{2})$
can be reconstructed from its double Satake diagram.
Then, a direct calculation gives the rank and the order of $(\Delta,\sigma_{1},\sigma_{2})$.

\begin{ex}
Let $(\mathfrak{g},\mathfrak{k}_{1},\mathfrak{k}_{2})=(\mathfrak{so}(8),\mathfrak{so}(1)\oplus\mathfrak{so}(7),\tilde{\kappa}(\mathfrak{so}(2)\oplus\mathfrak{so}(6)))$
and $(\Delta,\sigma_{1},\sigma_{2})$ denote its double $\sigma$-system.
Then,
$(S(\Pi,\Pi_{0}^{(1)},p^{(1)}),S(\Pi,\kappa(\Pi_{0}^{(2)}),\kappa\cdot p^{(2)}))$
as in Table \ref{table:Satake_DI}
gives the double Satake diagram of $(\Delta,\sigma_{1},\sigma_{2})$.
We write $\Pi=\{\alpha_{1},\dotsc,\alpha_{4}\}$
by means of the standard basis $e_{1},\dotsc, e_{4}$ of $\mathbb{R}^{4}$ as in Notation \ref{nota:Dr}.
Under this setting, we have
\begin{equation}\label{eqn:sigsso8so1726}
\sigma_{1}(e_{1},e_{2},e_{3},e_{4})=(e_{1},-e_{2},-e_{3},-e_{4}),\quad
\sigma_{2}(e_{1},e_{2},e_{3},e_{4})=(e_{2},e_{1},e_{4},e_{3}).
\end{equation}
Since we have
\[
\mathfrak{t}^{\sigma_{1}}\cap\mathfrak{t}^{\sigma_{2}}=
\mathbb{R}e_{1}\cap(\mathbb{R}(e_{1}+e_{2})\oplus\mathbb{R}(e_{3}+e_{4}))=\{0\},
\]
$\mathrm{rank}(\Delta,\sigma_{1},\sigma_{2})=0$ holds.
We also obtain $\mathrm{ord}(\Delta,\sigma_{1},\sigma_{2})=4$
by means of \eqref{eqn:sigsso8so1726}.
\end{ex}

We can carry out a similar calculation
for the other compact symmetric triads
$(\mathfrak{g},\theta_{1},\theta_{2})$ such that
$\mathfrak{g}$ is simple and that $\theta_{1}\not\sim\theta_{2}$.
Then, we have the following proposition.

\begin{pro}
Table {\rm \ref{table:rank_ord}} exhibits the ranks and the orders
of the isomorphism classes of the double $\sigma$-systems corresponding to compact 
symmetric triads $(\mathfrak{g},\theta_{1},\theta_{2})$ such that
$\mathfrak{g}$ is simple and that $\theta_{1}\not\sim\theta_{2}$.
\end{pro}

\begin{rem}\label{rem:table_ro}
As will be shown later,
Table {\rm \ref{table:rank_ord}} exhibits the ranks and the orders
of the isomorphism classes of compact symmetric triads $(G,\theta_{1},\theta_{2})$
such that $G$ is simple.
Indeed, this is shown by means of
Theorems \ref{thm:cst_exist_can} and \ref{thm:cst_RO_can}
in the next section.
\end{rem}

\begin{table}[H]
\centering
\caption{Rank and order for double $\sigma$-system corresponding to $(\mathfrak{g},\theta_{1},\theta_{2})$ with $\theta_{1}\not\sim\theta_{2}$}\label{table:rank_ord}
\renewcommand\arraystretch{2.0}
\tiny
\begin{tabular}{cccc}
\hline
\hline
$(\mathfrak{g}, \mathfrak{k}_{1},\mathfrak{k}_{2})$ & Rank & Order &  Remark\\
\hline
\hline
$(\mathfrak{su}(2m), \mathfrak{so}(2m), \mathfrak{sp}(m))$ & $m-1$ & $2$ & \\
\hline
$(\mathfrak{su}(n), \mathfrak{so}(n), \mathfrak{s}(\mathfrak{u}(a) \oplus\mathfrak{u}(b)))$ & $a$ & $2$ & $n\geq 2a$\\
\hline
$(\mathfrak{su}(2m), \mathfrak{sp}(m), \mathfrak{s}(\mathfrak{u}(a)\oplus\mathfrak{u}(b)))$
& $\left[\dfrac{a}{2}\right]$ & $\begin{cases}4 & (\text{$a$: odd, $m>a$} ),\\2 & (\text{otherwise})\end{cases}$ &$m\geq a$ \\
\hline
$(\mathfrak{su}(n), \mathfrak{s}(\mathfrak{u}(a)\oplus\mathfrak{u}(b)), \mathfrak{s}(\mathfrak{u}(c) \oplus\mathfrak{u}(d)))$
& $a$ & $2$ & $a<c \leq d<b$ \\
\hline
$(\mathfrak{so}(n), \mathfrak{so}(a)\oplus\mathfrak{so}(b), \mathfrak{so}(c) \oplus\mathfrak{so}(d))$
& $a$ & $2$ & $a<c \leq d<b$ \\
\hline
$(\mathfrak{so}(8), \mathfrak{so}(a)\oplus\mathfrak{so}(b),\tilde{\kappa}(\mathfrak{so}(c) \oplus\mathfrak{so}(d)))$
& $\begin{cases}0 & ((a,c)=(1,\{1,2,3\})),\\1&((a,c)=(2,\{2,3\})),\\2 & ((a,c)=(3,3))\end{cases}$ & $\begin{cases}
2 & ((a,c)=(2,2)),\\
3 & ((a,c)=(1,1),(3,3)),\\
4 & ((a,c)=(1,2),(2,3)),\\
6 & ((a,c)=(1,3))\end{cases}$& \\
\hline
$(\mathfrak{so}(2m), \mathfrak{so}(a)\oplus\mathfrak{so}(b), \mathfrak{u}(m))$
& $\left[\dfrac{a}{2}\right]$ & $\begin{cases}4 & (\text{$a$: odd, $m>a$} ),\\2 & (\text{otherwise})\end{cases}$ &$m\geq a$ \\
\hline
$(\mathfrak{so}(4m), \mathfrak{u}(2m), \mathfrak{u}(2m)')$ & $m-1$ & $2$ & \\
\hline
$(\mathfrak{sp}(n), \mathfrak{u}(n), \mathfrak{sp}(a)\oplus\mathfrak{sp}(b))$
& $a$ & $2$ & $n\geq 2a$\\
\hline
$(\mathfrak{sp}(n), \mathfrak{sp}(a)\oplus\mathfrak{sp}(b), \mathfrak{sp}(c)\oplus\mathfrak{sp}(d))$
& $a$ & $2$ & $a<c \leq d<b$ \\
\hline
$(\mathfrak{e}_{6}, \mathfrak{sp}(4), \mathfrak{su}(6)\oplus \mathfrak{su}(2))$ & $4$ & $2$ & \\
\hline
$(\mathfrak{e}_{6}, \mathfrak{sp}(4), \mathfrak{so}(10)\oplus\mathfrak{so}(2))$ & $2$ & $2$ & \\
\hline
$(\mathfrak{e}_{6}, \mathfrak{sp}(4), \mathfrak{f}_{4})$ & $2$ & $2$ & \\
\hline
$(\mathfrak{e}_{6}, \mathfrak{su}(6)\oplus\mathfrak{su}(2), \mathfrak{so}(10)\oplus\mathfrak{so}(2))$ & $2$ & $2$ & \\
\hline
$(\mathfrak{e}_{6}, \mathfrak{su}(6)\oplus\mathfrak{su}(2), \mathfrak{f}_{4})$ & $1$ & $2$ & \\
\hline
$(\mathfrak{e}_{6}, \mathfrak{so}(10)\oplus\mathfrak{so}(2), \mathfrak{f}_{4})$ & $1$ & $2$ & \\
\hline
$(\mathfrak{e}_{7}, \mathfrak{su}(8), \mathfrak{so}(12)\oplus\mathfrak{su}(2))$ & $4$ & $2$ & \\
\hline
$(\mathfrak{e}_{7}, \mathfrak{su}(8), \mathfrak{e}_{6}\oplus\mathfrak{so}(2))$ & $3$ & $2$ & \\
\hline
$(\mathfrak{e}_{7}, \mathfrak{so}(12) \oplus\mathfrak{su}(2), \mathfrak{e}_{6}\oplus\mathfrak{so}(2))$ & $2$ & $2$ & \\
\hline
$(\mathfrak{e}_{8}, \mathfrak{so}(16), \mathfrak{e}_{7} \oplus\mathfrak{su}(2))$ & $4$ & $2$ & \\
\hline
$(\mathfrak{f}_{4}, \mathfrak{su}(2)\oplus\mathfrak{sp}(3), \mathfrak{so}(9))$ & $1$ & $2$ & \\
\hline
\hline
\end{tabular}
\renewcommand\arraystretch{1.0}
\end{table}

\subsubsection{Special isomorphism and self-duality}\label{sec:cst_SI}

First, we consider special isomorphisms for compact symmetric triads.
In the theory of compact Lie algebras,
there are some special isomorphisms
for low-dimensional compact Lie algebras
(cf.~\cite[pp.~519--520]{Helgason}).
Hence we find that there are some overlaps in Table \ref{table:fixed_pt_algebra}.
This obeys special isomorphisms for compact symmetric triads with low rank
as follows.

\begin{cor}\label{cor:si_CST}
The following relations hold:
\begin{enumerate}
\item $(\mathfrak{so}(8),\mathfrak{u}(4),\mathfrak{so}(4)\oplus\mathfrak{so}(4))\sim
(\mathfrak{so}(8),\mathfrak{so}(2)\oplus\mathfrak{so}(6),\mathfrak{so}(4)\oplus\mathfrak{so}(4))$.
\item $(\mathfrak{so}(5),\mathfrak{so}(1)\oplus\mathfrak{so}(4),\mathfrak{so}(2)\oplus\mathfrak{so}(3))
\sim (\mathfrak{sp}(2),\mathfrak{sp}(1)\oplus\mathfrak{sp}(1),\mathfrak{u}(2))$.
\item $(\mathfrak{su}(4),\mathfrak{so}(4),\mathfrak{sp}(2))
\sim (\mathfrak{so}(6),\mathfrak{so}(3)\oplus\mathfrak{so}(3),\mathfrak{so}(1)\oplus\mathfrak{so}(5))$.
\item $(\mathfrak{su}(4),\mathfrak{so}(4),\mathfrak{s}(\mathfrak{u}(2)\oplus\mathfrak{u}(2)))
\sim (\mathfrak{so}(6),\mathfrak{so}(3)\oplus\mathfrak{so}(3),\mathfrak{so}(2)\oplus\mathfrak{so}(4))$.
\item $(\mathfrak{su}(4),\mathfrak{so}(4),\mathfrak{s}(\mathfrak{u}(1)\oplus\mathfrak{u}(3)))
\sim (\mathfrak{so}(6),\mathfrak{so}(3)\oplus\mathfrak{so}(3),\mathfrak{u}(3))$.
\item $(\mathfrak{su}(4),\mathfrak{sp}(2),\mathfrak{s}(\mathfrak{u}(2)\oplus\mathfrak{u}(2)))
\sim (\mathfrak{so}(6),\mathfrak{so}(1)\oplus\mathfrak{so}(5),\mathfrak{so}(2)\oplus\mathfrak{so}(4))$.
\item $(\mathfrak{su}(4),\mathfrak{sp}(2),\mathfrak{s}(\mathfrak{u}(1)\oplus\mathfrak{u}(3)))
\sim(\mathfrak{so}(6),\mathfrak{so}(1)\oplus\mathfrak{so}(5),\mathfrak{u}(3))$.
\item $(\mathfrak{su}(4),\mathfrak{s}(\mathfrak{u}(2)\oplus\mathfrak{u}(2)),\mathfrak{s}(\mathfrak{u}(1)\oplus\mathfrak{u}(3)))
\sim (\mathfrak{so}(6),\mathfrak{so}(2)\oplus\mathfrak{so}(4),\mathfrak{u}(3))$.
\end{enumerate}
\end{cor}

\begin{proof}
(1) For $i=1,2$, let $\theta_{i}$ and $\theta_{i}'$ be the involutions of $\mathfrak{g}=\mathfrak{so}(8)$ defined by
\[
\begin{array}{ccc}
\theta_{1}=\mathrm{Ad}(I_{2,6}),&
\theta_{1}'=\mathrm{Ad}(J_{4}),&
\theta_{2}=\theta_{2}'=\mathrm{Ad}(I_{4,4}).
\end{array}
\]
Then we have $\mathfrak{k}_{1}\simeq\mathfrak{so}(2)\oplus\mathfrak{so}(6)\simeq\mathfrak{u}(4)\simeq\mathfrak{k}_{1}'$
and $\mathfrak{k}_{2}=\mathfrak{k}_{2}'\simeq\mathfrak{so}(4)\oplus\mathfrak{so}(4)$.
It follows from Corollary \ref{cor:cst_classify}, (3-a)
that $(\mathfrak{g},\theta_{1},\theta_{2})\sim(\mathfrak{g},\theta_{1}',\theta_{2}')$ holds.
In a similar argument we get (2)--(8).
\end{proof}

A compact symmetric triad
$(\mathfrak{g},\theta_{1},\theta_{2})$
is said to be \textit{self-dual},
if it satisfies $(\mathfrak{g},\theta_{1},\theta_{2})\sim(\mathfrak{g},\theta_{2},\theta_{1})$.
Secondly, we classify self-dual compact symmetric triads
$(\mathfrak{g},\theta_{1},\theta_{2})$ in the case when $\mathfrak{g}$ is simple.
It is verified that,
if two compact symmetric triads
$(\mathfrak{g},\theta_{1},\theta_{2})$ and $(\mathfrak{g},\theta_{1}',\theta_{2}')$
are isomorphic, and 
$(\mathfrak{g},\theta_{1},\theta_{2})$
is self-dual,
then so is $(\mathfrak{g},\theta_{1}',\theta_{2}')$.
In the case when $\theta_{1}\sim\theta_{2}$,
it follows from 
$(\mathfrak{g},\theta_{1},\theta_{1})\sim(\mathfrak{g},\theta_{1},\theta_{2})$
that $(\mathfrak{g},\theta_{1},\theta_{2})$ is self-dual.
In the case when $\theta_{1}\not\sim\theta_{2}$,
we can determine whether $(\mathfrak{g},\theta_{1},\theta_{2})$
is self-dual or not by means of our classification as follows.

\begin{cor}\label{cor:sdual_CST}
Let $\theta_{1}$ and $\theta_{2}$ be involutions
on a compact simple Lie algebra with $\theta_{1}\not\sim\theta_{2}$.
The only self-dual compact symmetric triads are given by
$(\mathfrak{so}(4m),\mathrm{Ad}(J_{2m}),\mathrm{Ad}(J_{2m}'))$
with $m\geq 3$, and by
$(\mathfrak{so}(8),\mathrm{Ad}(I_{a,8-a}),\tilde{\kappa}\mathrm{Ad}(I_{a,8-a})\tilde{\kappa}^{-1})$ with $a\in\{1,2,3\}$.

In particular, $\mathfrak{k}_{1}\simeq\mathfrak{k}_{2}$ implies that
$(\mathfrak{g},\theta_{1},\theta_{2})$ is self-dual.
\end{cor}

\begin{proof}
It is clear that, if $(\mathfrak{g},\theta_{1},\theta_{2})$ is self-dual,
then $\mathfrak{k}_{1}\simeq\mathfrak{k}_{2}$ holds.
Conversely, from the classification for compact symmetric triads,
the only compact symmetric triads
$(\mathfrak{g},\theta_{1},\theta_{2})$ satisfying $\mathfrak{k}_{1}\simeq\mathfrak{k}_{2}$
are ones as in the statement.
Furthermore, it is verified that they are self-dual as follows:
By using $\mathrm{Ad}(I_{2m-1,2m+1})^{2}=1$ we have
\begin{align*}
(\mathfrak{g},\theta_{1},\theta_{2})
&=(\mathfrak{so}(4m),\mathrm{Ad}(J_{2m}),\mathrm{Ad}(I_{2m-1,2m+1})\mathrm{Ad}(J_{2m})\mathrm{Ad}(I_{2m-1,2m+1})^{-1})\\
&\sim(\mathfrak{so}(4m),\mathrm{Ad}(I_{2m-1,2m+1})\mathrm{Ad}(J_{2m})\mathrm{Ad}(I_{2m-1,2m+1})^{-1},\mathrm{Ad}(J_{2m}))\\
&=(\mathfrak{g},\theta_{2},\theta_{1}).
\end{align*}
It is shown that the double Satake diagram
for
$(\mathfrak{so}(8),\mathrm{Ad}(I_{a,8-a}),\tilde{\kappa}\mathrm{Ad}(I_{a,8-a})\tilde{\kappa}^{-1})$ and
that for
$(\mathfrak{so}(8),\tilde{\kappa}\mathrm{Ad}(I_{a,8-a})\tilde{\kappa}^{-1},\mathrm{Ad}(I_{a,8-a}))$
are isomorphic.
Thus, by Theorem \ref{thm:cst_dsatake_sim}, $(3)\Rightarrow(1)$
we have $(\mathfrak{so}(8),\mathrm{Ad}(I_{a,8-a}),\tilde{\kappa}\mathrm{Ad}(I_{a,8-a})\tilde{\kappa}^{-1})$
$\sim$$(\mathfrak{so}(8),\tilde{\kappa}\mathrm{Ad}(I_{a,8-a})\tilde{\kappa}^{-1},\mathrm{Ad}(I_{a,8-a}))$.
Hence we have the assertion.
\end{proof}

\section{Canonical forms in compact symmetric triads}\label{sec:cst_cf}

In Subsection \ref{sec:cst_can},
we define the canonicality for compact symmetric triads,
and give concrete examples of canonical compact symmetric triads.
In Subsection \ref{sec:cst_can_exist},
for any compact symmetric triad $(G,\theta_{1},\theta_{2})$,
we prove the existence of a canonical one
$(G,\theta_{1},\theta_{2}')\sim(G,\theta_{1},\theta_{2})$
in the case when $G$ is simple.
In Subsection \ref{sec:rankorder},
we give properties of canonical compact symmetric triads
(Theorem \ref{thm:cst_RO_can}).

\subsection{Definition and examples for canonical compact symmetric triads}\label{sec:cst_can}

Let $G$ be a compact connected semisimple
Lie group, and $\mathfrak{g}$ denote its Lie algebra.

\begin{dfn}\label{dfn:cst_can}
A compact symmetric triad
$(G,\theta_{1},\theta_{2})$
is said to be \textit{canonical},
if there exists a maximal abelian subalgebra $\mathfrak{t}$ of $\mathfrak{g}$
which satisfies the following conditions:
\begin{enumerate}
\item[(C1)] $\mathfrak{t}$ is quasi-canonical with respect to $(G,\theta_{1},\theta_{2})$,
that is,
$\mathfrak{t}$ satisfies the conditions (1) and (2) as in Definition \ref{dfn:cst_qcf}.
\item[(C2)] $\mathrm{ord}(\theta_{1}\theta_{2})=\mathrm{ord}(d\theta_{1}d\theta_{2}|_{\mathfrak{t}})$.
\end{enumerate}
Then, $\mathfrak{t}$ is said to be canonical with respect to 
$(G,\theta_{1},\theta_{2})$.
A \textit{canonical form} of $[(G,\theta_{1},\theta_{2})]$
is a representative $(G,\theta_{1}',\theta_{2}')$
of the isomorphism class $[(G,\theta_{1},\theta_{2})]$ such that
$(G,\theta_{1}',\theta_{2}')$
is canonical as a compact symmetric triad.
\end{dfn}

In the case when $(G,\theta_{1},\theta_{2})$ is canonical,
the condition (C2) implies
that
$\mathrm{ord}(\theta_{1}\theta_{2})$
and $\mathrm{ord}[(G,\theta_{1},\theta_{2})]$
are finite.
Here, we give examples of canonical compact symmetric triads
as follows.

\begin{ex}\label{ex:order1_can}
For any involution $\theta$ on $G$,
$(G,\theta,\theta)$ is canonical.
Indeed, let $\mathfrak{t}$
be a maximal abelian subalgebra of $\mathfrak{g}$ such that $\mathfrak{t}\cap\mathfrak{m}$
is a maximal abelian subspace of $\mathfrak{m}$.
Then $(G,\theta,\theta)$
and $\mathfrak{t}$ satisfy the two
conditions as in Definition
\ref{dfn:cst_can}.
\end{ex}

\begin{ex}\label{ex:com_can}
Any commutative compact symmetric triad
$(G,\theta_{1},\theta_{2})$ with $\theta_{1}\not\sim\theta_{2}$
is canonical.
Indeed, it follows from Lemma \ref{lem:comm_quasican}
that there exists
a maximal abelian subalgebra $\mathfrak{t}$ of $\mathfrak{g}$
such that $\mathfrak{t}\cap\mathfrak{m}_{i}$ and
$\mathfrak{t}\cap(\mathfrak{m}_{1}\cap\mathfrak{m}_{2})$
are maximal abelian subspaces of $\mathfrak{m}_{i}$ and $\mathfrak{m}_{1}\cap\mathfrak{m}_{2}$,
respectively.
We have shown that $\mathfrak{t}$ is quasi-canonical with respect to $(G,\theta_{1},\theta_{2})$.
In addition, from Lemma \ref{lem:t1simt2>dt1dt2=1}
we obtain
$2\leq \mathrm{ord}(d\theta_{1}d\theta_{2}|_{\mathfrak{t}})\leq \mathrm{ord}(\theta_{1}\theta_{2})=2$.
In particular, $\mathfrak{t}$ satisfies the condition (C2).
\end{ex}

In some sense,
the canonical forms are not uniquely determined,
namely, there exist two compact symmetric triads
$(G,\theta_{1},\theta_{2})\sim(G,\theta_{1}',\theta_{2}')$
such that they are canonical and that $(G,\theta_{1},\theta_{2})$
and $(G,\theta_{1}',\theta_{2}')$ are not equivalent
under the following equivalence relation $\equiv$.

\begin{dfn}
Two compact symmetric triads
$(G,\theta_{1},\theta_{2})$,
$(G,\theta_{1}',\theta_{2}')$ satisfy
$(G,\theta_{1},\theta_{2})\equiv(G,\theta_{1}',\theta_{2}')$,
if there exists $\varphi\in\mathrm{Aut}(G)$ satisfying
$\theta_{1}'=\varphi\theta_{1}\varphi^{-1}$ and 
$\theta_{2}'=\varphi\theta_{2}\varphi^{-1}$.
In a similar manner,
we define an equivalence relation on the set of compact symmetric triads at the Lie algebra level.
By the definition
$(G,\theta_{1},\theta_{2})\equiv(G,\theta_{1}',\theta_{2}')$
implies
$(G,\theta_{1},\theta_{2})\sim(G,\theta_{1}',\theta_{2}')$.
The converse does not hold in general.
\end{dfn}

The following example
gives compact symmetric triad
$(G,\theta_{1},\theta_{2})\sim(G,\theta_{1}',\theta_{2}')$
such that they are canonical and that $(G,\theta_{1},\theta_{2})\not\equiv(G,\theta_{1}',\theta_{2}')$
at the Lie algebra level.
We find that another example is given in \cite[Examples 2.14--16]{BIS}.

\begin{ex}\label{ex:spn}
Let $\mathfrak{g}=\mathfrak{su}(4m)$ with $m\geq 1$.
We define two involutions $\theta$ and $\theta'$ of $\mathfrak{g}$ as follows:
\[
\begin{array}{ccc}
\theta(Z)=\overline{Z},&
\theta'(Z)=I_{2m,2m}ZI_{2m,2m}&
(Z\in\mathfrak{g}),
\end{array}
\]
where $I_{2m,2m}$ is defined in \eqref{eqn:IJ}.
Then we have $\theta\theta'=\theta'\theta$.
Furthermore, 
$\mathfrak{g}^{\theta}=\mathfrak{so}(4m)$
and 
\[
\mathfrak{g}^{\theta'}
=\left\{
\left(
\begin{array}{cc}
Z_{1} & O \\
O & Z_{2}
\end{array}
\right)\biggm|
\begin{array}{l}
Z_{1}, Z_{2}\in\mathfrak{u}(2m),\\
\mathrm{Tr}(Z_{1}+Z_{2})=0
\end{array}
\right\}= \mathfrak{s}(\mathfrak{u}(2m)\oplus\mathfrak{u}(2m)).
\]
If we put
\[
I'_{2m,2m}=\dfrac{1}{\sqrt{2}}\left(\begin{array}{cc}
I_{2m} & \sqrt{-1}I_{2m}\\
\sqrt{-1}I_{2m} & I_{2m}
\end{array}\right)\in SU(4m),
\]
then the product
$I'_{2m,2m}I_{2m,2m}(I'_{2m,2m})^{-1}$ has the following expression:
\[
I'_{2m,2m}I_{2m,2m}(I'_{2m,2m})^{-1}
=\left(
\begin{array}{cc}
O & -\sqrt{-1}I_{2m} \\
\sqrt{-1}I_{2m} & O
\end{array}
\right)=:I_{2m,2m}''\in SU(4m).
\]
Since
$(I_{2m,2m}'')^{2}=I_{4m}$ holds,
we have another involution $\theta'':=\mathrm{Ad}(I_{2m,2m}'')$ of $\mathfrak{g}$.
Then $\theta''$ satisfies $\theta\theta''=\theta''\theta$
and $\theta'\sim\theta''$.
In addition, we have
$\mathfrak{g}^{\theta''}\simeq\mathfrak{g}^{\theta'}
=\mathfrak{s}(\mathfrak{u}(2m)\oplus\mathfrak{u}(2m))$.

Now, let us consider the following two compact symmetric triads:
\[
(\mathfrak{g},{\theta}_{1},{\theta}_{2})=(\mathfrak{su}(4m),\theta,\theta'),\quad
(\mathfrak{g},{\theta}_{1}',{\theta}_{2}')=(\mathfrak{su}(4m),\theta,\theta'').
\]
It follows from 
Corollary \ref{cor:cst_classify}
that
$(\mathfrak{g},{\theta}_{1},{\theta}_{2})
\sim (\mathfrak{g},{\theta}_{1}',{\theta}_{2}')$ holds.
In addition, by Example \ref{ex:com_can}
they are canonical.
A direct calculation shows that
\begin{align*}
\mathfrak{k}_{1}\cap\mathfrak{k}_{2}&=\left\{
\left(
\begin{array}{cc}
X_{1} & O \\
O & X_{2}
\end{array}
\right)\biggm|
X_{1}, X_{2}\in\mathfrak{so}(2m)
\right\}= \mathfrak{so}(2m)\oplus\mathfrak{so}(2m),\\
\mathfrak{k}_{1}'\cap\mathfrak{k}_{2}'&=\left\{
\left(
\begin{array}{cc}
X_{1} & X_{2} \\
-X_{2} & X_{1}
\end{array}
\right)\biggm|
\begin{array}{ll}
X_{1}\in\mathfrak{so}(2m),\\
X_{2}\in\mathfrak{gl}(2m,\mathbb{R});X_{2}={}^{t}X_{2}
\end{array}
\right\}\simeq \mathfrak{u}(2m).
\end{align*}
This implies that
$\mathfrak{k}_{1}\cap\mathfrak{k}_{2}$ is not isomorphic to
$\mathfrak{k}_{1}'\cap\mathfrak{k}_{2}'$.
Thus, we have $(\mathfrak{g},{\theta}_{1},{\theta}_{2})\not\equiv
(\mathfrak{g},{\theta}_{1}',{\theta}_{2}')$.
\end{ex}

It can be proved
the uniqueness
of canonical forms
by imposing an additional condition on the definition.
However,
when we observe at least the commutative case,
we do not need to determine a canonical form uniquely.

\subsection{Existence for canonical compact symmetric triads}\label{sec:cst_can_exist}

The purpose of this subsection
is to prove the following.

\begin{thm}\label{thm:cst_exist_can}
Assume that $G$ is simple.
For any compact symmetric triad $(G,\theta_{1},\theta_{2})$,
there exists a canonical compact symmetric triad
$(G,\theta_{1},\theta_{2}')\sim(G,\theta_{1},\theta_{2})$.
\end{thm}

Without loss of generalities we may assume that
$(G,\theta_{1},\theta_{2})$ is quasi-canonical by Proposition \ref{pro:exist_qcan}.
Let $\mathfrak{t}$
be a maximal abelian subalgebra of $\mathfrak{g}$
which is quasi-canonical with respect to $(G,\theta_{1},\theta_{2})$.
Denote by $(\Delta,\sigma_{1},\sigma_{2})=(\Delta,-d\theta_{1}|_{\mathfrak{t}},-d\theta_{2}|_{\mathfrak{t}})$
the double $\sigma$-system of $(G,\theta_{1},\theta_{2})$.
Let $\Pi$ be a $(\sigma_{1},\sigma_{2})$-fundamental system of $\Delta$.

\begin{lem}\label{lem:order<=>Pi}
Let $n=\mathrm{ord}(\theta_{1}\theta_{2}|_{\mathfrak{t}})\in\mathbb{N}$.
Then, we have$:$
\begin{enumerate}
\item For any $\beta\in\Pi_{1,0}\cup\Pi_{2,0}$,
we have $(\theta_{1}\theta_{2})^{n}=1$ on the root space $\mathfrak{g}(\mathfrak{t},\beta)$.
\item The order of the automorphism $\theta_{1}\theta_{2}$ on $\mathfrak{g}$
satisfies $(\theta_{1}\theta_{2})^{n}=1$
if and only if
$(\theta_{1}\theta_{2})^{n}=1$ holds on $\mathfrak{g}(\mathfrak{t},\alpha)$
for all $\alpha\in\Pi-(\Pi_{1,0}\cup\Pi_{2,0})$.
\end{enumerate}
\end{lem}

\begin{proof}
(1) 
Let $\beta$ be in $\Pi_{1,0}\cup\Pi_{2,0}$.
It can be verified that
$(\theta_{1}\theta_{2})^{n}=1$ holds on $\mathfrak{g}(\mathfrak{t},\beta)$
by a case-by-case verification.
Let us consider the case when $n$ is odd and $\beta\in\Pi_{1,0}$.
It is sufficient to show $(\theta_{2}\theta_{1})^{n}=1$ on $\mathfrak{g}(\mathfrak{t},\beta)$.
If we write $n=2l+1$,
then $(\theta_{1}\theta_{2})^{n}(\beta)=\beta$
yields
$\theta_{2}((\theta_{1}\theta_{2})^{l}(\beta))=(\theta_{1}\theta_{2})^{l}(\beta)$.
Hence we get
$\mathfrak{g}(\mathfrak{t},\beta)\subset \mathfrak{k}_{1}^{\mathbb{C}}$
and $\mathfrak{g}(\mathfrak{t},(\theta_{1}\theta_{2})^{l}(\beta))\subset\mathfrak{k}_{2}^{\mathbb{C}}$
by Lemma \ref{lem:Klein_pro4.1}, (2).
For any $X\in\mathfrak{g}(\mathfrak{t},\beta)$,
we obtain
\[
(\theta_{2}\theta_{1})^{n}(X)
=(\theta_{2}\theta_{1})^{2l}\theta_{2}X
=(\theta_{2}\theta_{1})^{l}\theta_{2}((\theta_{1}\theta_{2})^{l}(X))
=X.
\]
For the other cases, a similar argument shows
that $(\theta_{1}\theta_{2})^{n}=1$ on $\mathfrak{g}(\mathfrak{t},\beta)$
for each $\beta\in\Pi_{1,0}\cup\Pi_{2,0}$.

(2) 
The necessity is clear.
We will only prove the sufficiency.
From (1),
we have
$(\theta_{1}\theta_{2})^{n}=1$ on $\sum_{\alpha\in\Pi}\mathfrak{g}(\mathfrak{t},\alpha)$.
This yields $(\theta_{1}\theta_{2})^{n}=1$ on $\mathfrak{g}^{\mathbb{C}}$,
equivalently, $(\theta_{1}\theta_{2})^{n}=1$ on $\mathfrak{g}$.
Thus we have the assertion.
\end{proof}

In order to prove Theorem \ref{thm:cst_exist_can}
we need to give a refinement
of Lemma \ref{lem:order<=>Pi}, (2)
(see Lemma \ref{lem:order<=>Pi_refine}).
Let $(G,\theta_{1},\theta_{2})$ be a quasi-canonical compact symmetric triad
and $\mathfrak{t}$ be a quasi-canonical maximal abelian subalgebra of $\mathfrak{g}$ with respect to $(G,\theta_{1},\theta_{2})$.
Denote by $(S_{1}(\Pi,\Pi_{1,0},p_{1}),S_{2}(\Pi,\Pi_{2,0},p_{2}))$
the corresponding double Satake diagram
of $(G,\theta_{1},\theta_{2})$.
We put $n=\mathrm{ord}(\theta_{1}\theta_{2}|_{\mathfrak{t}})$.
Let $pr:\mathfrak{t}\to\mathfrak{t}\cap(\mathfrak{m}_{1}\cap\mathfrak{m}_{2})$
be the orthogonal projection.
Then,
for any $\alpha\in\mathfrak{t}$,
we have
\[
pr(\alpha)=\dfrac{1}{n}\sum_{k=0}^{n-1}\{(\theta_{1}\theta_{2})^{k}(\alpha)-\theta_{1}(\theta_{1}\theta_{2})^{k}(\alpha)\}
=\dfrac{1}{n}\sum_{k=0}^{n-1}\{(\theta_{1}\theta_{2})^{k}(\alpha)-\theta_{2}(\theta_{1}\theta_{2})^{k}(\alpha)\},
\]
which is also expressed as
\begin{equation}\label{eqn:pr_expression2}
pr(\alpha)=\dfrac{1}{n}\sum_{k=0}^{n-1}(\theta_{1}\theta_{2})^{k}(H-\theta_{i}(H)).
\end{equation}

The following lemma is fundamental in our argument.

\begin{lem}\label{lem:pr_S1S2_fr}
Under the above settings,
we have:
\begin{enumerate}
\item Fix $i\in\{1,2\}$.
For any $\alpha,\beta\in\Pi-\Pi_{i,0}$,
$\alpha=p_{i}(\beta)$ yields $pr(\alpha)=pr(\beta)$.
\item $pr(\alpha)=0$ holds for all $\alpha\in \Pi_{1,0}\cup\Pi_{2,0}$
\end{enumerate}
\end{lem}

This lemma can be shown by means of the expression \eqref{eqn:pr_expression2}.
We omit the detail.

We set $\Pi_{0}=\{\alpha\in \Pi\mid pr(\alpha)=0\}$.
By Lemma \ref{lem:pr_S1S2_fr}, (2),
we have $\Pi_{1,0}\cup\Pi_{2,0}\subset\Pi_{0}$.
In fact, it is verified that
$\Pi_{0}$
is expressed as
\[
\Pi_{0}=\Pi_{1,0}\cup\Pi_{2,0}\cup
\{\alpha\in\Pi-(\Pi_{1,0}\cup\Pi_{2,0})\mid
p_{1}(\alpha)\in \Pi_{2,0} \text{ or }
p_{2}(\alpha)\in \Pi_{1,0}
\}.
\]
In the case when $(G,\theta_{1},\theta_{2})$
is commutative,
it follows from
Lemma \ref{lem:a=0a_i=0}, (2)
that $\Pi_{0}=\Pi_{1,0}\cup\Pi_{2,0}$ holds.

Let $\Pi^{*}$
be a subset of $\Pi-\Pi_{0}$ satisfying
\[
\{\alpha,p_{1}(\alpha),p_{2}(\alpha)\mid \alpha\in \Pi^{*}\}=\Pi-\Pi_{0}
\]
with minimum cardinality among
all such subsets.
We call such $\Pi^{*}$ a \textit{core} of $\Pi-\Pi_{0}$.
Clearly,
if $\Pi=\Pi_{0}$, then 
we have $\Pi^{*}=\emptyset$.
By means of the Satake involutions $p_{1}, p_{2}$,
we can reconstruct
$\Pi_{0}$
and 
$\Pi-\Pi_{0}$
from $\Pi_{1,0}\cup\Pi_{2,0}$
and $\Pi^{*}$, respectively,
Then, $\Pi$
is obtained from $\Pi^{*}\cup\Pi_{1,0}\cup\Pi_{2,0}$
and so is $\Delta^{+}$.

By Lemma \ref{lem:a=0a_i=0}, we get
$\mathfrak{t}\cap(\mathfrak{m}_{1}\cap\mathfrak{m}_{2})
=pr(\mathrm{span}_{\mathbb{R}}\Pi)
=\mathrm{span}_{\mathbb{R}}\{pr(\alpha)\mid \alpha\in\Pi-\Pi_{0}\}$.
Furthermore,
we have the following proposition.

\begin{pro}\label{pro:cardPi=rank}
Assume that $G$ is simple.
Then, there exists a core $\Pi^{*}\subset \Pi-\Pi_{0}$
satisfying the following conditions$:$
\begin{enumerate}
\item $\{pr(\alpha)\mid \alpha\in\Pi^{*}\}$
are linearly independent.
\item $\mathfrak{t}\cap(\mathfrak{m}_{1}\cap\mathfrak{m}_{2})
=\mathrm{span}_{\mathbb{R}}\{pr(\alpha)\mid \alpha\in\Pi^{*}\}$.
\end{enumerate}
In particular, 
the cardinality of $\Pi^{*}$
is equal to the dimension of $\mathfrak{t}\cap(\mathfrak{m}_{1}\cap\mathfrak{m}_{2})$.
\end{pro}

The proof is given by a case-by-case verification
based on the classification.

\begin{ex}
Let us consider  the case when $(\mathfrak{g},\mathfrak{k}_{1},\mathfrak{k}_{2})=
(\mathfrak{so}(8),\mathfrak{so}(3)\oplus\mathfrak{so}(5),\tilde{\kappa}(\mathfrak{so}(3)\oplus\mathfrak{so}(5)))$.
Its double Satake diagram
is given by
$(S(\Pi,\psi(\Pi_{0}^{(3)}),p^{(3)}),S(\Pi,\kappa(\Pi_{0}^{(3)}),\kappa\cdot p^{(3)}))$
as in Table \ref{table:Satake_DI}.
Then $\Pi^{*}=\{\alpha_{2},\alpha_{3}\}$
gives a core of $\Pi-\Pi_{0}$.
Since we have
\[
pr(\alpha_{2})=\alpha_{2},\quad
pr(\alpha_{3})=\dfrac{1}{3}(\alpha_{1}+\alpha_{3}+\alpha_{4}),
\]
$\{pr(\alpha_{2}),pr(\alpha_{3})\}$
are linearly independent.
From $\dim(\mathfrak{t}\cap(\mathfrak{m}_{1}\cap\mathfrak{m}_{2}))=2$,
we have $\mathfrak{t}\cap(\mathfrak{m}_{1}\cap\mathfrak{m}_{2})=\mathrm{span}_{\mathbb{R}}
\{pr(\alpha_{2}),pr(\alpha_{3})\}$.
\end{ex}

In a similar manner,
we can prove Proposition \ref{pro:cardPi=rank}
for the other cases.
We omit the details.

The following is a refinement of Lemma
\ref{lem:order<=>Pi}, (2).

\begin{lem}\label{lem:order<=>Pi_refine}
Let $n=\mathrm{ord}(\theta_{1}\theta_{2}|_{\mathfrak{t}})\in\mathbb{N}$
and $\Pi^{*}\subset \Pi-\Pi_{0}$ be a core.
Then the order of the automorphism $\theta_{1}\theta_{2}$ on $\mathfrak{g}$
satisfies $(\theta_{1}\theta_{2})^{n}=1$
if and only if
$(\theta_{1}\theta_{2})^{n}=1$ holds on $\mathfrak{g}(\mathfrak{t},\alpha)$
for all $\alpha\in\Pi^{*}$.
\end{lem}

\begin{proof}
We will only prove the sufficiency.
We define a subspace $\mathfrak{h}$ of $\mathfrak{g}^{\mathbb{C}}$
as follows:
\[
\mathfrak{h}
=\mathfrak{t}^{\mathbb{C}}
\oplus\sum_{\beta\in\Pi_{1,0}\cup\Pi_{2,0}}
\left(\mathfrak{g}(\mathfrak{t},\beta)\oplus\mathfrak{g}(\mathfrak{t},-\beta)\right)
\oplus\sum_{\alpha\in\Pi^{*}}
\left(\mathfrak{g}(\mathfrak{t},\alpha)\oplus\mathfrak{g}(\mathfrak{t},-\alpha)\right).
\]
By the definition, we have $(\theta_{1}\theta_{2})^{n}=1$ on $\mathfrak{h}$.
This implies that
$(\theta_{1}\theta_{2})^{n}=1$ holds on $\mathfrak{h}+\theta_{1}(\mathfrak{h})+\theta_{2}(\mathfrak{h})$.
Since
$\mathfrak{h}+\theta_{1}(\mathfrak{h})+\theta_{2}(\mathfrak{h})$
generates $\mathfrak{g}^{\mathbb{C}}$,
we get $(\theta_{1}\theta_{2})^{n}=1$ on $\mathfrak{g}^{\mathbb{C}}$.
Thus we have the assertion.
\end{proof}

We are ready to prove Theorem \ref{thm:cst_exist_can}.

\begin{proof}[Proof of Theorem $\ref{thm:cst_exist_can}$]
Without loss of generalities we may assume that $(G,\theta_{1},\theta_{2})$ is quasi-canonical.
Let $\mathfrak{t}$ be a maximal abelian subalgebra of $\mathfrak{g}$
which is quasi-canonical with respect to $(G,\theta_{1},\theta_{2})$.
Let $n=\mathrm{ord}(\theta_{1}\theta_{2}|_{\mathfrak{t}})$
and $\Pi^{*}\subset \Pi-\Pi_{0}$ be a core
as in Proposition \ref{pro:cardPi=rank}.

First, 
we show $\mathrm{ord}(\theta_{1}\theta_{2})=n$ in the case when $\Pi^{*}=\emptyset$.
Indeed,
we get $\mathrm{ord}(\theta_{1}\theta_{2})\geq\mathrm{ord}(\theta_{1}\theta_{2}|_{\mathfrak{t}})=n$.
In addition, Lemma \ref{lem:order<=>Pi}, (1),
we have $(\theta_{1}\theta_{2})^{n}=1$.
Hence, we obtain $\mathrm{ord}(\theta_{1}\theta_{2})=n$,
so that $(G,\theta_{1},\theta_{2})$ is canonical.

Secondly,
we consider the case when $\Pi^{*}\neq\emptyset$.
For any $H\in\mathfrak{t}\cap(\mathfrak{m}_{1}\cap\mathfrak{m}_{2})$,
if we put $g=\exp(H)$,
then $\mathrm{Ad}(g)$
gives the identity transformation on $\mathfrak{t}$.
Hence $\mathfrak{t}$
is also quasi-canonical with respect to $(G,\theta_{1},\tau_{g}\theta_{2}\tau_{g}^{-1})=:(G,\theta_{1},\theta_{2}')$.
Then it is sufficient to show that there exists $H\in\mathfrak{t}\cap(\mathfrak{m}_{1}\cap\mathfrak{m}_{2})$
such that $(\theta_{1}\theta_{2}')^{n}=1$ holds.

Let $H$ be any element in $\mathfrak{t}\cap(\mathfrak{m}_{1}\cap\mathfrak{m}_{2})$.
Let $\{X_{\alpha}\}_{\alpha\in\Delta}$ be a Chevalley basis of $\mathfrak{g}^{\mathbb{C}}$
with $\overline{X_{\alpha}}=-X_{-\alpha}$ (cf.~Lemma \ref{lem:Cbasis_conj}).
For each $\alpha\in\Delta$, we define a complex numbers
$S_{\alpha}$ by
$(\theta_{1}\theta_{2})^{n}X_{\alpha}=S_{\alpha}X_{\alpha}$.
By the definition, we have
$(\theta_{1}\theta_{2}')^{n}X_{\alpha}=
e^{\sqrt{-1}\INN{2n\,\cdot\, pr(\alpha)}{H}}S_{\alpha}X_{\alpha}$
for each
$\alpha\in\Delta$.
Then,
it follows from Lemma \ref{lem:order<=>Pi_refine}
that $(\theta_{1}\theta_{2}')^{n}=1$ holds if and only if
$e^{\sqrt{-1}\INN{2n\,\cdot\, pr^{(n)}(\alpha)}{H}}S_{\alpha}=1$
for all $\alpha\in\Pi^{*}$.
From Lemma \ref{lem:Klein_pro4.1}
it is shown that $|S_{\alpha}|=1$ holds,
so that there exists $u_{\alpha}\in\mathbb{R}$
such that $S_{\alpha}=e^{\sqrt{-1}u_{\alpha}}$.
It follows from Proposition \ref{pro:cardPi=rank}, (1)
that the square matrix $(\INN{pr(\alpha)}{pr(\beta)})_{\alpha,\beta\in\Pi^{*}}$
is invertible, so that
the following equation has a solution $H$:
\[
\INN{2n\,\cdot\, pr(\alpha)}{H}+u_{\alpha}=0
\quad
(\alpha\in\Pi^{*}).
\]
Then $(\theta_{1}\theta_{2}')^{n}=1$ holds
for the solution $H$.

From the above argument, we have complete the proof.
\end{proof}

\subsection{Properties of canonical compact symmetric triads}\label{sec:rankorder}

The purpose of this subsection
is to prove the following.

\begin{thm}\label{thm:cst_RO_can}
Assume that $G$ is simple.
Let $(G,\theta_{1},\theta_{2})$
be a canonical compact symmetric triad.
Then, the followings hold$:$
\begin{enumerate}
\item Let $\mathfrak{t}$ be a maximal abelian subalgebra of $\mathfrak{g}$
which is canonical with respect to $(G,\theta_{1},\theta_{2})$.
Then, $\mathfrak{t}\cap(\mathfrak{m}_{1}\cap\mathfrak{m}_{2})$
is a maximal abelian subspace of $\mathfrak{m}_{1}\cap\mathfrak{m}_{2}$.
\item $\mathrm{ord}[(G,\theta_{1},\theta_{2})]=\mathrm{ord}(\theta_{1}\theta_{2})$.
\end{enumerate}
\end{thm}

\begin{rem}\label{rem:commutable}
Let $(G,\theta_{1},\theta_{2})$ be a canonical compact symmetric triad
and $\mathfrak{t}$ be a canonical maximal abelian subalgebra of $\mathfrak{g}$
with respect to $(G,\theta_{1},\theta_{2})$.
Then, if
$\theta_{1}$ and $\theta_{2}$ is commutative on $\mathfrak{t}$,
then $[(G,\theta_{1},\theta_{2})]$ is commutable.
In addition, Theorem \ref{thm:cst_RO_can}, (2)
implies that the converse is also true
in the case when $G$ is simple.
Thus, we have the complete classification of commutable compact symmetric triads
$[(G,\theta_{1},\theta_{2})]$
by means of Table \ref{table:rank_ord}.
We note that,
if the simple Lie group $G$ is of exceptional type,
then $[(G,\theta_{1},\theta_{2})]$ is commutable.
\end{rem}

\subsubsection{Proof of Theorem $\ref{thm:cst_RO_can}$, $(1)$}

We give two 
sufficient conditions
for $\mathfrak{t}^{\sigma_{1}}\cap\mathfrak{t}^{\sigma_{2}}$
to be a maximal abelian subspace of $\mathfrak{m}_{1}\cap\mathfrak{m}_{2}$.
One is that $(G,\theta_{1},\theta_{2})$ is commutative
as shown in Lemma \ref{lem:comm_quasican}.
The other is the following lemma.

\begin{lem}\label{lem:rank_ineq}
Let $(G,\theta_{1},\theta_{2})$ be a compact symmetric triad
and $\mathfrak{t}$ be a maximal abelian subalgebra of $\mathfrak{g}$
such that $\mathfrak{t}\cap\mathfrak{m}_{i}$
$(i=1,2)$ is a maximal abelian subspace
of $\mathfrak{m}_{i}$.
We denote by $(\Delta,\sigma_{1},\sigma_{2})$
the double $\sigma$-system of $(G,\theta_{1},\theta_{2})$
with respect to $\mathfrak{t}$.
Then the following holds:
\begin{equation}\label{eqn:rank_ineq}
\dim(\mathfrak{t}^{\sigma_{1}}\cap\mathfrak{t}^{\sigma_{2}})\leq
\min\{\mathrm{rank}(G,\theta_{i})\mid i=1,2\}.
\end{equation}
Furthermore,
if the equality in this inequality
holds, then $\mathfrak{t}^{\sigma_{1}}\cap\mathfrak{t}^{\sigma_{2}}=\mathfrak{t}\cap(\mathfrak{m}_{1}\cap\mathfrak{m}_{2})$
becomes a maximal abelian subspace of $\mathfrak{m}_{1}\cap\mathfrak{m}_{2}$.
\end{lem}

\begin{proof}
By the definition we have the inequality \eqref{eqn:rank_ineq}.
Assume that the equality in \eqref{eqn:rank_ineq}
holds.
Let $\mathfrak{a}$ be a maximal abelian subspace of $\mathfrak{m}_{1}\cap\mathfrak{m}_{2}$
containing $\mathfrak{t}^{\sigma_{1}}\cap\mathfrak{t}^{\sigma_{2}}$.
From $\dim(\mathfrak{a})\leq \mathrm{rank}(G,\theta_{i})$,
we have
$\dim(\mathfrak{t}^{\sigma_{1}}\cap\mathfrak{t}^{\sigma_{2}})\leq
\dim(\mathfrak{a})\leq \min\{\mathrm{rank}(G,\theta_{i})\mid i=1,2\}$.
By the assumption we obtain $\mathfrak{t}^{\sigma_{1}}\cap\mathfrak{t}^{\sigma_{2}}=\mathfrak{a}$.
Thus, we have the assertion.
\end{proof}

There are some non-commutative, canonical compact symmetric triads
such that the equality in \eqref{eqn:rank_ineq}
does not hold.
In the case when $G$ is simple,
we can classify such compact symmetric triads at the Lie algebra level
by means of Table \ref{table:rank_ord},
which are listed in Table \ref{table:r<rr}.

\begin{table}[H]
\centering
\renewcommand\arraystretch{1.3}
\caption{
Canonical compact symmetric triads $(G,\theta_{1},\theta_{2})$ satisfying
$\mathrm{ord}(\theta_{1}\theta_{2})\geq 3$ and
$\dim(\mathfrak{t}^{\sigma_{1}}\cap\mathfrak{t}^{\sigma_{2}})<
\min\{\mathrm{rank}(G,\theta_{i})\mid i=1,2\}$
}\label{table:r<rr}
\begin{tabular}{cc}
\hline
\hline
$(\mathfrak{g},\mathfrak{k}_{1},\mathfrak{k}_{2})$ & Remark \\
\hline
\hline
$(\mathfrak{su}(2a+2b+2),\mathfrak{sp}(a+b+1),
\mathfrak{s}(\mathfrak{u}(2a+1)\oplus\mathfrak{u}(2b+1)))$ &
$0\leq a<b$ \\
$(\mathfrak{so}(2a+2b+2),\mathfrak{so}(2a+1)\oplus\mathfrak{so}(2b+1),
\mathfrak{u}(a+b+1))$ & $0\leq a <b$ \\
$(\mathfrak{so}(8),\mathfrak{so}(1)\oplus\mathfrak{so}(7),\tilde{\kappa}(\mathfrak{so}(c)\oplus\mathfrak{so}(8-c)))$ &
$c=1,2,3$ \\
$(\mathfrak{so}(8),\mathfrak{so}(2)\oplus\mathfrak{so}(6),\tilde{\kappa}(\mathfrak{so}(c)\oplus\mathfrak{so}(8-c)))$ &
$c=2,3$ \\
$(\mathfrak{so}(8),\mathfrak{so}(3)\oplus\mathfrak{so}(5),\tilde{\kappa}(\mathfrak{so}(3)\oplus\mathfrak{so}(5)))$ \\

\hline
\hline
\end{tabular}
\end{table}

For these canonical compact symmetric triads,
we will prove 
Theorem \ref{thm:cst_RO_can}, (1)
by a case-by-case verification.

\begin{ex}\label{ex:so3535_maximal}
Let us consider the case when
$(\mathfrak{g},\mathfrak{k}_{1},\mathfrak{k}_{2})=(\mathfrak{so}(8),\mathfrak{so}(3)\oplus\mathfrak{so}(5),\kappa(\mathfrak{so}(3)\oplus\mathfrak{so}(5)))$.
Then we have
\[
2=\dim(\mathfrak{t}^{\sigma_{1}}\cap\mathfrak{t}^{\sigma_{2}})
\leq \max\{\dim(\mathfrak{t}^{\sigma_{1}}\cap s\mathfrak{t}^{\sigma_{2}})\mid s\in W(\Delta)\}\leq 3.
\]
We will show that
$\max\{\dim(\mathfrak{t}^{\sigma_{1}}\cap s\mathfrak{t}^{\sigma_{2}})\mid s\in W(\Delta)\}=2$.
Under Notation \ref{nota:Dr}
we have
\[\mathfrak{t}^{\sigma_{1}}=
\mathbb{R}e_{3}
\oplus\mathbb{R}(e_{1}-e_{2})
\oplus\mathbb{R}(e_{2}-e_{3}),~
\mathfrak{t}^{\sigma_{2}}=\mathbb{R}(
e_{1}-e_{2}+e_{3}-e_{4})\oplus\mathbb{R}(e_{2}-e_{3})\oplus\mathbb{R}(e_{3}+e_{4}).
\]
Suppose
for contradiction that
there exists $s\in W(\Delta)$
satisfying $\dim(\mathfrak{t}^{\sigma_{1}}\cap s\mathfrak{t}^{\sigma_{2}})=3$.
Then we have $s^{-1}\mathfrak{t}^{\sigma_{1}}=\mathfrak{t}^{\sigma_{2}}$.
It follows from the expression of $\mathfrak{t}^{\sigma_{1}}$
that
there exists $j\in\{1,2,3,4\}$
satisfying $e_{j}\in w^{-1}\mathfrak{t}^{\sigma_{1}}$.
This contradicts that
$\mathfrak{t}^{\sigma_{2}}$ does not contain
all the vectors $e_{1},e_{2},e_{3},e_{4}$.
In addition, by Proposition \ref{pro:rank=maxdim}
we obtain
$\dim(\mathfrak{t}^{\sigma_{1}}\cap\mathfrak{t}^{\sigma_{2}})
= \max\{\dim(\mathfrak{t}^{\sigma_{1}}\cap s\mathfrak{t}^{\sigma_{2}})\mid s\in W(\Delta)\}
=\mathrm{rank}(G,\theta_{1},\theta_{2})$.
\end{ex}

In a similar manner as in Example \ref{ex:so3535_maximal},
we have
Theorem \ref{thm:cst_RO_can}, (1)
for
$(\mathfrak{g},\mathfrak{k}_{1},\mathfrak{k}_{2})=(\mathfrak{so}(8),\mathfrak{so}(a)\oplus\mathfrak{so}(8-a),\tilde{\kappa}(\mathfrak{so}(c)\oplus\mathfrak{so}(8-c)))$
with $(a,c)=(1,\{1, 2,3\})$, $(2,\{2,3\})$.

\begin{ex}
Let us consider the case when
$(\mathfrak{g},\mathfrak{k}_{1},\mathfrak{k}_{2})=
(\mathfrak{so}(2a+2b+2),\mathfrak{so}(2a+1)\oplus\mathfrak{so}(2b+1),
\mathfrak{u}(a+b+1))$ with $0\leq a<b$.
We will show the following relation:
\[
\mathrm{rank}(\mathfrak{g},\mathfrak{k}_{1},\mathfrak{k}_{2})=a(=\dim(\mathfrak{t}\cap(\mathfrak{m}_{1}\cap\mathfrak{m}_{2}))).
\]
We define two involutions $\theta_{1}',\theta_{2}'$ on $\mathfrak{so}(2a+2b+2)$ as follows:
\[
\theta'_{1}(X)=I_{2a+1,2b+1}XI_{2a+1,2b+1},\quad
\theta'_{2}(X)=J_{a+b+1}XJ_{a+b+1}^{-1}.
\]
Then we have $\mathfrak{g}^{\theta_{1}'}=\mathfrak{so}(2a+1)\oplus\mathfrak{so}(2b+1)$
and $\mathfrak{g}^{\theta_{2}'}=\mathfrak{u}(a+b+1)$.
By the classification,
$(\mathfrak{g},\theta_{1}',\theta_{2}')$
is in $[(\mathfrak{g},\mathfrak{k}_{1},\mathfrak{k}_{2})]$.
From 
\[
{\mathfrak m}_{1}'\cap {\mathfrak m}_{2}'
=\mathfrak{g}^{-\theta_{1}'}\cap\mathfrak{g}^{-\theta_{2}'}
=\left\{\left.
\begin{pmatrix}
O_{2a+1} & O     & W     & O     \\
O     & O_{b-a} &  O    & O   \\
W     & O     & O_{2a+1} & O     \\
O     & O  &  O    & O_{b-a} 
\end{pmatrix}\right|
W\in \mathfrak{so}(2a+1)
\right\},
\]
we obtain
\[
[\mathfrak{m}_{1}'\cap {\mathfrak m}_{2}',\mathfrak{m}_{1}'\cap {\mathfrak m}_{2}']=
\left\{\left.
\begin{pmatrix}
X &      &      &      \\
     & O &      &    \\
     &      & X &      \\
     &   &      & O 
\end{pmatrix}\right|
X\in \mathfrak{so}(2a+1)
\right\}= \mathfrak{so}(2a+1)
\]
For any $Z\in
[\mathfrak{m}_{1}'\cap {\mathfrak m}_{2}',\mathfrak{m}_{1}'\cap {\mathfrak m}_{2}']$
and 
$Y\in \mathfrak{m}_{1}'\cap {\mathfrak m}_{2}'$ with
\[
Z=\begin{pmatrix}
X &      &      &      \\
     & O &      &    \\
     &      & X &      \\
     &   &      & O 
\end{pmatrix},\quad
Y=\begin{pmatrix}
O_{2a+1} & O        & W         & O     \\
O          & O_{b-a} &  O        & O   \\
W          & O        & O_{2a+1} & O     \\
O          & O        &  O         & O_{b-a} 
\end{pmatrix},
\]
we have
\[
[Z,Y]=\begin{pmatrix}
O_{2a+1} & O     & [X,W]     & O     \\
O     & O_{b-a} &  O    & O   \\
[X,W]     & O     & O_{2a+1} & O     \\
O     & O  &  O    & O_{b-a} 
\end{pmatrix}.
\]
This yields
\begin{align*}
&([\mathfrak{m}_{1}'\cap {\mathfrak m}_{2}',\mathfrak{m}_{1}'\cap {\mathfrak m}_{2}']
\oplus \mathfrak{m}_{1}'\cap {\mathfrak m}_{2}',[\mathfrak{m}_{1}'\cap {\mathfrak m}_{2}',\mathfrak{m}_{1}'\cap {\mathfrak m}_{2}'])\\
&\phantom{hogehogehogehoge}\simeq (\mathfrak{so}(2a+1)\oplus \mathfrak{so}(2a+1),
\Delta (\mathfrak{so}(2a+1)\oplus \mathfrak{so}(2a+1))).
\end{align*}
Hence we get
\[
\mathrm{rank}(\mathfrak{g},\mathfrak{k}_{1},\mathfrak{k}_{2})
=\mathrm{rank}(\mathfrak{g},\theta_{1}',\theta_{2}')
=\mathrm{rank}(\mathfrak{so}(2a+1))=a,\]
so that $\mathfrak{t}\cap(\mathfrak{m}_{1}\cap\mathfrak{m}_{2})$
is a maximal abelian subspace of $\mathfrak{m}_{1}\cap\mathfrak{m}_{2}$.
\end{ex}

\begin{ex}
Let us consider the case when
$(\mathfrak{g},\mathfrak{k}_{1},\mathfrak{k}_{2})=(\mathfrak{su}(2a+2b+2),\mathfrak{sp}(a+b+1),
\mathfrak{s}(\mathfrak{u}(2a+1)\oplus\mathfrak{u}(2b+1)))$ with $0\leq a<b$.
It is sufficient to show $\mathrm{rank}(\mathfrak{g},\mathfrak{k}_{1},\mathfrak{k}_{2})=a$.
We define two involutions $\theta_{1}',\theta_{2}'$ on $\mathfrak{su}(2a+2b+2)$ as follows:
\[
\theta'_{1}(X)=J_{a+b+1}\bar{X}J_{a+b+1}^{-1},\quad \theta'_{2}(X)= I_{2a+1,2b+1}XI_{2a+1,2b+1}.
\]
Then we have $\mathfrak{g}^{\theta_{1}'}=\mathfrak{sp}(a+b+1)$
and $\mathfrak{g}^{\theta_{2}'}=\mathfrak{s}(\mathfrak{u}(2a+1)\oplus\mathfrak{u}(2b+1))$,
from which
$(\mathfrak{g},\theta_{1}',\theta_{2}')$
is in $[(\mathfrak{g},\mathfrak{k}_{1},\mathfrak{k}_{2})]$.
A direct calculation shows
\[
{\mathfrak m}_{1}'\cap {\mathfrak m}_{2}'
=\mathfrak{g}^{-\theta_{1}'}\cap\mathfrak{g}^{-\theta_{2}'}
=\left\{\left.
\begin{pmatrix}
O_{2a+1} & O     & Y     & O     \\
O     & O_{b-a} &  O    & O   \\
\bar{Y}     & O     & O_{2a+1} & O     \\
O     & O  &  O    & O_{b-a} 
\end{pmatrix}\right|
Y=-{}^tY \in M(2a+1,\mathbb{C}) 
\right\}.\]
In the case when $a=0$,
we have $\mathfrak{m}_{1}\cap\mathfrak{m}_{2}=\{0\}$.
This implies that $\mathrm{rank}(\mathfrak{g},\mathfrak{k}_{1},\mathfrak{k}_{2})=
\mathrm{rank}(\mathfrak{g},\theta_{1}',\theta_{2}')=0=a$.
In what follows, we assume that $a\geq 1$.
Since $\mathfrak{u}(2a+1)$ can be expressed by
\[
\mathfrak{u}(2a+1)=\mathrm{span}_{\mathbb{R}}\{X\bar{Y}-Y\bar{X}\mid X,Y\in\mathfrak{gl}(n,\mathbb{C}),
{}^tX=-X,{}^tY=-Y\},
\]
we have
\[
[\mathfrak{m}_{1}'\cap {\mathfrak m}_{2}',\mathfrak{m}_{1}'\cap {\mathfrak m}_{2}']=
\left\{\left.
\begin{pmatrix}
X &      &      &      \\
     & O &      &    \\
     &      & \bar{X} &      \\
     &   &      & O 
\end{pmatrix}\right|
X\in \mathfrak{u}(2a+1)
\right\}=\mathfrak{u}(2a+1).
\]
Hence we get
\[
([\mathfrak{m}_{1}'\cap {\mathfrak m}_{2}',\mathfrak{m}_{1}'\cap {\mathfrak m}_{2}']
\oplus \mathfrak{m}_{1}'\cap {\mathfrak m}_{2}',[\mathfrak{m}_{1}'\cap {\mathfrak m}_{2}',\mathfrak{m}_{1}'\cap {\mathfrak m}_{2}'])
\simeq (\mathfrak{so}(4a+2),\mathfrak{u}(2a+1)),
\]
from which 
$\mathrm{rank}(\mathfrak{g},\mathfrak{k}_{1},\mathfrak{k}_{2})
=\mathrm{rank}(\mathfrak{so}(4a+2),\mathfrak{u}(2a+1))=a$.
\end{ex}

From the above argument we have complete the proof
of Theorem \ref{thm:cst_RO_can}, (1).

\subsubsection{Proof of Theorem $\ref{thm:cst_RO_can}$, $(2)$}

In the rest of this paper,
we give the proof of Theorem \ref{thm:cst_RO_can}, (2).
In the case when $\mathrm{rank}(G,\theta_{1},\theta_{2})=0$,
we have 
$\mathrm{ord}[(G,\theta_{1},\theta_{2})]=\mathrm{ord}(\theta_{1}\theta_{2})$
by Proposition \ref{pro:rank0_ordconst}.
In what follows,
we will prove $\mathrm{ord}[(G,\theta_{1},\theta_{2})]=\mathrm{ord}(\theta_{1}\theta_{2})$
in the case when $\mathrm{rank}(G,\theta_{1},\theta_{2})\geq 1$.
Our proof is based on a case-by-case verification
for the order of a canonical one $(G,\theta_{1},\theta_{2})$.
We note that $n:=\mathrm{ord}(\theta_{1}\theta_{2})\in\{1,2,3,4,6\}$ holds
by the classification.

In the case when $n=1$,
we have $1\leq \mathrm{ord}[(G,\theta_{1},\theta_{2})]\leq \mathrm{ord}(\theta_{1}\theta_{2})=1$.
This yields
$\mathrm{ord}[(G,\theta_{1},\theta_{2})]=\mathrm{ord}(\theta_{1}\theta_{2})$.

Next, we consider the case when $n=2$.
Suppose for a contradiction that
there exists $(G,\theta_{1}',\theta_{2}')\sim(G,\theta_{1},\theta_{2})$
such that $\mathrm{ord}(\theta_{1}'\theta_{2}')=1$.
Then we have $\theta_{1}'=\theta_{2}'$.
As explained in Example \ref{ex:order1_can},
$(G,\theta_{1}',\theta_{1}')$ is canonical.
If we denote by $(\Delta',\sigma_{1}',\sigma_{1}')=(\Delta',-d\theta_{1}'|_{\mathfrak{t}'},-d\theta_{1}'|_{\mathfrak{t}'})$
the double $\sigma$-system of $(G,\theta_{1}',\theta_{1}')$,
then
we have $(\Delta,\sigma_{1},\sigma_{2})\equiv(\Delta',\sigma_{1}',\sigma_{1}')$
by Proposition \ref{pro:cst_dstake_determ}.
This yields
$1=\mathrm{ord}(d\theta_{1}'d\theta_{1}'|_{\mathfrak{t}'})=\mathrm{ord}(d\theta_{1}d\theta_{2}|_{\mathfrak{t}})=2$,
which is a contradiction.
Thus we have $2\leq \mathrm{ord}[(G,\theta_{1},\theta_{2})]\leq \mathrm{ord}(\theta_{1}\theta_{2})=2$,
that is, $\mathrm{ord}[(G,\theta_{1},\theta_{2})]=\mathrm{ord}(\theta_{1}\theta_{2})$.

Here, we note that $n=6$ implies $\mathrm{rank}(G,\theta_{1},\theta_{2})=0$ by the classification.
Hence the rest of our proof consists of the case when $n\in\{3,4\}$,
so that $(G,\theta_{1},\theta_{2})$ or $(G,\theta_{2},\theta_{1})$
is locally isomorphic to one of the following two cases
among the compact symmetric triads:

\medskip

(Case 1) $(\mathfrak{g},\mathfrak{k}_{1},\mathfrak{k}_{2})=(\mathfrak{so}(8),\mathfrak{so}(a)\oplus\mathfrak{so}(8-a),\tilde{\kappa}(\mathfrak{so}(3)\oplus\mathfrak{so}(5)))$
for $a=2,3$:
First, we consider the case when $a=2$.
Then we have $\mathrm{ord}(\theta_{1}\theta_{2})=4$.
Since $\mathfrak{k}_{1}\not\simeq\mathfrak{k}_{2}$ obeys
$\theta_{1}\not\sim\theta_{2}$,
we have
$\mathrm{ord}[(G,\theta_{1},\theta_{2})]\geq 2$.
Suppose for a contradiction that
there exists $(G,\theta_{1}',\theta_{2}')\sim(G,\theta_{1},\theta_{2})$
such that $\mathrm{ord}(\theta_{1}'\theta_{2}')=2$.
As explained in Example \ref{ex:com_can},
$(G,\theta_{1}',\theta_{2}')$ is canonical.
Let 
$(\Delta',\sigma_{1}',\sigma_{2}')=(\Delta',-d\theta_{1}'|_{\mathfrak{t}'},-d\theta_{2}'|_{\mathfrak{t}'})$
denote the double $\sigma$-system of $(G,\theta_{1}',\theta_{2}')$.
From
$(\Delta,\sigma_{1},\sigma_{2})\equiv(\Delta',\sigma_{1}',\sigma_{2}')$,
we have
$2=\mathrm{ord}(d\theta_{1}'d\theta_{2}'|_{\mathfrak{t}'})
=\mathrm{ord}(d\theta_{1}d\theta_{2}|_{\mathfrak{t}})=4$,
which is contradiction.
Furthermore,
from $(\theta_{1}\theta_{2})^{4}=1$,
it can be verified
that there exist no compact symmetric triads $(G,\theta_{1}',\theta_{2}')\sim(G,\theta_{1},\theta_{2})$
satisfying $\mathrm{ord}(\theta_{1}'\theta_{2}')=3$ by Proposition \ref{pro:ordn_n+1_sim}.
Thus we have $\mathrm{ord}[(G,\theta_{1},\theta_{2})]=\mathrm{ord}(\theta_{1}\theta_{2})$.

Secondly, we consider the case when $a=3$.
Then we have
$\mathrm{rank}(G,\theta_{1},\theta_{2})=2$
and
$\mathrm{ord}(\theta_{1}\theta_{2})=3$.
Suppose for a contradiction that
there exists $(G,\theta_{1}',\theta_{2}')\sim(G,\theta_{1},\theta_{2})$
such that $\mathrm{ord}(\theta_{1}'\theta_{2}')=1$,
that is, $\theta_{1}'=\theta_{2}'$.
Let 
$(\Delta',\sigma_{1}',\sigma_{1}')=(\Delta',-d\theta_{1}'|_{\mathfrak{t}'},-d\theta_{1}'|_{\mathfrak{t}'})$
denote the double $\sigma$-system of $(G,\theta_{1}',\theta_{1}')$.
From
$(\Delta,\sigma_{1},\sigma_{2})\equiv(\Delta',\sigma_{1}',\sigma_{1}')$,
we have
$2
=\dim(\mathfrak{t}^{\sigma_{1}}\cap\mathfrak{t}^{\sigma_{2}})
=\dim(\mathfrak{t}^{\sigma_{1}'}\cap\mathfrak{t}^{\sigma_{1}'})=3$,
which is contradiction.
In a similar argument as in the case when $a=2$,
it can be shown that there exist no compact symmetric triads
$(G,\theta_{1}',\theta_{2}')\sim(G,\theta_{1},\theta_{2})$
satisfying
$\mathrm{ord}(\theta_{1}'\theta_{2}')=2$.
Thus, we have $\mathrm{ord}[(G,\theta_{1},\theta_{2})]=\mathrm{ord}(\theta_{1}\theta_{2})$.

\medskip

(Case 2) $(\mathfrak{g},\mathfrak{k}_{1},\mathfrak{k}_{2})=(\mathfrak{su}(2m),\mathfrak{sp}(m),\mathfrak{s}(\mathfrak{u}(2a+1)\oplus\mathfrak{u}(2b+1)))$ or $(\mathfrak{so}(2m),\mathfrak{so}(2a+1)\oplus\mathfrak{so}(2b+1),\mathfrak{u}(m))$
for $a<b$, $m=a+b+1$:
Then we have $\mathrm{ord}(\theta_{1}\theta_{2})=4$
and $\theta_{1}\not\sim\theta_{2}$.
By a similar argument as in (Case 1), $a=2$,
it can be shown
that $\mathrm{ord}[(G,\theta_{1},\theta_{2})]=\mathrm{ord}(\theta_{1}\theta_{2})$.

\medskip

From the above argument we have complete the proof
of Theorem $\ref{thm:cst_RO_can}, (2)$.


\begin{thebibliography}{99}
\bibitem{Araki} S.~Araki, \textit{On root systems and an infinitesimal classification of irreducible symmetric spaces},
J.~Math.~Osaka City Univ., \textbf{13} (1962), 1--34.
\bibitem{BIS} K.~Baba, O.~Ikawa and A.~Sasaki,
\textit{A duality between non-compact semisimple symmetric pairs
and commutative compact semisimple symmetric triads and its general theory},
Diff.~Geom.~and its Applications~\textbf{76} (2021), 101751.
\bibitem{BIS2} K.~Baba, O.~Ikawa and A.~Sasaki, \textit{An alternative proof for Berger's classification of semisimple pseudo-Riemannian symmetric pairs from the view point of compact symmetric triads}, in preparation.
\bibitem{Bo} N.~Bourbaki, \textit{Groupes et algebres de Lie}, Hermann, Paris, 1978.
\bibitem{GT} O.~Geortsches and G.~Thorbergsson,
\textit{On the geometry of orbits of Hermann actions},
Geom.~Dedicata, \textbf{129} (2007), 101--118.
\bibitem{HPTT} E.~Heintze, R.~S.~Palais, C.~Terng and G.~Thobergsson, Hyperpolar actions on symmetric spaces,
Geometry, topology and physics, Conf.~Proc.~Lecture Notes Geom.~Topology, IV,
Int.~Press, Cambridge, MA, (1995), 214--245.
\bibitem{Helgason} S.~Helgason, \textit{Differential geometry, Lie groups, and symmetric spaces},  Academic Press, 1978.
\bibitem{Hermann} R.~Hermann, \textit{Totally geodesic orbits of groups of isometries}, Nederl.~Akad.~Wetensch.~Proc.~Ser.~A \textbf{65} (1962), 291--298.
\bibitem{Ikawa} O.~Ikawa, \textit{The geometry of symmetric triad and orbit spaces of Hermann actions}, J.~Math.~Soc.~Japan, \textbf{63},
(2011), 79--136.
\bibitem{Ikawa3} O.~Ikawa, \textit{The geometry of orbits of Hermann type actions}, Contemporary Perspectives in Differential Geometry and its Related Fields, (2018), 67--78.
\bibitem{ITT} O.~Ikawa, M.~S.~Tanaka and H.~Tasaki, \textit{The fixed point set of a holomorphic isometry, the intersection of two real forms in a Hermitian symmetric space of compact type and symmetric triads}, Journal of Int.~J.~Math., \textbf{26} (2015).
\bibitem{Klein} S.~Klein, \textit{Reconstructing the geometric structure of a Riemannian symmetric space from its Satake diagram}, Geom.~Dedicata, \textbf{138} (2009), 25--50.
\bibitem{Knapp} A.~W.~Knapp, \textit{Lie groups beyond an introduction second edition}, Birkhauser, 2002.
\bibitem{Kollross} A.~Kollross, \textit{A classification of hyperpolar and cohomogeneity one actions}, Trans.~Amer.~Math.~Soc., \textbf{354}, (2001), 571--612.
\bibitem{Matsuki97} T.~Matsuki, \textit{Double Coset Decompositions of Reductive Lie Groups Arising from Two Involutions}, J.~Algebra, \textbf{197} (1997), 49--91.
\bibitem{Matsuki} T.~Matsuki, \textit{Classification of two involutions on semisimple compact Lie groups and root systems}, J.~Lie Theory, \textbf{12} (2002), 41--68.
\bibitem{O} Y.~Ohnita, \textit{On classification of minimal orbits of the Hermann action satisfying Koike's conditions} (Joint work with Minoru Yoshida),
Proceedings of The 21st International Workshop on Hermitian Symmetric Spaces and Submanifolds, \textbf{21} (2017), 1--15.
\bibitem{Ohno} S.~Ohno, \textit{A sufficient condition for orbits of Hermann actions to be weakly reflective}, Tokyo J.~Math. \textbf{39} (2016), 537--563.
\bibitem{Ohno21} S.~Ohno, \textit{Geometric Properties of Orbits of Hermann actions}, accepted to Tokyo J.~Math.
\bibitem{OSU} S.~Ohno, T.~Sakai and H.~Urakawa,
\textit{Biharmonic homogeneous submanifolds in
compact symmetric spaces and compact Lie groups},
Hiroshima Math.~J., \textbf{49} (2019), 47--115.
\bibitem{OS} T.~Oshima and J.~Sekiguchi, \textit{The Restricted Root System of a Semisimple Symmetric Pair}, Advanced studies in Pure Mathematics \textbf{4} (1984),  433--487.
\bibitem{Satake} I.~Satake, \textit{On representations and compactifications of symmetric Riemannian spaces}, Ann.~of Math., \textbf{71} (1960), 77--110.
\bibitem{Sugiura} M.~Sugiura, \textit{Conjugate classes of Cartan subalgebras in real semisimple Lie algebras}, J.~Math.~Soc.~Japan, \textbf{11},
(1959), 374--434. Correction to my paper: Conjugate classes of Cartan subalgebras in real semisimple Lie algebras, J.~Math.~Soc.~Japan, \textbf{23},
(1971), 379--383.
\bibitem{Warner} G.~Warner, \textit{Harmonic analysis on semi-simple {L}ie groups. {I}},
Springer-Verlag, New York-Heidelberg, 1972.
\end{thebibliography}
\end{document}